\documentclass[10pt]{amsart}

\usepackage[latin2]{inputenc}
\usepackage{amsmath}
\usepackage{graphicx}
\usepackage{amssymb}
\usepackage{esint}
\usepackage{color}
\usepackage{amsthm}
\usepackage{epsfig}
\usepackage{enumitem}
\usepackage{mathtools}
\usepackage{esvect}

\usepackage{longtable}

\usepackage[dvipsnames]{xcolor}
\usepackage{tikz}
\usepackage{xxcolor}

\usepackage{tikz}
\usetikzlibrary{calc,intersections,through,backgrounds,patterns,arrows.meta, math,through}
\usepackage{tkz-euclide}
\usepackage[labelformat = brace, position=top]{subcaption}
\usepackage[english]{babel}

\newtheorem{theorem}{Theorem}
\newtheorem{proposition}[theorem]{Proposition}
\newtheorem{lemma}[theorem]{Lemma}

\newtheorem{definition}[theorem]{Definition}

\newtheorem*{theorem*}{Theorem}

\theoremstyle{definition}

\def\XXint#1#2#3{{\setbox0=\hbox{$#1{#2#3}{\int}$ }
\vcenter{\hbox{$#2#3$ }}\kern-.6\wd0}}

\definecolor{Yellow}{rgb}{0.95,0.9,0.0} 
\definecolor{Red}{rgb}{0.8,0.1,0.1}
\definecolor{Green}{rgb}{0.1,0.65,0.2}
\definecolor{Blue}{rgb}{0.1,0.1,0.8}
\definecolor{Purple}{rgb}{0.7,0.1,0.7}
\definecolor{Grey}{rgb}{0.6,0.6,0.6}

\definecolor{YELLOW}{rgb}{0.95,0.9,0.0} 
\definecolor{RED}{rgb}{0.8,0.1,0.1}
\definecolor{GREEN}{rgb}{0.25,0.65,0.1}
\definecolor{BLUE}{rgb}{0.1,0.1,0.8}
\definecolor{PURPLE}{rgb}{0.7,0.1,0.7}

\newcommand{\supp}{\operatorname{supp}}

\newcommand{\dist}{\operatorname{dist}}

\newcommand{\Id}{\operatorname{Id}}

\newcommand{\Rd}[1][d]{{\mathbb{R}^{#1}}}
\allowdisplaybreaks[4]

\newcommand{\dH}{\,\mathrm{d}\mathcal{H}^{d-1}}
\newcommand{\dS}{\,\mathrm{d}S}

\newcommand{\dx}{\,\mathrm{d}x}

\newcommand{\dnablachii}{\,\mathrm{d}|\nabla \chi_i|}

\newcommand{\eps}{\varepsilon}
\renewcommand{\vec}[1]{{\operatorname{#1}}}
\newcommand{\mres}{\mathbin{\vrule height 1.6ex depth 0pt width 0.13ex\vrule height 0.13ex depth 0pt width 1.3ex}}
\newcommand{\cupdot}{\mathbin{\mathaccent\cdot\cup}}

\def\onedot{$\mathsurround0pt\ldotp$}
\def\cdddot#1{
  \mathbin{\vcenter{\baselineskip.67ex
    \hbox{\onedot}\hbox{\onedot}\hbox{\onedot}%
  }}%
}

\usepackage[linktocpage=true,colorlinks=true,linkcolor=Blue,citecolor=Green]{hyperref}

\begin{document}

\title[Local minimizers of the interface length functional]
{Local minimizers of the interface length functional based on a
concept of local paired calibrations}

\author{Julian Fischer}
\address{Institute of Science and Technology Austria (IST Austria), Am~Campus~1, 
3400 Klosterneuburg, Austria}
\email{julian.fischer@ist.ac.at}

\author{Sebastian Hensel}
\address{Hausdorff Center for Mathematics, Universit{\"a}t Bonn, Endenicher Allee 62, 53115 Bonn, Germany}
\email{sebastian.hensel@hcm.uni-bonn.de}

\author{Tim Laux}
\address{Hausdorff Center for Mathematics, Universit{\"a}t Bonn, Endenicher Alllee 62, 53115 Bonn, Germany}
\email{tim.laux@hcm.uni-bonn.de}

\author{Theresa M.\ Simon}
\address{Westf{\"a}lische Wilhelms-Universit{\"a}t M{\"u}nster, Orl{\'e}ansring 10, 48149 M{\"u}nster, Germany}
\email{theresa.simon@uni-muenster.de}
%

\begin{abstract}
We establish that regular flat partitions are locally minimizing for the
interface energy with respect to $L^1$ perturbations of the phases. 
Regular flat partitions are partitions of open sets in $\Rd[2]$ whose network of interfaces consists of 
finitely many straight segments with a singular set made up of finitely many triple junctions 
at which the \textit{Herring angle condition} is satisfied. 
This result not only holds
for the case of the perimeter functional but for a general class of surface tension matrices.
Our proof relies on a localized version of the paired calibration method which was introduced by
Lawlor and Morgan (Pac.\ J.\ Appl.\ Math., 166(1), 1994) in conjunction with a relative energy functional that precisely captures the suboptimality of classical calibration estimates.
Vice versa, we show that any stationary point of the length functional (in a sense of metric spaces) has to be a regular flat partition.
\end{abstract}

\keywords{Optimal partitions, Steiner trees, calibrated geometry}

\maketitle
\tableofcontents

\section{Introduction}

The concept of calibrations is arguably the most elegant tool for proving that a given surface minimizes the area functional.
In this paper, we propose a new concept of \emph{local} calibrations which allows us to verify the \emph{local} minimality of a given configuration.

\medskip

The field of calibrated geometry can be traced back to the work of Wirtinger~\cite{Wirtinger36} and de Rham~\cite{deRham57} on complex geometry and has ties to classical field theory in the calculus of variations developped by Weierstra{\ss}, Kneser, Zermelo, Hilber, Mayer, and Charat{\'e}odory, 
see for example \cite[Chapter 6]{giaquinta1}.
Federer~\cite{Federer65} extended this framework to singular geometries and proved in particular that any complex variety minimizes the area functional. 
It were Harvey and Lawson who coined the term ``calibration'' in their treatise~\cite{HarveyLawson81}. 
Ever since, calibrated geometry has been a vibrant field of study. 
Crucial to our present work is the fundamental contribution of Lawlor and Morgan~\cite{Lawlor-Morgan} which introduces the concept of paired calibrations. 
In particular, this concept allows to show that the cone over the tetrahedron is a global minimizer of the total interfacial area among all ``soap films'' separating each face of the tetrahedron from the others. 
Many other surfaces were shown to be global minimizers of the area functional using calibrations. 
For example, De~Philippis and Paolini~\cite{DePhilippisPaolini} calibrated the Simons cone, which provides a simpler alternative to the original proof by Bombieri, De~Giorgi, and Giusti~\cite{Bombieri}.

\medskip

There are several further extensions of calibrations.
Brakke~\cite{Brakke95} introduced calibrations in covering spaces, which allows to find minimizers in different topology classes.
Morgan~\cite{morgan2005clusters} works with calibrations of currents with coefficients in a group, which allows to globally calibrate certain Steiner trees, see also the work of Marchese and Massaccesi~\cite{MarcheseMassaccesi14}, Pluda and Carioni~\cite{carioni2020carioni,carioni2021different}, and Pluda and Pozzetta~\cite{pluda2022minimizing}.
Furthermore, calibrations can be used for free discontinuity problems such as the minimization of the Mumford--Shah functional, see~\cite{AlbertiBouchitteDalMaso}.

\medskip

The contributions of the present work are the following. 
First, we propose a new concept of \emph{local} paired calibrations in general dimensions, which is a general tool to prove \emph{local} minimality of a given configuration of surface clusters.
Here, local is meant in the natural $L^1$ topology for the enclosed volumes. 
Second, we use this new notion to show that in a given planar domain, any regular flat partition minimizes the total length functional subject to Dirichlet boundary conditions in an $L^1$-neighborhood, see Theorem~\ref{MainResult}.
Finally, vice versa, we show that any local minimizer of the total length functional is a regular flat partition in a subdomain, see Theorem~\ref{thm:stationary}.
The combination of the two results gives a complete characterization of local minimizers of the interface length functional.

\medskip

One of the main technical challenges in proving such results is the complicated non-convex energy landscape of this basic problem.
More precisely, one can easily see that local minimizers are in general not isolated.
For example, our first main result shows that all networks depicted in Figure~\ref{fig:hexagon} are local minimizers of the total length functional. 
But clearly they are part of a one-parameter family with constant total length so that these local minima are degenerate---at least in one direction in state space.
It is this degeneracy of the energy landscape that causes one of the main technical problems: for competitors in this family, our estimates need to be exact.

\medskip

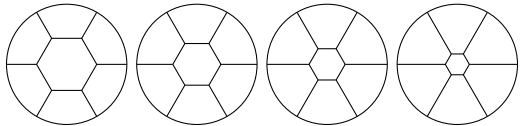
\begin{figure}
	\begin{tikzpicture}[scale=0.4]
		  			\draw (0,0) circle (2);
		  			\foreach \r/\c in {1/black}
			  			{
			  			\draw[\c] (0:\r) \foreach \x in {60,120,...,360} {  -- (\x:\r) };
			  			\foreach \x in {60,120,...,360}
			  				\draw[\c] {  (\x:\r) -- (\x:2) };
			  			}
		  		\end{tikzpicture}
		  		\begin{tikzpicture}[scale=0.4]
		  			\draw (0,0) circle (2);
		  			\foreach \r/\c in {0.8/black}
			  			{
			  			\draw[\c] (0:\r) \foreach \x in {60,120,...,360} {  -- (\x:\r) };
			  			\foreach \x in {60,120,...,360}
			  				\draw[\c] {  (\x:\r) -- (\x:2) };
			  			}
		  		\end{tikzpicture}
		  		\begin{tikzpicture}[scale=0.4]
		  			\draw (0,0) circle (2);
		  			\foreach \r/\c in {0.6/black}
			  			{
			  			\draw[\c] (0:\r) \foreach \x in {60,120,...,360} {  -- (\x:\r) };
			  			\foreach \x in {60,120,...,360}
			  				\draw[\c] {  (\x:\r) -- (\x:2) };
			  			}
		  		\end{tikzpicture}
				\begin{tikzpicture}[scale=0.4]
			  			\draw (0,0) circle (2);
			  			\foreach \r/\c in {0.4/black}
				  			{
				  			\draw[\c] (0:\r) \foreach \x in {60,120,...,360} {  -- (\x:\r) };
				  			\foreach \x in {60,120,...,360}
				  				\draw[\c] {  (\x:\r) -- (\x:2) };
				  			}
		  		\end{tikzpicture}
		  		\caption{All these Steiner trees have the same total length. Theorem~\ref{MainResult} shows that each of these configurations is a local minimizer in some $L^1$-neighborhood.}
	  			\label{fig:hexagon}
		  	\end{figure}

The classical use of calibrations is the following one-line proof of global minimality. 
Let $\Sigma$ be a $d$-dimensional oriented submanifold (say, of some Euclidean space) with boundary $\Gamma$. 
We call $\Sigma $ calibrated if there exists a globally defined closed $d$-form $\omega$ such that $\omega = \textup{Vol}_d$ on $\Sigma$ and $\omega \leq \textup{Vol}_d$ on any $d$-plane. 
Then, for any other $d$-dimensional surface  $\widetilde \Sigma$ spanning the same boundary $\Gamma$ it holds
\begin{align*}
	\textup{Vol}_d(\Sigma) 
	= \int_\Sigma \omega 
	= \int_{\widetilde \Sigma} \omega 
	\leq \textup{Vol}_d(\widetilde \Sigma).
\end{align*}	

\medskip

A crucial obstacle in lifting the above classical proof of \emph{global} minimality to the \emph{local} case in the natural topology is that one has to carefully analyze the excess term~$E[\tilde\Sigma,\omega]=\int_{\tilde \Sigma} (\textup{Vol}_d-\omega)$, which in all previous works is trivially estimated from below by $0$.
However, for general stationary points, such a globally closed $d$-form does not exist and one has to allow the exterior differential $\mathrm{d}\omega$ to be non-zero (but small) away from $\Sigma$.
Our new approach is based on the fact that the error resulting from this can be compensated by the coercivity property of the excess term~$E[\tilde\Sigma,\omega]$. 
Indeed,~$E[\tilde \Sigma,\omega]$ can be viewed as a relative energy or a tilt-excess. 
In our previous work~\cite{FHLS20}, we used the relative energy to prove the weak-strong uniqueness of multi-phase mean curvature flow by monitoring this functional over time and closing a Gronwall argument.
In the present work, we study the more classical static problem of optimal partitions.
Here, we require deeper, more precise coercivity estimates for the relative energy which allow us to prove that Steiner trees are local minimizers of the total interface length functional, see Theorem~\ref{MainResult} below. 
Note that we prove local minimality in the natural $L^1$ topology, corresponding to the strongest result one can hope for.
In particular, they are also local minimizers with respect to the Hausdorff distance, which has more the character of an $L^\infty$ topology.

\medskip

In a recent, related but independent work, Pluda and Pozzetta \cite{pluda2022minimizing} have shown that for equal surface tensions and among competitors with at most three phases, regular flat partitions are local minimizers in the sense that the overall interface length cannot be decreased by modifying the partition in a suitably small neighbourhood of the interfaces (such as the one depicted in Figure~\ref{fig:nbhds1}).
Theorem~\ref{MainResult} in our work substantially strengthens this result in that we establish local minimality with respect to the $L^1$ distance in the full domain.
Furthermore, Theorem~\ref{MainResult} allows for a rather broad class of surface tensions and for an arbitrary number of phases.
In particular, we prove local minimality also with respect to perturbations with phases not present in the original partition.
Also Pluda and Pozzetta rely on a notion of local calibration, in the sense that they restrict classical paired calibrations to their chosen domain.
In contrast, our notion provides a calibration throughout the entire domain by means of a cutoff procedure, enabling us to work in the $L^1$ topology.
Naturally, these vector fields are not divergence-free, which is the reason why we need to carefully exploit the coercivity of the relative energy in order to compensate for the additional terms.

\medskip

Our second main result, Theorem~\ref{thm:stationary}, establishes the  precise structure of stationary points of the interface length functional with respect to the Hausdorff distance. 
This extends the work of Allard and Almgren~\cite{AllardAlmgren} who show that stationary points in the sense of varifolds are unions of straight line segments that satisfy a force balance condition at triple or higher-order junctions. 
Indeed, our self-contained proof shows that the network of interfaces is made up of finitely many straight line segments that only meet in triple junctions at which the Herring condition holds.
The difference to~\cite{AllardAlmgren} is that we do not only make use of variations along smooth vector fields but also more singular deformations, which allows us to probe changes in the topology of the network and enables us to rule out higher-order junctions.
Let us note that also in this theorem, we obtain the sharpest statement with regards to different topologies, as the notion of stationary point in the $L^1$ topology is stronger than in the Hausdorff sense.

\medskip

The more natural notion of stationarity that allows us to consider topology changing deformation is intrinsic to the functional framework and is understood in the sense that the local (downward) slope $|\nabla E|$ vanishes.
This generalization of the gradient is precisely the notion introduced by De~Giorgi, Marino, and Tosques~\cite{DeGiorgiMarinoTosques} to define gradient flows in metric spaces, see also~\cite{AmbrosioGigliSavare}.
In our case of the Hausdorff-type distance $d_H$, their local slope reads
\begin{align*}
	|\nabla E|(\chi)= \max\bigg\{ 0, \limsup_{d_H( \chi,\tilde\chi)\to 0} \frac{  E[\chi]-E[\tilde \chi] }{d_H( \chi,\tilde\chi)}\bigg\}.
\end{align*}
It is easy to see that this slope does not vanish for a partition 
into four quadrants with four $90^\circ$ angles, for example. 
Indeed, the total length decreases linearly with the radius of the ball when replacing the quadruple junction by two triple junctions in the ball (see, e.g., Lemma~\ref{lemma:steinertreeconstrunction}).

%
%
%
			
\section{Main results \& definitions}

\subsection{Main results}
We establish that partitions of open sets in $\Rd[2]$, whose network of interfaces consists of 
finitely many straight segments with a singular set made up of finitely many triple junctions, 
at each of which the \textit{Herring angle condition} is satisfied, are locally minimizing for the
interface energy with respect to $L^1$ perturbations of the phases.

\begin{theorem}
\label{MainResult}
Let $d=2$, $P\geq 2$ be an integer, $\sigma\in\Rd[P\times P]$ an admissible matrix
of surface tensions in the sense of Definition~\ref{DefinitionAdmissibleSurfaceTensions} 
and $\bar\chi=(\bar\chi_1,\ldots,\bar\chi_P)$ a regular flat partition in the sense of 
Definition~\ref{DefinitionFlatPartition} of a bounded
domain $D\subset\Rd[2]$ with smooth boundary~$\partial D$. 

Then there exists $\bar s\in (0,1]$ such that~$\bar\chi$ 
is a local Dirichlet minimizer within an $L^1$ neighborhood of size~$\bar s$ of the interface energy functional~$E$ 
from~\eqref{EnergyFunctionalInterfaces} in the sense that
for all competing partitions $\chi=(\chi_1,\ldots,\chi_P)$ 
of~$D$ with finite interface energy it holds 
$E[\bar\chi]\leq E[\chi]$ whenever $\max_{i=1,\ldots,P} 
\|\chi_i{-}\bar\chi_i\|_{L^1(D)}\leq \bar s$ and $\chi_{i}{-}\bar\chi_{i}$ has zero trace 
on~$\partial D$ for all phases $i\in\{1,\ldots,P\}$. 
\end{theorem}

The reverse implication is addressed by our second main result.
It states that any stationary point of the length functional in a domain has to be a regular flat partition in any compactly contained subdomain.
In particular, any local minimizer in the Hausdorff distance~$d_H$ and hence any local Dirichlet minimizer in~$L^1$ (as in Theorem~\ref{MainResult}) is a regular flat partition.

\begin{theorem}\label{thm:stationary}
	Let $d=2$, $P\geq2$, and suppose the surface tensions are equal, say, $\sigma_{i,j}=1$ for all $i,j\in\{1,\ldots,P\}$ with $i\neq j$.
	Let $\chi = (\chi_1,\ldots,\chi_P)$ be a stationary point of the interface length functional in a bounded convex domain $D\subset \Rd[2]$ with smooth boundary $\partial D$ according to Definition~\ref{def:stationary}.  
	
	Then the collection of all interfaces $\cup_i \partial^\ast \{\chi_i=1\} \cap D$ consists of finitely many straight line segments whose beginning and end may either be on the boundary $\partial D$ or at a triple junction at which three lines meet at a contact angle of $120^\circ$.
\end{theorem}

\subsection{Basic definitions}
We collect in this section the necessary definitions and notation. We start
with the admissible class of surface tension matrices.

\begin{definition}[Admissible matrix of surface tensions]
\label{DefinitionAdmissibleSurfaceTensions}
Let $P\geq 2$ be an integer and $\sigma=(\sigma_{i,j})_{i,j=1,\ldots,P}\in \Rd[P\times P]$ be a 
symmetric matrix with $\sigma_{i,i}= 0$ and $\sigma_{i,j}>0$ for all $i,j\in\{1,\ldots,P\},i\neq j$. 
The matrix $\sigma$ is called an \emph{admissible matrix of surface tensions} if the following 
stability condition is satisfied: there exists a 
non-degenerate $(P{-}1)$-simplex $(q_1,\ldots, q_P)$ in $\Rd[P-1]$ such that 
\begin{align}
\label{eq:coercivitySurfaceTensions}
\sigma_{i,j} = |q_i-q_j|
\end{align} 
for all $i,j \in \{1,\ldots,P\}$.
\end{definition}

Note that an admissible matrix of surface tensions
obviously satisfies the strict triangle inequality, i.e.,
\begin{align}
\label{TriangleInequalitySurfaceTensions}
\sigma_{i,j} < \sigma_{i,k} + \sigma_{k,j}
\end{align}
for all pairwise distinct~$i,j,k\in\{1,\ldots,P\}$.
The strict triangle inequality~\eqref{TriangleInequalitySurfaceTensions} rules out the possibility
of the nucleation of a third phase along an interface between two phases, whereas our
stronger assumption~\eqref{eq:coercivitySurfaceTensions} is a stability 
condition with respect to nucleation of clusters
of phases at a triple junction. 

\begin{definition}[Partitions with finite interface energy, cf.\ \cite{AmbrosioFuscoPallara}]
\label{DefinitionPartition}
Let $D\subset\Rd$ be an open set, $P\geq 2$ be an integer, $\sigma\in\Rd[P\times P]$ be an admissible
matrix of surface tensions in the sense of Definition~\ref{DefinitionAdmissibleSurfaceTensions}, 
and $\chi=(\chi_1,\ldots,\chi_P)\colon D\to \{0,1\}^P$ be a measurable map
such that $\chi_i=\chi_{\Omega_i}$ for all $i\in\{1,\ldots,P\}$
in terms of a family of Lebesgue measurable sets $\Omega=(\Omega_1,\ldots,\Omega_P)$.

We call the map~$\chi$, or equivalently the family of sets~$\Omega$, 
a \emph{partition of $D$ with finite interface energy} if
the following two conditions are satisfied:
\begin{itemize}[leftmargin=0.7cm] 
\item[i)] The family~$\Omega$ is a partition of~$D$ in the
					sense that for all $i,j\in\{1,\ldots,P\}$ with $i\neq j$ we have
					$\Omega_i\subset D$ and the sets $\Omega_i\cap\Omega_j$
					as well as $D\setminus\bigcup_{i=1}^P\Omega_i$ have
					$\mathcal{H}^d$ measure zero.
\item[ii)] Define for each~$i,j\in\{1,\ldots,P\}$ with $i\neq j$ 
					 the interface $S_{i,j}:=\partial^*\Omega_i\cap\partial^*\Omega_j$.
					 We then require that the partition~$\Omega$ has finite interfacial
					 surface area with respect to the surface tension matrix~$\sigma$
					 in the sense that
\begin{align}\label{EnergyFunctionalInterfaces}
E[\chi] := E[\Omega] :=
\sum_{i,j=1,i\neq j}^P \sigma_{i,j}\int_{S_{i,j}}1\,\mathrm{d}\mathcal{H}^{d-1}
< \infty.
\end{align}
\end{itemize}
\end{definition}

Note that for a partition with finite interface energy~$\Omega=(\Omega_1,\ldots,\Omega_P)$ 
of an open set~$D\subset\Rd$, each phase~$\Omega_i$ is a set of finite perimeter in $D$. 

Next we introduce the notion of a regular flat partition, see Figure~\ref{fig:partition} for an illustration, 
for which we show the local minimality of the interface energy functional (cf.\ Theorem~\ref{MainResult} above).

\begin{definition}[Regular flat partition]
\label{DefinitionFlatPartition}
Let $d=2$, let $P\geq 2$ be an integer, and let $\sigma\in\Rd[P\times P]$ be an admissible matrix
of surface tensions in the sense of Definition~\ref{DefinitionAdmissibleSurfaceTensions}. 
Let $\bar\chi=(\bar\chi_1,\ldots,\bar\chi_P)$ be a partition with finite interface energy 
of a bounded domain $D\subset\Rd[2]$ with smooth boundary $\partial D$ 
in the sense of Definition~\ref{DefinitionPartition}. The partition $\bar\chi$,
or equivalently the underlying family of sets $\bar\Omega=(\bar\Omega_1,\ldots,\bar\Omega_P)$ 
such that $\bar\chi_i=\chi_{\bar\Omega_i}$ for all~$i\in\{1,\ldots,P\}$,
is called a \emph{regular flat partition of $D$} if
\begin{itemize}[leftmargin=0.7cm] 
\item[i)] Each phase~$\bar\Omega_i$, $i\in\{1,\ldots,P\}$, is a non-empty open subset of $D$
					consisting of finitely many connected components, each of which having
					a piecewise smooth boundary such that the closure of $\partial^*\bar\Omega_i$
					equals $\partial\bar\Omega_i$ for all	$i\in\{1,\ldots,P\}$.
\item[ii)] For all $i,j\in\{1,\ldots,P\}$ with $i\neq j$,
					the interfaces 
					\begin{align}
					I_{i,j} := \partial\bar\Omega_i\cap\partial\bar\Omega_j
					\end{align}
          are required to be represented by a finite disjoint union of straight line segments.
					For each such interface~$I_{i,j}$ we denote by~$\bar{\vec{n}}_{i,j}$ the associated unit normal
					vector field pointing from~$\bar\Omega_i$ into~$\bar\Omega_j$.
\item[iii)] Two distinct interfaces may intersect in $D$ only at the endpoints of
					 their associated straight line segments, and all such
					 intersection points in $D$ are triple junctions. If three distinct interfaces
					 $I_{i,j}$, $I_{j,k}$ and $I_{k,i}$ meet at a triple junction~$p\in D$, $i,j,k\in\{1,\ldots,P\}$, 
					 they join at angles as encoded by the Herring angle condition
					 \begin{align}\label{HerringAngleCondition}
					 \sigma_{i,j}\bar{\vec{n}}_{i,j}(p) 
					 + \sigma_{j,k}\bar{\vec{n}}_{j,k}(p) 
					 + \sigma_{k,i}\bar{\vec{n}}_{k,i}(p) = 0.
					 \end{align}
\item[iv)] Two distinct interfaces do not intersect on the boundary $\partial D$.
					 If an interface intersects the boundary $\partial D$, then only
					 at exactly one of its endpoints and by forming an angle~$\theta\in (0,\pi)$.
\end{itemize}
\end{definition}

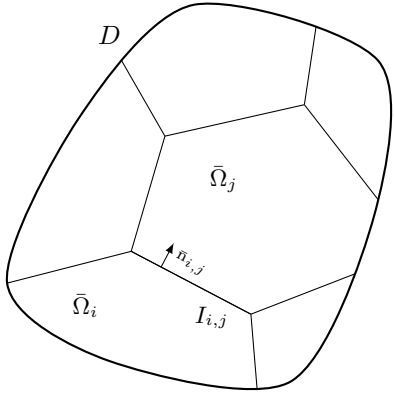
\begin{figure}
	\begin{tikzpicture}[scale=.8]
		%
		\draw[thick] plot [smooth cycle] coordinates {(0:3)(40:4)(90:3.5)(140:2.5)(200:3.3)(240:2.8)(300:3.2)};
		\draw (200:3.3) -- (210:1.2);
		\draw (210:1.2) -- (110:1.4);
		\draw (110:1.4) -- (115:2.84);
		\draw (110:1.4) -- (45:2.6);
		\draw (45:2.6) -- (57:3.73);
		\draw (45:2.6) -- (5:3.07);
		\draw (210:1.2) -- node[near start, sloped, above] {{\tiny$\qquad\;\;\bar{\vec{n}}_{i,j}$}} (300:1.9);
		\draw (210:1.2) --	node[near start,sloped,inner sep=0cm,above,anchor=south west,
														 minimum height=.35cm,minimum width=.35cm]	(N) {} (300:1.9);
		\draw[-latex] (N.south west) -- (N.north west);
		\draw (300:1.9) -- (290:3.08);
		\draw (300:1.9) -- (340:2.85);
		%
	%
		%
		\draw (0.5,0.6) node {{\small$\bar\Omega_j$}};
	 	\draw (0.3,-1.7) node {{\small$I_{i,j}$}};
		\draw (-1.8,-1.5) node {{\small$\bar\Omega_i$}};
		\draw (-1.4,3) node {{$D$}};
	\end{tikzpicture}
	\caption{An illustration of a regular flat partition of a convex and bounded domain~$D$
					 with smooth boundary~$\partial D$.	\label{fig:partition}}
\end{figure}

Next we state our precise definition of stationarity. 
As mentioned in the introduction, the notion of stationary point has to be slightly stronger than only permitting domain variations along smooth vector fields to allow us to change the topology of a partition.
Stationarity here means that the local (downward) slope $|\nabla E|$ (cf.\ \cite{DeGiorgiMarinoTosques, AmbrosioGigliSavare}) vanishes.
This slope is of the form
\begin{align*}
	|\nabla E|(\chi)= \max\bigg\{ 0, \limsup_{d_H( \chi,\tilde\chi)\to 0} \frac{  E[\chi]-E[\tilde \chi] }{d_H( \chi,\tilde\chi)}\bigg\},
\end{align*}
where the distance $d_H$ is precisely given by
\begin{align*}
	d_H( \chi, \tilde\chi)
	:= \max_{1\leq i\leq P} d\big(\{ \chi_i=1\},\{\tilde \chi_i=1\}\big),
\end{align*}
and~$d$ on the right-hand side denotes the Hausdorff distance of sets. 
We state the relation $|\nabla E|(\chi) =0$ in the equivalent and for us more 
convenient form~\eqref{EqCondition} below.

\begin{definition}[Stationary points of the interface length functional]\label{def:stationary}
	Let~$\sigma \in \Rd[P\times P]$ be an admissible matrix of surface tensions in the sense of Definition~\ref{DefinitionAdmissibleSurfaceTensions} and let $\chi=(\chi_1,\ldots,\chi_P)$ be a finite partition of a bounded domain~$D \subset \Rd[d]$. 
	We say that~$\chi$ is a \emph{stationary point of the total interface length functional~$E$ in the domain~$D$} provided~$\chi$ partitions~$\partial D$ into finitely many arc segments and for any other partition $\tilde\chi = (\tilde \chi_1,\ldots,\tilde \chi_P)$ of $D$ with the same trace on~$\partial D$, it holds
	\begin{align}
		\label{EqCondition}
		E[\tilde \chi]
		\geq
		E[\chi]
		-o\big(d_H(\tilde \chi,\chi)\big).
	\end{align} 
\end{definition}

Lawlor and Morgan \cite{Lawlor-Morgan} introduced the concept of a paired
calibration for a partition. A key ingredient for the proof of Theorem~\ref{MainResult} 
is a suitable ``local variant'' of this notion. As a preparation towards the 
definition, we decompose the network of interfaces $\mathcal{I}:=\bigcup_{i,j\in\{1,\ldots,P\},i\neq j} I_{i,j}$ of a given 
regular flat partition $\bar\chi$ according to its topological features.
I.e., we split it into straight line segments on the one hand, and triple junctions
as well as boundary endpoints on the other hand. 
Suppose that the partition $\bar\chi$ has $N$ of such topological features $\mathcal{T}_n$, $1\leq n\leq N$. 
We then split $\{1,\ldots,N\}=:\mathcal{C} \cupdot \mathcal{P} \cupdot \mathcal{B}$ 
with the convention that indices $p\in\mathcal{P}$
enumerate the triple junctions present in the network, that indices $b\in\mathcal{B}$
enumerate the boundary endpoints, and that indices $c\in\mathcal{C}$ enumerate 
the connected components of the non-empty two-phase interfaces $I_{i,j}\neq\emptyset$
(which in our particular case are straight line segments). 
We then define for all $p\in\mathcal{P}$ resp.\ $b\in\mathcal{B}$ 
a singleton $\mathcal{T}_p = \{\mathrm{t}_p\}$ resp.\ $\mathcal{T}_b=\{\mathrm{t}_b\}$ containing
the associated triple junction $\mathrm{t}_p \in \Rd[2]$ resp.\
the associated boundary endpoint $\mathrm{t}_b \in \Rd[2]$.
For all $c\in\mathcal{C}$ we define a set $\mathcal{T}_c\subset I_{i,j}$ representing
the corresponding connected straight line segment (without boundary points)
of a non-empty two-phase interface~$I_{i,j}$.

Let a phase $i\in\{1,\ldots,P\}$ and a topological feature $n\in\{1,\ldots,N\}$ be fixed. 
We say that the $i$-th phase of the partition $\bar\chi$ is
\emph{present at the topological feature~$\mathcal{T}_{n}$}
if $\partial\bar\Omega_i\cap\mathcal{T}_{n}\neq\emptyset$.
Otherwise, we say that the phase is \emph{absent at $\mathcal{T}_n$}.
For a straight line segment $c\in\mathcal{C}$ and a triple junction $p\in\mathcal{P}$
(resp.\ a boundary endpoint $b\in\mathcal{B}$),
we write $c\sim p$ (resp.\ $c\sim b$) if and only if $\mathcal{T}_p$ (resp.\ $\mathcal{T}_b$) 
is an endpoint of $\mathcal{T}_c$. Otherwise, we write $c\not\sim p$ (resp.\ $c\not\sim b$).

With this language in place, we turn our attention to a convenient yet slightly technical
notion of a pair of admissible localization scales $(\bar r,\delta)$ for a regular flat partition.

\begin{definition}[Admissible localization scales]
\label{def:locRadius}
Let $D\subset\Rd[2]$ be a bounded domain with smooth boundary $\partial D$. 
Let $P\geq 2$ be an integer, and let $\bar\chi=(\bar\chi_1,\ldots,\bar\chi_P)$ 
be a regular flat partition of~$D$ in the sense of Definition~\ref{DefinitionFlatPartition}
with underlying family of sets $\bar\Omega=(\bar\Omega_1,\ldots,\bar\Omega_P)$.
We call $(\bar r, \delta)\in (0,1] {\times} (0,\frac{1}{2}]$ a
\emph{pair of admissible localization scales for the partition $\bar\chi$} if:
\begin{itemize}[leftmargin=0.7cm] 
\item[i)] For all distinct triple junctions $p,p'\in\mathcal{P}$ and all distinct
				  boundary endpoints $b,b'\in~\mathcal{B}$ we have
					\begin{align*}
					B_{2\bar r}(\mathcal{T}_{p}) \cap B_{2\bar r}(\mathcal{T}_{p'}) &= \emptyset,
					\quad B_{2\bar r}(\mathcal{T}_{p}) \subset\subset D\setminus B_{2\bar r}(\partial D),
					\\
					B_{2\bar r}(\mathcal{T}_{b}) \cap B_{2\bar r}(\mathcal{T}_{b'}) &= \emptyset.
					\end{align*}
					Moreover, for all~$n\in\mathcal{P}\cup\mathcal{B}$ and all~$c\in\mathcal{C}$
					such that~$c\not\sim n$ it holds
					\begin{align*}
					B_{2\bar r}(\mathcal{T}_n) \cap B_{2\bar r}(\mathcal{T}_c) = \emptyset,
					\end{align*}
					and for all~$c\in\mathcal{C}$ such that~$c\not\sim b$ for all~$b\in\mathcal{B}$,
					we have
					\begin{align*}
						B_{2\bar r}(\mathcal{T}_c) \subset\subset D\setminus B_{2\bar r}(\partial D).
					\end{align*}
\item[ii)] For all distinct straight line segments $c,c'\in\mathcal{C}$ it holds
					 $B_{2\bar r}(\mathcal{T}_{c}) \cap B_{2\bar r}(\mathcal{T}_{c'}) \neq \emptyset$
					 if and only if there is a triple junction $p\in\mathcal{P}$ such that
					 $c\sim p$ and $c'\sim p$. In this case, we on top require that 
					 \begin{align*}
					 B_{\delta\bar r}(\mathcal{T}_{c}) \cap B_{\delta\bar r}(\mathcal{T}_{c'})
					 \subset\subset B_{\bar r/2}(\mathcal{T}_p). 
					 \end{align*}
\item[iii)] Let $i\in\{1,\ldots,P\}$. For each connected component $C_i$ of the $i$-th
					  phase~$\bar\Omega_i$, the set $C_i\setminus\{\dist(\cdot,\partial C_i)\leq \bar r\}$
						is an open and connected set with compact Lipschitz boundary.
\item[iv)] The smooth domain boundary~$\partial D$ satisfies the exterior resp.\ the interior
						ball condition on scale~$2\bar r$.
\end{itemize}
\end{definition}

It is immediately clear from Definition~\ref{DefinitionFlatPartition} of a regular flat partition
that an associated admissible pair of localization scales $(\bar r,\delta)$ always exists. Given
such an admissible pair, it is convenient to introduce a ``dumbbell-type neighborhood'' of the network
of interfaces $\mathcal{I} = \bigcup_{i,j\in\{1,\ldots,P\},\,i\neq j} I_{i,j}$ as follows:
\begin{align}
\label{eq:dumbbellNbhdNetwork}
U^{\mathcal{I}}_{(\bar r,\delta)} := \bigcup_{c\in\mathcal{C}} \bigg(
B_{\delta\bar r}(\mathcal{T}_c)\setminus \bigcup_{n\in\mathcal{P}\cup\mathcal{B}} 
B_{\bar r}(\mathcal{T}_n)\bigg) \cup \bigcup_{n\in\mathcal{P}\cup\mathcal{B}} B_{\bar r}(\mathcal{T}_n)
\end{align}
see Figure~\ref{fig:nbhds}. 
There are natural analogs of the decomposition of the network $\mathcal{I}$
into its topological features and its dumbbell-type neighborhood on the
level of individual interfaces and phases, respectively. To this end, 
fix two distinct phases $i,j\in\{1,\ldots,P\}$
such that $I_{i,j}\neq\emptyset$.
We then denote by $\mathcal{N}_{i,j} \subset \{1,\ldots,N\}$ the subset
of those topological features $n\in\{1,\ldots,N\}$ with both the phases $i$ and $j$ being present
at $\mathcal{T}_n$ (i.e., $\mathcal{T}_n\subset I_{i,j}$). 
In analogy to the splitting $\{1,\ldots,N\} =: \mathcal{C} \cupdot \mathcal{P} \cupdot \mathcal{B}$ 
we decompose $\mathcal{N}_{i,j} =: \mathcal{C}_{i,j} \cupdot \mathcal{P}_{i,j} \cupdot \mathcal{B}_{i,j}$.
More precisely, $\mathcal{C}_{i,j}$ labels the straight line segments present in the interface $I_{i,j}$,
whereas $\mathcal{P}_{i,j}$ and $\mathcal{B}_{i,j}$ label the endpoints of $I_{i,j}$
which are triple junctions or boundary endpoints, respectively. A ``dumbbell-type neighborhood''
for the interface $I_{i,j}\neq\emptyset$ is then simply defined as
\begin{align}
\label{eq:dumbbellNbhdInterface}
U^{I_{i,j}}_{(\bar r,\delta)} := \bigcup_{c\in\mathcal{C}_{i,j}} \bigg(
B_{\delta\bar r}(\mathcal{T}_c)\setminus \bigcup_{n\in\mathcal{P}_{i,j}\cup\mathcal{B}_{i,j}} 
B_{\bar r}(\mathcal{T}_n)\bigg) \cup \bigcup_{n\in\mathcal{P}_{i,j}\cup\mathcal{B}_{i,j}} B_{\bar r}(\mathcal{T}_n).
\end{align}
A ``dumbbell-type neighborhood'' for the boundary $\mathcal{I}_i := \bigcup_{j=1,\,j\neq i}^P I_{i,j}$
of the $i$-th phase $\bar\Omega_i$ is finally provided by
\begin{align}
\label{eq:dumbbellNbhdPhase}
U^{\mathcal{I}_i}_{(\bar r,\delta)} := 
\bigcup_{\substack{j=1,\, j\neq i \\ I_{i,j}\neq\emptyset}}^P U^{I_{i,j}}_{(\bar r,\delta)}.
\end{align}
We refer to Figure~\ref{fig:nbhds} for an illustration of these neighborhoods.

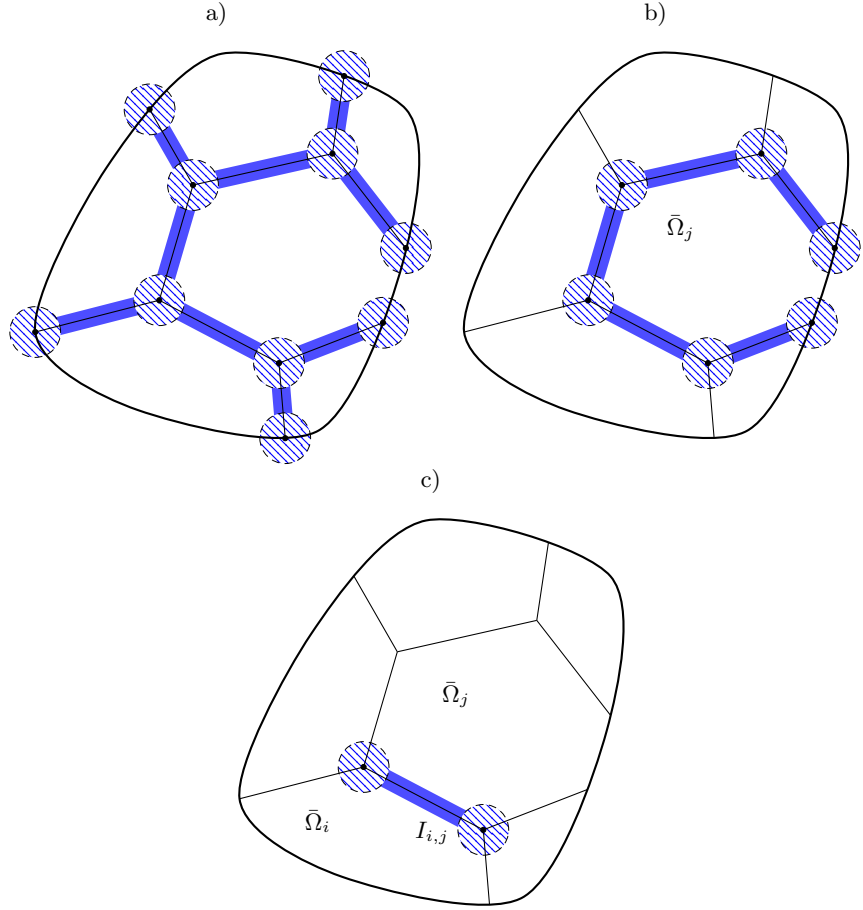
\begin{figure}
\centering
\subcaptionbox{\label{fig:nbhds1}}{
\centering
\begin{tikzpicture}[scale=.8]
\foreach \p in {(110:1.4),(115:2.84),(45:2.6),(57:3.73),(5:3.07),(340:2.85),(290:3.08),(300:1.9),(210:1.2),(200:3.3)}
{
		\draw[name path = circle\p, draw=none] \p circle (12pt);
}
		%
		\draw[line width=7pt, color=blue, opacity=0.7] (110:1.4) -- (115:2.84);
		\draw[line width=7pt, color=blue, opacity=0.7] (200:3.3) -- (210:1.2);
		\draw[line width=7pt, color=blue, opacity=0.7] (210:1.2) -- (110:1.4);
		\draw[line width=7pt, color=blue, opacity=0.7] (110:1.4) -- (45:2.6);
		\draw[line width=7pt, color=blue, opacity=0.7] (45:2.6) -- (57:3.73);
		\draw[line width=7pt, color=blue, opacity=0.7] (45:2.6) -- (5:3.07);
		\draw[line width=7pt, color=blue, opacity=0.7] (210:1.2) -- (300:1.9);
		\draw[line width=7pt, color=blue, opacity=0.7] (300:1.9) -- (290:3.08);
		\draw[line width=7pt, color=blue, opacity=0.7] (300:1.9) -- (340:2.85);
\foreach \p in {(110:1.4),(115:2.84),(45:2.6),(57:3.73),(5:3.07),(340:2.85),(290:3.08),(300:1.9),(210:1.2),(200:3.3)}
{
		\fill[color=white] \p circle (12pt);
		\draw[pattern = north west lines, pattern color = blue, 
					densely dashed, line width=0.2pt] \p circle (12pt);
		\draw[fill] \p circle (1.2pt);
}
		\draw (110:1.4) -- (115:2.84);
		\draw (200:3.3) -- (210:1.2);
		\draw (210:1.2) -- (110:1.4);
		\draw (110:1.4) -- (45:2.6);
		\draw (45:2.6) -- (57:3.73);
		\draw (45:2.6) -- (5:3.07);
		\draw (210:1.2) -- (300:1.9);
		\draw (300:1.9) -- (290:3.08);
		\draw (300:1.9) -- (340:2.85);
		\draw[thick] plot [smooth cycle] coordinates {(0:3)(40:4)(90:3.5)(140:2.5)(200:3.3)(240:2.8)(300:3.2)};
\end{tikzpicture}
}
\subcaptionbox{\label{fig:nbhds2}}{
\centering
\begin{tikzpicture}[scale=.8]
\foreach \p in {(110:1.4),(45:2.6),(5:3.07),(340:2.85),(300:1.9),(210:1.2)}
{
		\draw[name path = circle\p, draw=none] \p circle (12pt);
}
		%
		\draw[line width=7pt, color=blue, opacity=0.7] (210:1.2) -- (110:1.4);
		\draw[line width=7pt, color=blue, opacity=0.7] (110:1.4) -- (45:2.6);
		\draw[line width=7pt, color=blue, opacity=0.7] (45:2.6) -- (5:3.07);
		\draw[line width=7pt, color=blue, opacity=0.7] (210:1.2) -- (300:1.9);
		\draw[line width=7pt, color=blue, opacity=0.7] (300:1.9) -- (340:2.85);
\foreach \p in {(110:1.4),(45:2.6),(5:3.07),(340:2.85),(300:1.9),(210:1.2)}
{
		\fill[color=white] \p circle (12pt);
		\draw[pattern = north west lines, pattern color = blue, 
					densely dashed, line width=0.2pt] \p circle (12pt);
		\draw[fill] \p circle (1.2pt);
}
		\draw (110:1.4) -- (115:2.84);
		\draw (200:3.3) -- (210:1.2);
		\draw (210:1.2) -- (110:1.4);
		\draw (110:1.4) -- (45:2.6);
		\draw (45:2.6) -- (57:3.73);
		\draw (45:2.6) -- (5:3.07);
		\draw (210:1.2) -- (300:1.9);
		\draw (300:1.9) -- (290:3.08);
		\draw (300:1.9) -- (340:2.85);
		\draw (0.5,0.6) node {{\small$\bar\Omega_j$}};
		\draw[thick] plot [smooth cycle] coordinates {(0:3)(40:4)(90:3.5)(140:2.5)(200:3.3)(240:2.8)(300:3.2)};
\end{tikzpicture}
}
\subcaptionbox{\label{fig:nbhds3}}{
\centering
\begin{tikzpicture}[scale=.8]
\foreach \p in {(300:1.9),(210:1.2)}
{
		\draw[name path = circle\p, draw=none] \p circle (12pt);
}
		%
		\draw[line width=7pt, color=blue, opacity=0.7] (210:1.2) -- (300:1.9);
\foreach \p in {(300:1.9),(210:1.2)}
{
		\fill[color=white] \p circle (12pt);
		\draw[pattern = north west lines, pattern color = blue, 
					densely dashed, line width=0.2pt] \p circle (12pt);
		\draw[fill] \p circle (1.2pt);
}
		\draw (110:1.4) -- (115:2.84);
		\draw (200:3.3) -- (210:1.2);
		\draw (210:1.2) -- (110:1.4);
		\draw (110:1.4) -- (45:2.6);
		\draw (45:2.6) -- (57:3.73);
		\draw (45:2.6) -- (5:3.07);
		\draw (210:1.2) -- (300:1.9);
		\draw (300:1.9) -- (290:3.08);
		\draw (300:1.9) -- (340:2.85);
		\draw (0.5,0.6) node {{\small$\bar\Omega_j$}};
	 	\draw (0.1,-1.7) node {{\small$I_{i,j}$}};
		\draw (-1.8,-1.5) node {{\small$\bar\Omega_i$}};
		\draw[thick] plot [smooth cycle] coordinates {(0:3)(40:4)(90:3.5)(140:2.5)(200:3.3)(240:2.8)(300:3.2)};
\end{tikzpicture}
}
\caption{An illustration of the ``dumbbell-type neighborhood'' with respect to
				 a) the network of interfaces~$\mathcal{I}$,
				 b) the grain boundary~$\mathcal{I}_j$ of the $j$th phase~$\bar\Omega_j$,
				 c) the interface~$I_{i,j}$ between phase~$\bar\Omega_i$ and phase~$\bar\Omega_j$.	
				 \label{fig:nbhds}}
\end{figure}

\begin{definition}[Local paired calibration]
\label{DefinitionLocalCalibration}
Let $D\subset\Rd[2]$ be a bounded domain with smooth boundary $\partial D$,
and let $P\geq 2$ be an integer. Let $\bar\chi=(\bar\chi_1,\ldots,\bar\chi_P)$ 
be a regular flat partition of~$D$ in the sense of Definition~\ref{DefinitionFlatPartition}
with respect to an admissible matrix of surface tensions $\sigma\in\Rd[P\times P]$,
and let $\bar\Omega=(\bar\Omega_1,\ldots,\bar\Omega_P)$ be the 
underlying family of sets. Denote the interfaces by $I_{i,j}=\partial\bar\Omega_i\cap\partial\bar\Omega_j$
for all $i,j\in\{1,\ldots,P\}$ with $i\neq j$, and their union by 
$\mathcal{I}:=\bigcup_{i,j\in\{1,\ldots,P\},\,i\neq j} I_{i,j}$.
Let $(\bar r,\delta)$ be an admissible pair of localization scales for the
partition~$\bar\chi$ in the sense of Definition~\ref{def:locRadius}.

We call a family of Lipschitz vector fields $\xi_i\colon\overline{D}\to\Rd[2]$, $i\in\{1,\ldots,P\}$, 
a \emph{local paired calibration for the partition $\bar\chi$ with respect to $(\bar r,\delta)$} 
if the following conditions are satisfied for all $i,j\in\{1,\ldots,P\}$ with $i\neq j$:
\begin{itemize}[leftmargin=0.7cm] 
\item[i)]  \emph{(Coercivity through locality constraint)} It holds
					 \begin{align*}
					 \supp\xi_i \subset \overline{D} \cap U^{\mathcal{I}}_{(\bar r,\delta)},
					 \end{align*}
					 with the ``dumbbell-type neighborhood'' $U^{\mathcal{I}}_{(\bar r,\delta)}$
					 of the network of interfaces $\mathcal{I}$ defined in~\eqref{eq:dumbbellNbhdNetwork}.
\item[ii)] \emph{(Coercivity through length constraint)} If $I_{i,j} \neq \emptyset$, then 
					 \begin{align*}
					 \xi_i-\xi_j &= \sigma_{i,j}\bar{\vec{n}}_{i,j} \text{ along } I_{i,j}, \\
					 |\xi_i-\xi_j| &\leq \sigma_{i,j} \text{ on } \overline{D},
					 \end{align*}
           and there exists a constant $\delta_1=\delta_1(\sigma)\in (0,1)$
           such that the vector field $\xi_i{-}\xi_j$ is short in the sense that 
					 \begin{align*}
					 |\xi_i-\xi_j| &\leq \delta_1\sigma_{i,j} \quad\text{on } 
					 \overline{D} \setminus U^{I_{i,j}}_{(\frac{\bar r}{2},\delta)},
					 \end{align*}
					 with the ``dumbbell-type neighborhood'' $U^{I_{i,j}}_{(\frac{\bar r}{2},\delta)}$
					 of the interface $I_{i,j}$ defined in~\eqref{eq:dumbbellNbhdInterface}.
					 In case $I_{i,j} = \emptyset$, the condition of the previous display 
					 is required throughout $\overline{D}$. 
\item[iii)] \emph{(Coercivity through orientation constraint)} 
            For each constant~$\kappa\in (0,1)$,
					  there exists $\delta_2=\delta_2(\sigma,\kappa)\in (0,\delta)$ such that if $I_{i,j}\neq\emptyset$, 
						then the validity of the inequality $\big|\frac{\xi_i(x)-\xi_j(x)}{\sigma_{i,j}}\big|>1-\delta_2$
						for $x\in \overline{D}\cap\{\dist(\cdot,I_{i,j})\leq 2\bar r\}$ 
						implies the estimate $\big|\frac{\xi_i(x)-\xi_j(x)}{\sigma_{i,j}}-\bar{\vec{n}}_{i,j}\big| 
						\leq \kappa$ for the same point $x\in\overline{D}$.
\item[iv)] \emph{(Coercivity through flux constraint)} 
           There exists~$\delta_3=\delta_3(\sigma)\in (0,\delta)$ such that if
					 $I_{i,j}\neq\emptyset$, the following constraint on the divergence
				   of the vector field $\xi_i{-}\xi_j$ is satisfied:
           \begin{align*}
					 \nabla\cdot\xi_i - \nabla\cdot\xi_j = 0 \quad\text{in } 
					 \{\dist(\cdot,\mathcal{I}_{i}) < \delta_3\bar r\}\cap\bar\Omega_j.
					 \end{align*}
\end{itemize}
\end{definition}

We will verify in Lemma~\ref{LemmaExistenceLocalCalibration} below that each
regular flat partition of $D$ admits a local paired calibration in the above sense.

\subsection{Auxiliary results}
A second key ingredient for the proof of Theorem~\ref{MainResult}---next to the notion
of a local paired calibration---is provided by the following elementary result 
relating the interface energy of two partitions $\chi$ and $\bar\chi$.

\begin{lemma}[Relative energy equality]
\label{LemmaEnergyEstimate}
Let $d\geq 2$ and $D\subset\Rd$ be a bounded domain with smooth boundary. Let $P\geq 2$ be an integer, 
and let $\chi=(\chi_1,\ldots,\chi_P)$ resp.\ $\bar\chi=(\bar\chi_1,\ldots,\bar\chi_P)$
be two partitions of $D$ with finite interface energy
in the sense of Definition~\ref{DefinitionPartition} with respect to an admissible
matrix of surface tensions $\sigma\in\Rd[P\times P]$. Denote then by $\mathcal{S}=(S_{i,j})_{i,j\in\{1,\ldots,P\},i\neq j}$
resp.\ $\mathcal{I}=(I_{i,j})_{i,j\in\{1,\ldots,P\},i\neq j}$ the associated family of interfaces
of the partitions $\chi$ resp.\ $\bar\chi$. We also define for all distinct $i,j\in\{1,\ldots,P\}$
\begin{align}
\label{UnitNormalsChi}
\vec{n}_{i,j} &:=\frac{\nabla \chi_{j}}{|\nabla \chi_{j}|}
=-\frac{\nabla \chi_i}{|\nabla \chi_i|},
\quad \mathcal{H}^{d-1}\text{ almost everywhere on } S_{i,j},
\\ \label{UnitNormalsBarChi}
\bar{\vec{n}}_{i,j} &:=\frac{\nabla \bar\chi_{j}}{|\nabla \bar\chi_{j}|}
=-\frac{\nabla \bar\chi_i}{|\nabla \bar\chi_i|},
\quad \mathcal{H}^{d-1}\text{ almost everywhere on } I_{i,j},
\end{align}
the (measure-theoretic) unit normal vectors of the interfaces $S_{i,j}$ resp.\ $I_{i,j}$ pointing 
from the $i$\emph{th} to the $j$\emph{th} phase of the partition $\chi$ resp.\ $\bar\chi$.

Assume there exists a family of Lipschitz vector fields $(\xi_i)_{i=1,\ldots,P}$ on $\overline{D}$
satisfying $|\xi_i-\xi_j|\leq \sigma_{i,j}$
on $\overline{D}$, as well as $\xi_i-\xi_j=\sigma_{i,j}\bar{\vec{n}}_{i,j}$ up to 
$\mathcal{H}^{d{-}1}$ null sets on the interfaces $I_{i,j}$ for all $i,j\in\{1,\ldots,P\}$
with $i\neq j$. Then it holds
\begin{equation}\label{EnergyEstimate}
\begin{aligned}
E[\chi] &= E[\bar\chi] + E[\chi|\bar\chi]
+ \sum_{i=1}^P \int_{\partial D} 2(\chi_i{-}\bar\chi_i) \vec{n}_{\partial D}\cdot\xi_i \dS
\\&~~~
+ \sum_{i,j=1,i\neq j}^P \int_{D} 2(\chi_i{-}\bar\chi_i) \bar\chi_j (\nabla\cdot\xi_i{-}\nabla\cdot\xi_j) \dx,
\end{aligned}
\end{equation}
with the non-negative \emph{relative energy functional} $E[\chi|\bar\chi]\geq 0$ defined by
\begin{align}
\label{DefinitionRelativeEntropy}
E[\chi|\bar\chi] :=
\sum_{i,j=1,i\neq j}^P \sigma_{i,j} \int_{S_{i,j}} 
1-\frac{\xi_i{-}\xi_j}{\sigma_{i,j}}\cdot\vec{n}_{i,j} \,\mathrm{d}\mathcal{H}^{d-1}.
\end{align}
\end{lemma}

A family of vector fields $(\xi_i)_{i=1,\ldots,P}$ as in Lemma~\ref{LemmaEnergyEstimate}
which in addition satisfies $\nabla\cdot\xi_i=0$ \emph{globally in $D$} for all $i\in\{1,\ldots,P\}$
is called a \emph{paired calibration} for the partition $\bar\chi$ following the work of 
Lawlor and Morgan \cite{Lawlor-Morgan}. It is a classical fact that the existence of
a paired calibration implies global minimality for the underlying network of interfaces. 
Note that this result is of course recovered by \eqref{EnergyEstimate}.

However, the above reasoning towards global minimality does not rely on any of the properties of
the relative energy functional $E[\chi|\bar\chi]$ except for its non-negativity. 
The idea behind the proof of our main result, Theorem~\ref{MainResult},
stems from the observation that for local minimality 
it may suffice to enforce the divergence constraint only in a small neighborhood around
the interfaces (or more precisely, in the quantitative sense of Definition~\ref{DefinitionLocalCalibration}~\textit{iv)}). 
For the remaining contributions from the bulk term in~\eqref{EnergyEstimate},
we aim to argue that they can be absorbed by the relative energy functional due to its coercivity properties; 
at least for small perturbations of the phases in $L^1$.

Before we will make this rigorous, we first have to show that regular flat partitions indeed admit 
local paired calibrations in the sense of Definition~\ref{DefinitionLocalCalibration}, provided that the domain is locally convex at the boundary pounts $b \in \mathcal{B}$ -- a technical assumption that will be removed in the proof of Theorem~\ref{MainResult}. This is
the content of the following result.

\begin{lemma}[Existence of local paired calibrations for regular flat partitions]
\label{LemmaExistenceLocalCalibration}
Let $d=2$, $P\geq 2$, and $\bar\chi=(\bar\chi_1,\ldots,\bar\chi_P)$ be a regular flat partition 
of a bounded domain with smooth boundary $D\subset\Rd[2]$
in the sense of Definition~\ref{DefinitionFlatPartition}. Moreover,
let $(\bar r,\delta)\in (0,1]{\times}(0,\frac{1}{2}]$ be an associated pair of admissible localization
scales in the sense of Definition~\ref{def:locRadius}, such that for all $b\in \mathcal{B}$ the set $B_{2 \bar r}(\mathcal{T}_b) \cap D$ is convex.
Then there exists a family of Lipschitz vector fields 
$(\xi_i)_{i=1,\ldots,P}$ on $\overline{D}$ such that this family is a local
paired calibration for $\bar\chi$ with respect to $(\bar r,\delta)$ in the sense of 
Definition~\ref{DefinitionLocalCalibration}. 
\end{lemma}

\section{Proof of auxiliary results}

\subsection{Relative energy equality: Proof of Lemma~\ref{LemmaEnergyEstimate}}
Using the definitions~\eqref{EnergyFunctionalInterfaces} and~\eqref{DefinitionRelativeEntropy}, we may expand the energy functional by adding zero in form of
\begin{align}
\label{auxEnergyEstimate1}
E[\chi] = E[\chi|\bar\chi]
+ \sum_{i,j=1,i\neq j}^P \int_{S_{i,j}} (\xi_i{-}\xi_j)\cdot\vec{n}_{i,j} \dH.
\end{align}
Non-negativity of the functional $E[\chi|\bar\chi]$ follows from the bound
$|\xi_i-\xi_j|\leq \sigma_{i,j}$ being valid for all $i\neq j$, cf.~Definition~\ref{DefinitionLocalCalibration} ii). Exploiting
that $\mathcal{H}^{d-1}$ almost everywhere on the interface $S_{i,j}$
the skew-symmetry relation $\vec{n}_{i,j}=-\vec{n}_{j,i}$ holds true, we
may further compute by the definition of $\vec{n}_{i,j}$ in \eqref{UnitNormalsChi}
and an integration by parts
\begin{align}
\nonumber
&\sum_{i,j=1,i\neq j}^P \int_{S_{i,j}} (\xi_i{-}\xi_j)\cdot\vec{n}_{i,j} \dH
\\&\nonumber
= -\sum_{i=1}^P \int_{D} 2\xi_i\cdot\frac{\nabla\chi_i}{|\nabla\chi_i|} \dnablachii
\\&\label{auxEnergyEstimate2}
= \sum_{i=1}^P \int_{D} 2\chi_i(\nabla\cdot\xi_i) \dx
+ \sum_{i=1}^P \int_{\partial D} 2\chi_i\vec{n}_{\partial D}\cdot\xi_i \dS.
\end{align}
Since $\chi$ resp.\ $\bar\chi$ are partitions of~$D$, i.e., 
the identities $\sum_{i=1}^P\chi_i=1$, $\sum_{i=1}^P\bar\chi_i = 1$
as well as $\bar\chi_i\bar\chi_j = 0$ are satisfied~$\mathcal{H}^d$ almost everywhere in~$D$
for all $i\neq j$, it follows from adding zero that
\begin{equation}\label{auxEnergyEstimate3}
\begin{aligned}
&\sum_{i=1}^P \int_{D} 2\chi_i(\nabla\cdot\xi_i) \dx
\\&
= \sum_{i,j=1,i\neq j}^P \int_{D} 2(\chi_i{-}\bar\chi_i)\bar\chi_j(\nabla\cdot\xi_i{-}\nabla\cdot\xi_j) \dx
+ \sum_{j=1}^P \int_{D} 2\bar\chi_j(\nabla\cdot\xi_j) \dx.
\end{aligned}
\end{equation}
Another integration by parts together with the fact that $\xi_i-\xi_j=\sigma_{i,j}\bar{\vec{n}}_{i,j}$
on the interface $I_{i,j}=\partial^*\bar\Omega_i\cap\partial^*\bar\Omega_j$ moreover shows that
\begin{align}
\label{auxEnergyEstimate4}
\sum_{j=1}^P \int_{D} 2\bar\chi_j(\nabla\cdot\xi_j) \dx 
= E[\bar\chi] - \sum_{j=1}^P \int_{\partial D} 2\bar\chi_j\vec{n}_{\partial D}\cdot\xi_j \dS.
\end{align}
Hence, the asserted equation \eqref{EnergyEstimate} is a consequence of the identities
\eqref{auxEnergyEstimate1}--\eqref{auxEnergyEstimate4}. \qed

\subsection{Existence of local paired calibrations: Proof of Lemma~\ref{LemmaExistenceLocalCalibration}}
Recall first from the discussion preceding Definition~\ref{def:locRadius}
the decomposition of the network of interfaces $\mathcal{I}=\bigcup_{i,j\in\{1,\ldots,P\},i\neq j} I_{i,j}$
into its topological features. The proof is now split into several steps.
We start with the construction of local building blocks for the family of
calibration vector fields $(\xi_i)_{i\in\{1,\ldots,P\}}$.

\textit{Step 1: Construction of auxiliary vectors $\vv{\xi^{c}_{i}}$ for each phase $i$ and $c\in\mathcal{C}$.}
Let a phase $i\in\{1,\ldots,P\}$ and a straight line segment $c\in\mathcal{C}$ be fixed. 
We then define
\begin{align}
\label{def:buildingBlockXiTwoPhase}
\vv{\xi^c_i}:= \frac{\sigma_{i,k}}{2}\bar{\vec{n}}_{l,k}
+\frac{\sigma_{i,l}}{2}\bar{\vec{n}}_{k,l} 
\quad\text{if } \mathcal{T}_c\subset I_{k,l} 
\text{ for } k,l\in\{1,\ldots,P\},\,k\neq l.
\end{align}
In the case of $i\in\{k,l\}$, i.e., the phase $i$ being present at
the straight line segment $\mathcal{T}_c \subset I_{k,l}$,
note that the above definition reduces to 
\begin{align}
\label{def:buildingBlockXiTwoPhaseMajorityPhases}
\vv{\xi^c_i} = \frac{\sigma_{i,l}}{2}\bar{\vec{n}}_{i,l} \quad \text{if } i=k, \quad
\vv{\xi^c_i} = \frac{\sigma_{i,k}}{2}\bar{\vec{n}}_{i,k} \quad \text{if } i=l,
\end{align}
as a consequence of $\sigma_{k,k}=\sigma_{l,l}=0$. 

\textit{Step 2: Construction of auxiliary vectors $\vv{\xi^{b}_{i}}$ for each phase $i$ and $b\in\mathcal{B}$.}
Let a phase $i\in\{1,\ldots,P\}$ and a boundary endpoint $b\in\mathcal{B}$ be fixed.
We then define in analogy to the case of straight line segments from the previous step
\begin{align}
\label{def:buildingBlockXiBoundary}
\vv{\xi^b_i} := \frac{\sigma_{i,k}}{2}\bar{\vec{n}}_{l,k}
+\frac{\sigma_{i,l}}{2}\bar{\vec{n}}_{k,l} 
\quad\text{if } \mathcal{T}_b\subset I_{k,l} 
\text{ for } k,l\in\{1,\ldots,P\},\,k\neq l.
\end{align}

\textit{Step 3: Construction of auxiliary vectors $\vv{\xi^{p}_{i}}$ for each phase $i$ and $p\in\mathcal{P}$.}
We proceed by exploiting the stability condition for triple junctions~\eqref{eq:coercivitySurfaceTensions}.
Let a phase $i\in\{1,\ldots,P\}$ and a triple junction $p\in\mathcal{P}$ be fixed.
Assume the pairwise distinct phases $k,l,m\in\{1,\ldots,P\}$
to be present at the triple junction $\mathcal{T}_p=\{\mathrm{t}_p\}$.
As a consequence of the stability condition~\eqref{eq:coercivitySurfaceTensions},
there exists a non-degenerate $(P{-}1)$-simplex $(q^p_1,\ldots, q^p_P)$ in $\Rd[P-1]$ such that
the matrix of surface tensions is given by
\begin{align}
\label{eq:embeddingSurfaceTensionMatrix}
\sigma_{i,j} = |q_i^p - q_j^p| \quad \text{for all } i,j\in\{1,\ldots,P\}.
\end{align}
Observe that we may assume without loss of generality that $q^p_k = 0$.

The triangle $(q_k^p,q_l^p,q_m^p)$ is non-degenerate and spans a plane 
$E^p$ in $\Rd[P-1]$ containing the origin $0\in\Rd[P-1]$. 
We denote the orthogonal projection onto $E^p$ by $\pi^p$, and  
then define
\begin{align}
\label{def:buildingBlockXiTripleJunction}
	\Rd[2]\ni\vv{\xi^p_i} := R^p\pi^p q_i^p
\end{align}
for the affine isometry $R^p\colon E^p \to \Rd[2]$ determined by
\begin{align}
\label{def:buildingBlockXiTripleJunctionMajorityPhases}
\begin{cases}
R^p q_k^p = \frac{\sigma_{k,l}}{3}\bar{\vec{n}}_{k,l}(\mathrm{t}_p)
+\frac{\sigma_{k,m}}{3}\bar{\vec{n}}_{k,m}(\mathrm{t}_p),
\\
R^p q_l^p = \frac{\sigma_{l,m}}{3}\bar{\vec{n}}_{l,m}(\mathrm{t}_p)
+\frac{\sigma_{l,k}}{3}\bar{\vec{n}}_{l,k}(\mathrm{t}_p), 
\\
R^p q_m^p = \frac{\sigma_{m,k}}{3}\bar{\vec{n}}_{m,k}(\mathrm{t}_p)
+\frac{\sigma_{m,l}}{3}\bar{\vec{n}}_{m,l}(\mathrm{t}_p).
\end{cases}
\end{align}
(The first equation of~\eqref{def:buildingBlockXiTripleJunctionMajorityPhases}
fixes the translation due to $q^k_p=0$, the other two the linear isometry by means of
the Herring condition $\sigma_{k,l} \bar{\vec{n}}_{k,l} + \sigma_{l,m}\bar{\vec{n}}_{l,m} 
+ \sigma_{m,k}\bar{\vec{n}}_{m,k} = 0$ at the triple junction $\mathcal{T}_p$.)

\textit{Step 4: The calibration property for the auxiliary vectors $\vv{\xi^n_i}$.}
Before we move on with the construction of the calibration vector fields $\xi_i$---for which the 
auxiliary vectors $\vv{\xi^n_i}$ at topological features $n\in\{1,\ldots,N\}$ from the previous three steps
serve as the main building blocks---we first prove that for all phases $i,j\in\{1,\ldots,P\}$ 
with $i\neq j$ and all topological features $n\in\{1,\ldots,N\}$ we have
\begin{align}
\label{eq:calibrationProperty}
\vv{\xi_i^n} - \vv{\xi_j^n} = \sigma_{i,j}\bar{\vec{n}}_{i,j}
\quad\text{if both phases } i \text{ and } j \text{ are present at } \mathcal{T}_n. 
\end{align}
Moreover, we claim that there exists a constant $\delta_1(\sigma)\in (0,1)$ such that
for all phases $i,j\in\{1,\ldots,P\}$ with $i\neq j$ and all topological features $n\in\{1,\ldots,N\}$ it holds
\begin{align}
\label{Length}
|\vv{\xi_i^n} - \vv{\xi_j^n}| \leq \delta_1(\sigma)\sigma_{i,j}
\quad\text{if either phase } i \text{ or phase } j \text{ is absent at } \mathcal{T}_n.
\end{align}

\textit{Proof of~\eqref{eq:calibrationProperty}:} In case of a straight line segment
$c\in\mathcal{C}$ with $\mathcal{T}_c\subset I_{i,j}$ for distinct phases $i,j\in\{1,\ldots,P\}$, 
the identity~\eqref{eq:calibrationProperty}
is an immediate consequence of~\eqref{def:buildingBlockXiTwoPhaseMajorityPhases},
$\bar{\vec{n}}_{i,j}=-\bar{\vec{n}}_{j,i}$
and the symmetry of the surface tension matrix.
The case of a boundary endpoint is treated in exactly the same way.
Finally, in case of $p\in\mathcal{P}$
with pairwise distinct phases $i,j,k\in\{1,\ldots,P\}$ being present at the triple junction~$\mathcal{T}_p$, 
we may compute based on~\eqref{def:buildingBlockXiTripleJunction}, \eqref{def:buildingBlockXiTripleJunctionMajorityPhases}
and the Herring angle condition~\eqref{HerringAngleCondition}
\begin{align*}
\vv{\xi_i^p} - \vv{\xi_j^p} &=
\frac{\sigma_{i,j}}{3}\bar{\vec{n}}_{i,j}+\frac{\sigma_{i,k}}{3}\bar{\vec{n}}_{i,k}
- \Big(\frac{\sigma_{j,i}}{3}\bar{\vec{n}}_{j,i}+\frac{\sigma_{j,k}}{3}\bar{\vec{n}}_{j,k}\Big)
= \sigma_{i,j}\bar{\vec{n}}_{i,j}.
\end{align*}

\textit{Proof of~\eqref{Length}:}
In case of a straight line segment $c\in\mathcal{C}$ such that 
$\mathcal{T}_c\subset I_{k,l}$ for some pairwise distinct phases $k,l\in\{1,\ldots,P\}$, and at least one of
the two distinct phases $i,j\in\{1,\ldots,P\}$ not being present at $\mathcal{T}_c$, this follows 
from~\eqref{def:buildingBlockXiTwoPhase} in form of
\begin{align*}
\vv{\xi^c_i} - \vv{\xi^c_j} = \frac{\sigma_{i,j}}{2}
\Big(\frac{\sigma_{l,i}{-}\sigma_{l,j}}{\sigma_{i,j}}\bar{\vec{n}}_{k,l}
+\frac{\sigma_{k,i}{-}\sigma_{k,j}}{\sigma_{i,j}}\bar{\vec{n}}_{l,k}\Big)
\end{align*}
and the strict triangle inequality \eqref{TriangleInequalitySurfaceTensions}.
The same argument applies to boundary endpoints $b\in\mathcal{B}$.

Let $p\in\mathcal{P}$ be a triple junction with pairwise distinct phases $k,l,m\in\{1,\ldots,P\}$
being present, and assume that at least one of the two distinct phases $i,j\in\{1,\ldots,P\}$
is absent at $\mathcal{T}_p$. The argument in this regime is based on the following simple observation
(recall for what follows the notation from \textit{Step 3} of this proof):
projecting an out-of-plane vector onto $E^p$ via $\pi^p$ it becomes short.
More precisely, we argue in favor of \eqref{Length} in two subcases: 

If exactly one of the two indices $i$ or $j$---say for concreteness $j$---corresponds to a phase being 
present at $\mathcal{T}_p$, then $\pi^p q_j^p = q_j^p$. 
Note that with the simplex $(q_1^p,\ldots,q_P^p)$, also the $3$-simplex $(q_k^p,q_l^p,q_m^p,q_i^p)$ is non-degenerate, 
so that $q_i^p$ necessarily is an out-of-plane vector with respect to the plane $E^p$, 
hence $\pi^p q_i^p \neq q_i^p$, which implies the strict inequality in this 
subcase by~\eqref{def:buildingBlockXiTripleJunction}. 

If both $i$ and $j$ correspond to phases being absent at $\mathcal{T}_p$, we only need to argue that $q_i^p-q_j^p$ does not 
lie in the linear subspace $E^p\subset\Rd[P-1]$. Otherwise the $4$-simplex $(q_k^p,q_l^p,q_m^p,q_i^p,q_j^p)$
would be degenerate, a contradiction. So indeed, \eqref{Length} holds true
as a consequence of the definition~\eqref{def:buildingBlockXiTripleJunction}.  

\textit{Step 5: Choice of an auxiliary localization scale $\delta'=\delta'(\sigma) \in (0,\delta)$.}
Recall that $(\bar r,\delta)$ denotes a pair of admissible localization scales
for the regular flat partition~$\bar\chi$, see Definition~\ref{def:locRadius}. 
The aim of this step is to introduce another localization scale
$\delta'\lesssim\delta$ which will be of frequent use
in the subsequent step for the construction of the vector fields $\xi_i$.

For each straight line segment $c\in\mathcal{T}_c$, denote by $\mathcal{L}^c$ the line in $\Rd[2]$
containing $\mathcal{T}_c$, choose a unit normal vector $\bar{\vec{n}}_c$ along it, and introduce a change of variables
\begin{align}
\label{eq:defDiffeoStraightLineSegments}
\Psi_{\mathcal{L}_c}\colon \mathcal{L}_c \times (-\infty,\infty) \to \Rd[2],
\quad (x,s) \mapsto x + s\bar{\vec{n}}_c.
\end{align}
Fix now a triple junction $p\in\mathcal{P}$. Note that for every
pair of distinct straight line segments 
$c,c'\in\mathcal{C}$ with $c\sim p$ and $c'\sim p$ it holds
as a consequence of Property~\textit{ii)} in Definition~\ref{def:locRadius}
\begin{align}
\label{eq:supportPropertyByLocalization1}
\Psi_{\mathcal{L}_c}\big(\mathcal{T}_c {\times} (-\delta\bar r, \delta\bar r)\big)
\cap \Psi_{\mathcal{L}_{c'}}\big(\mathcal{T}_{c'} {\times} (-\delta\bar r, \delta\bar r)\big) 
\subset\subset B_{\bar r/2}(\mathcal{T}_p).
\end{align}

\begin{figure}
\centering
\begin{tikzpicture}[scale=.8]
\fill[color=green!10] (0,0) -- (127:2.4) arc(127:250:2.4) -- cycle;
\fill[pattern = north west lines, pattern color=green] (0,0) -- (127:2.4) arc(127:250:2.4) -- cycle;
\fill[color=yellow!10] (0,0) -- (250:2.4) arc(250:377:2.4) -- cycle;
\fill[pattern = north east lines, pattern color=yellow] (0,0) -- (250:2.4) arc(250:377:2.4) -- cycle;
\fill[color=red!10] (0,0) -- (17:2.4) arc(17:127:2.4) -- cycle;
\fill[pattern = vertical lines, pattern color=red] (0,0) -- (17:2.4) arc(17:127:2.4) -- cycle;
\fill[color=blue!50, opacity=0.2] (-4.5,2.4) rectangle (0,-2.4);
\draw[thick, densely dashed] (0,0) circle (2.4cm);
\draw[fill] (0,0) circle (2pt);
\draw (0,0) node[below] {\small$\mathcal{T}_p$};
\draw[thick] (-4.5,0) -- node[very near start, below] {\small$\mathcal{T}_c$} (0,0);
\draw[thick,rotate=140] (-3,0) -- (0,0);
\draw[thick,rotate=-106] (-3,0) -- (0,0);
\draw (-1.5,1) node {\small$W^p_c$};
\draw (-3.3,2.1) node {\tiny$\Psi(\mathcal{T}_c{\times}({-}\bar r,\bar r))$};
\end{tikzpicture}
\caption{An illustration of the wedge decomposition of the neighborhood
around a triple junction~$\mathcal{T}_p$. For instance,
the wedge~$W^p_c$ associated to the straight line segment~$\mathcal{T}_c$
with endpoint at the triple junction~$\mathcal{T}_p$ is given by
the green region hatched from top left to bottom right.
\label{fig:wedges}}
\end{figure}
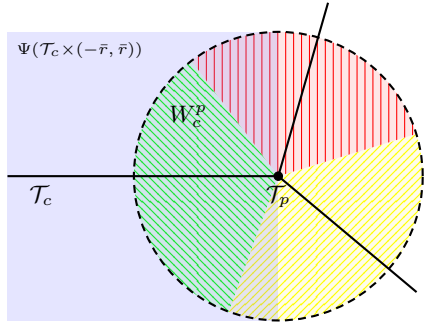

There moreover exists a family of three pairwise disjoint open wedges 
$(W^p_c)_{c\in\mathcal{C},c\sim p}$ in~$\Rd[2]$ such that 
\begin{align}
\label{eq:decompByWedges}
\mathcal{H}^2\bigg(B_{\bar r}(\mathcal{T}_p)\setminus 
\bigcup_{c\in\mathcal{C},c\sim p} W^p_{c} \bigg)
&= 0,
\\ \label{eq:inclusionWedges}
B_{\bar r}(\mathcal{T}_p) \cap W^p_c &\subset 
\Psi_{\mathcal{L}_c}\big(\mathcal{T}_c{\times}(-\bar r,\bar r)\big) 
\text{ for all } c\in\mathcal{C},\,c\sim p,
\end{align}
cf.\ Figure~\ref{fig:wedges} for an illustration of this wedge decomposition.
In fact, these wedges may be constructed through angle bisection at the triple junction $\mathcal{T}_p$
in the sense that
\begin{align}
B_{\bar r}(\mathcal{T}_p) \cap \partial W^p_c \cap \partial W^p_{c'}
= B_{\bar r}(\mathcal{T}_p) \cap \big\{x\in\Rd[2]\colon \dist(x,\mathcal{T}_c) = \dist(x,\mathcal{T}_{c'})\big\}
\end{align}
for all distinct $c,c'\in\mathcal{C}$ such that $c\sim p$ and $c'\sim p$.

Denote next by $P_{\mathcal{T}_c}(x)\in\mathcal{T}_c$ the unique nearest point on 
the straight line segment~$\mathcal{T}_c$ to a point $x\in\Rd[2]$.
Choosing a constant $\delta'=\delta'(\sigma,p)\in (0,\delta)$ small enough,
we may then guarantee that
\begin{equation}
\label{eq:supportPropertyByLocalization2}
\begin{aligned}
&B_{\bar r}(\mathcal{T}_p) 
\cap \bigg\{x\colon\dist(P_{\mathcal{T}_c}(x),\mathcal{T}_p) 
\in \Big(\frac{1}{4}\bar r,\frac{3}{4}\bar r\Big)\bigg\} 
\cap \Psi_{\mathcal{L}_c}\big(\mathcal{T}_c {\times} (-\delta'\bar r, \delta'\bar r)\big)
\\&
\subset\subset B_{\bar r}(\mathcal{T}_p) \cap W^p_c.
\end{aligned}
\end{equation}
We finally define 
\begin{align}
\label{eq:defAuxLocScale}
\delta':=\min_{p\in\mathcal{P}} \delta'(\sigma,p) \in (0,\delta).
\end{align}

\textit{Step 6: Construction of the vector fields $\xi_i$.} 
We proceed by introducing two classes of cutoff functions:
first, a family of localization functions $(\eta_c)_{c\in\mathcal{C}}$
and second, for each triple junction $p\in\mathcal{P}$ a family
of weights $(\lambda^p_c)_{c\in\mathcal{C},c\sim p}$. The former
will take care of restricting the support of the vector fields $(\xi_i)_{i\in\{1,\ldots,P\}}$
to the dumbbell-type neighborhood $U^{\mathcal{I}}_{(\bar r,\delta)}$ from~\eqref{eq:dumbbellNbhdNetwork} 
for the network of interfaces $\mathcal{I}$.
The latter in turn will be used to transition in a consistent manner between local auxiliary
calibration vectors $\vv{\xi_i^c}$ and $\vv{\xi_i^p}$ for straight line segments $c\in\mathcal{C}$
and triple junctions $p\in\mathcal{P}$ such that $c\sim p$.

Let $\theta\colon\Rd[]\to [0,1]$ 
be a smooth function such that $\theta\equiv 1$ on $(-\infty,\frac{1}{4}]$, 
$\theta\equiv 0$ on $[\frac{3}{4},\infty)$, and such that $\theta$ 
is strictly decreasing in $(\frac{1}{4},\frac{3}{4})$.
We then define for each straight line segment $c\in\mathcal{C}$ 
a localization function 
\begin{align}
\label{eq:localizationStraightLineSegment}
\eta_c\colon\Rd[2]\to [0,\infty),\quad
x \mapsto \theta\Big(\frac{\dist(x,\mathcal{T}_c)}{\delta'\bar r}\Big).
\end{align}
Let now a triple junction $p\in\mathcal{P}$ be fixed. We then define
for each straight line segment $c\in\mathcal{C}$ with $c\sim p$ a weight
\begin{align}
\label{eq:localizationTripleJunction}
\lambda_c^p\colon\Rd[2] \to [0,\infty),\quad
x \mapsto \theta\Big(\frac{\dist(P_{\mathcal{T}_c}(x),\mathcal{T}_p)}{\bar r}\Big),
\end{align}
where $P_{\mathcal{T}_c}x\in\mathcal{T}_c$ again denotes the unique nearest point on 
the straight line segment~$\mathcal{T}_c$ to a point $x\in\Rd[2]$.
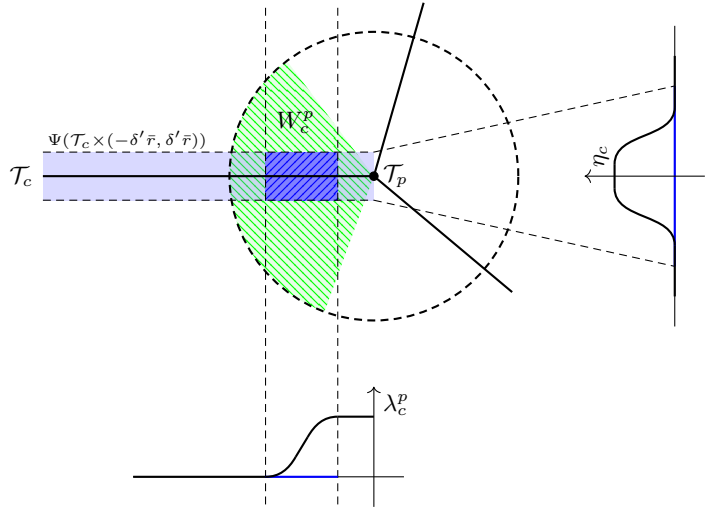
\begin{figure}
\centering
\begin{tikzpicture}[scale=.8]
\fill[color=green!10] (0,0) -- (127:2.4) arc(127:250:2.4) -- cycle;
\fill[pattern = north west lines, pattern color=green] (0,0) -- (127:2.4) arc(127:250:2.4) -- cycle;
\fill[color=blue!50, opacity=0.3] (-5.5,0.4) rectangle (0,-0.4);
\fill[color=blue!50] (-1.8,0.4) rectangle (-0.6,-0.4);
\fill[pattern = north east lines, pattern color=blue] (-1.8,0.4) rectangle (-0.6,-0.4);
\draw[thick, densely dashed] (0,0) circle (2.4cm);
\draw[fill] (0,0) circle (2pt);
\draw (0,0) node[right] {\small$\mathcal{T}_p$};
\draw[thick] (-5.5,0) -- node[left,pos=0] {\small$\mathcal{T}_c$} (0,0);
\draw[thick,rotate=140] (-3,0) -- (0,0);
\draw[thick,rotate=-106] (-3,0) -- (0,0);
\draw (-1.3,0.9) node {\small$W^p_c$};
\draw (-4.1,0.6) node {\tiny$\Psi(\mathcal{T}_c{\times}({-}\delta'\bar r,\delta'\bar r))$};
\draw[densely dashed] (-1.8,2.8) -- (-1.8,-5.5);
\draw[densely dashed] (-0.6,2.8) -- (-0.6,-5.5);
\draw (-4,-5) -- (0.5,-5);
\draw[thick] (-4,-5) -- (-1.8,-5);
\draw[thick] (-0.6,-4) -- (0,-4);
\draw[thick] (-1.3,-4.67) -- (-1.1,-4.34);
\draw[thick, color=blue] (-1.8,-5) -- (-0.6,-5);
\draw[thick,smooth] (-1.8,-5) to[out=0,in=-120] (-1.3,-4.67);
\draw[thick,smooth] (-1.1,-4.34) to[out=60,in=180] (-0.6,-4);
\draw[->] (0,-5.5) -- (0,-3.5) node [right,pos=0.85] {\small$\lambda^p_c$};
\draw (5,2.5) -- (5,-2.5);
\draw[densely dashed] (-5.5,0.4) -- (0,0.4);
\draw[densely dashed] (-5.5,-0.4) -- (0,-0.4);
\draw[densely dashed] (0,0.4) -- (5,1.5);
\draw[densely dashed] (0,-0.4) -- (5,-1.5);
\draw[thick,blue] (5,1.5) -- (5,-1.5);
\draw[<-] (3.5,0) -- node[right,pos=0.14,rotate=90] {\small$\eta_c$} (5.5,0);
\draw[thick] (5,2) -- (5,1.15);
\draw[thick] (4,0.2) -- (4,-0.2);
\draw[thick] (5,-2) -- (5,-1.15);
\draw[thick,smooth] (5,1.15) to[out=-90,in=90] (4,0.2);
\draw[thick,smooth] (5,-1.15) to[out=90,in=-90] (4,-0.2);
\end{tikzpicture}
\caption{The cutoff function~$\eta_c$ localizes (in the direction 
				 orthogonal to~$\mathcal{T}_c$) to the tubular neighborhood
				 $\Psi(\mathcal{T}_c{\times}(-\delta'\bar r,\delta'\bar r))$.
				 The weight~$\lambda^p_c$ is used to interpolate in the construction
				 of~$\xi_i$ between the building blocks
				 associated with the straight line segment~$\mathcal{T}_c$ or the triple 
				 junction~$\mathcal{T}_p$, respectively. After localization by means 
				 of~$\eta_c$, the interpolation is carried out in the blue region hatched 
				 from bottom left to top right.\label{fig:cutoffs}}
\end{figure}
We refer to Figure~\ref{fig:cutoffs} for a visualization of these auxiliary functions.
Note that as a consequence of the choice~\eqref{eq:defAuxLocScale}
of the auxiliary localization scale~$\delta' < \delta$, it follows 
from~\eqref{eq:supportPropertyByLocalization2} and the
definitions~\eqref{eq:localizationStraightLineSegment}--\eqref{eq:localizationTripleJunction}
that
\begin{align}
\label{eq:supportPropertyByLocalization3}
B_{\bar r}(\mathcal{T}_p) \cap \Psi_{\mathcal{L}_c}(\mathcal{T}_c{\times}(-\bar r,\bar r))
\cap \supp\big(\eta_c\lambda^p_c\big) \cap \supp\big(\eta_c(1{-}\lambda^p_c)\big)
\subset\subset B_{\bar r}(\mathcal{T}_p) \cap W^p_c.
\end{align}

We have everything in place to finally state the definition of the family of Lipschitz vector fields 
$(\xi_i)_{i\in\{1,\ldots,P\}}$. To this end, fix a phase $i\in\{1,\ldots,P\}$ and
consider a straight line segment $c\in\mathcal{C}$. Away from triple junctions and boundary endpoints
but in the vicinity of $\mathcal{T}_c$ we let
\begin{align}
\label{eq:defLocalCalibrationStraighLineSegment}
\xi_i(x) := \eta_c(x)\vv{\xi_i^c}, \quad\text{for } x \in \overline{D} 
\cap \Psi_{\mathcal{L}_c}\big(\mathcal{T}_c {\times} (-\delta\bar r, \delta \bar r)\big)
\setminus\bigcup_{n\in\mathcal{P}\cup\mathcal{B}} B_{\bar r}(\mathcal{T}_n).
\end{align}
Note that this definition is consistent across different straight line segments
thanks to the localization property~\eqref{eq:supportPropertyByLocalization1},
the definition~\eqref{eq:localizationStraightLineSegment} and $\delta'<\delta$
due to~\eqref{eq:defAuxLocScale}. Consider next a boundary endpoint $b\in\mathcal{B}$. We then define
\begin{align}
\label{eq:defLocalCalibrationBoundary}
\xi_i(x) := \eta_c(x)\vv{\xi_i^b}, \quad\text{for } x \in \overline{D} \cap B_{\bar r}(\mathcal{T}_b) 
\text{ and the unique } c\in\mathcal{C} \text{ with } c\sim b.
\end{align}
Finally, let $p\in\mathcal{P}$ be a triple junction. 
In this case, we define
\begin{align}
\label{eq:defLocalCalibrationTripleJunction}
\xi_i(x) := \eta_c(x) \Big( \lambda^p_c(x) \vv{\xi^p_i}
+ \big(1{-}\lambda^p_c(x)\big) \vv{\xi^c_i}\Big),
\, x \in B_{\bar r}(\mathcal{T}_p) \cap \overline{W^p_c},\,c\in\mathcal{C},\, c\sim p.
\end{align}
Because of~\eqref{eq:decompByWedges} this indeed defines the vector field $\xi_i$ in 
the whole ball $B_{\bar r}(\mathcal{T}_p)$. 

In view of~\eqref{eq:supportPropertyByLocalization1} as well as the 
definitions~\eqref{eq:defAuxLocScale}, \eqref{eq:localizationStraightLineSegment}
and~\eqref{eq:defLocalCalibrationStraighLineSegment}--\eqref{eq:defLocalCalibrationTripleJunction}
it is safe to define $\xi_i$ outside the dumbbell-type neighborhood $U^{\mathcal{I}}_{(\bar r,\delta)}$
from~\eqref{eq:dumbbellNbhdNetwork} of the network of interfaces~$\mathcal{I}$ 
by means of
\begin{align}
\label{eq:defLocalCalibrationBulk}
\xi_i(x) := 0, \quad x \in \overline{D} \setminus U^{\mathcal{I}}_{(\bar r,\delta)}.
\end{align}
We even observe that the definitions~\eqref{eq:defAuxLocScale}--\eqref{eq:localizationTripleJunction}
and~\eqref{eq:defLocalCalibrationStraighLineSegment}--\eqref{eq:defLocalCalibrationBulk}
in fact entail qualitative continuity of $\xi_i$ in $\overline{D}$ for each $i\in\{1,\ldots,P\}$
as a consequence of the properties~\eqref{eq:supportPropertyByLocalization1}--\eqref{eq:supportPropertyByLocalization2}
and~\eqref{eq:supportPropertyByLocalization3}.

\textit{Step 7: Properties of the vector fields $\xi_i$.} 
Since $\bar\chi$ is a regular flat partition of~$D$ in the sense of Definition~\ref{DefinitionFlatPartition}, not only
the distance function $x\mapsto\dist(x,\mathcal{T}_c)$ to~$\mathcal{T}_c$, but also the nearest point projection $x\mapsto P_{\mathcal{T}_c}(x)$ onto~$\mathcal{T}_c$ 
is Lipschitz for each straight line segment $c\in\mathcal{C}$. In particular,
all localization functions from~\eqref{eq:localizationStraightLineSegment}
and all weights from~\eqref{eq:localizationTripleJunction} are Lipschitz.
The Lipschitz property therefore lifts to $\xi_i$ for all $i\in\{1,\ldots,P\}$ because of the
definitions~\eqref{eq:defLocalCalibrationStraighLineSegment}--\eqref{eq:defLocalCalibrationBulk}.

Property~\textit{i)} of Definition~\ref{DefinitionLocalCalibration} is immediate from the
construction of the $\xi_i$ in the previous step. For all what follows in this step, fix phases $i,j\in\{1,\ldots,P\}$ 
with $i\neq j$. To verify the properties \textit{ii)}--\textit{iv)} from 
Definition~\ref{DefinitionLocalCalibration}, it suffices to check them in the 
dumbbell-type neighborhood $U^{\mathcal{I}}_{(\bar r,\delta)}$ of the network of interfaces $\mathcal{I}$. 

Let $c\in\mathcal{C}$ be a straight line segment. Away from triple junctions
and boundary endpoints but in the vicinity of $\mathcal{T}_c$, we have as a consequence
of the definition~\eqref{eq:defLocalCalibrationStraighLineSegment} the representation
\begin{align}
\label{eq:representationCalibrationVectors1}
\xi_i-\xi_j = \eta_c(\vv{\xi^c_i} - \vv{\xi^c_j})
\quad\text{in } \overline{D} \cap \Psi_{\mathcal{L}_c}\big(\mathcal{T}_c 
{\times} (-\delta\bar r,\delta\bar r)\big)
\setminus\bigcup_{n\in\mathcal{P}\cup\mathcal{B}} B_{\bar r}(\mathcal{T}_n).
\end{align}
Hence, we obtain in view of~\eqref{eq:calibrationProperty} and~\eqref{Length} that
\begin{equation}
\label{eq:representationCalibrationVectors1b}
\begin{aligned}
\xi_i-\xi_j &= \eta_c\sigma_{i,j}\bar{\vec{n}}_{i,j} && \text{if } \mathcal{T}_c \subset I_{i,j}, \\
|\xi_i - \xi_j| &\leq \eta_c\delta_1(\sigma)\sigma_{i,j} && \text{else},
\end{aligned}
\end{equation}
in $\overline{D} \cap \Psi_{\mathcal{L}_c}\big(\mathcal{T}_c 
{\times} (-\delta\bar r,\delta\bar r)\big)
\setminus\bigcup_{n\in\mathcal{P}\cup\mathcal{B}} B_{\bar r}(\mathcal{T}_n)$. 
The previous display immediately implies Property~\textit{ii)} within
$\overline{D} \cap \Psi_{\mathcal{L}_c}\big(\mathcal{T}_c 
{\times} (-\delta\bar r,\delta\bar r)\big)
\setminus\bigcup_{n\in\mathcal{P}\cup\mathcal{B}} B_{\bar r}(\mathcal{T}_n)$
due to the choice~\eqref{eq:localizationStraightLineSegment} of the cutoff~$\eta_c$
and the localization properties of the scales~$(\bar r,\delta)$.
For a proof of Property~\textit{iii)}, consider a point
$x \in \overline{D} \cap \Psi_{\mathcal{L}_c}\big(\mathcal{T}_c 
{\times} (-\delta\bar r,\delta\bar r)\big)
\setminus\bigcup_{n\in\mathcal{P}\cup\mathcal{B}} B_{\bar r}(\mathcal{T}_n)$
such that $|\frac{\xi_i(x)-\xi_j(x)}{\sigma_{i,j}}|>1-\delta_2(\sigma,\kappa)$.
Restricting $\delta_2(\sigma,\kappa)\in (0,1{-}\delta_1(\sigma))$, it follows
from~\eqref{eq:representationCalibrationVectors1b} that necessarily
$\mathcal{T}_c\subset I_{i,j} \neq \emptyset$. Hence, further restricting the constant to 
$\delta_2(\sigma,\kappa) < (1{-}\delta_1(\sigma)) \wedge \frac{\kappa}{2}$
then implies by~\eqref{eq:representationCalibrationVectors1b} and the assumption on~$x$ that 
$\eta_c(x) > 1 {-} \delta_2(\sigma) \geq 1 {-} \frac{\kappa}{2}$, and therefore
$|\frac{\xi_i(x)-\xi_j(x)}{\sigma_{i,j}} {-} \bar{\vec{n}}_{i,j}| = 1 - \eta_c(x) < \frac{\kappa}{2}$.
In other words, Property~\textit{iii)} holds true throughout 
$\overline{D} \cap \Psi_{\mathcal{L}_c}\big(\mathcal{T}_c 
{\times} (-\delta\bar r,\delta\bar r)\big)
\setminus\bigcup_{n\in\mathcal{P}\cup\mathcal{B}} B_{\bar r}(\mathcal{T}_n)$.
For a proof of Property~\textit{iv)}, we may in fact
simply choose $\delta_3(\sigma) \in (0, \frac{\delta'}{4})$
due to the following two observations (recall from~\eqref{eq:defAuxLocScale}
the definition of the scale~$\delta'$).
First, if $\mathcal{T}_c\not\subset I_{i,j}$ then
$\bar\Omega_j\cap \Psi_{\mathcal{L}_c}\big(\mathcal{T}_c 
{\times} (-\delta\bar r,\delta\bar r)\big)
\setminus\bigcup_{n\in\mathcal{P}\cup\mathcal{B}} B_{\bar r}(\mathcal{T}_n)
\subset \{\dist(\cdot,\mathcal{I}_i) > \delta\bar r\}$
because of~\eqref{eq:supportPropertyByLocalization1}, and nothing has to be proven in terms of Property~\textit{iv)}.
Second, if $\mathcal{T}_c\subset I_{i,j}$ it follows from~\eqref{eq:representationCalibrationVectors1b}
and the definition~\eqref{eq:localizationStraightLineSegment} of the localization function~$\eta_c$
that Property~\textit{iv)} holds true in $\bar\Omega_j \cap \{\dist(\cdot,\mathcal{I}_i) \leq \delta_3(\sigma)\bar r\}
\cap \Psi_{\mathcal{L}_c}\big(\mathcal{T}_c {\times} (-\delta\bar r,\delta\bar r)\big)
\setminus\bigcup_{n\in\mathcal{P}\cup\mathcal{B}} B_{\bar r}(\mathcal{T}_n)$.

Fix next a boundary endpoint $b\in\mathcal{B}$. By the properties
of the admissible pair of localization scales $(\bar r,\delta)$ from Definition~\ref{def:locRadius}
and the definition~\eqref{eq:defLocalCalibrationBoundary}, the same
argument as in the preceding discussion applies to the convex
region $\overline{D} \cap B_{\bar r}(\mathcal{T}_b)$.

We finally consider a triple junction $p\in\mathcal{P}$ and
convince ourselves of the validity of the properties~\textit{ii)}--\textit{iv)}
in the region~$\overline{D} \cap B_{\bar r}(\mathcal{T}_p)$.
Based on the decomposition~\eqref{eq:decompByWedges} 
of $B_{\bar r}(\mathcal{T}_p)$ into wedges $(W^p_c)_{c\in\mathcal{C},c\sim p}$,
let us further fix a straight line segment $c\in\mathcal{C}$ with $c\sim p$.
Plugging in the definition~\eqref{eq:defLocalCalibrationTripleJunction} we obtain
\begin{align}
\label{eq:representationCalibrationVectors2}
\xi_i {-} \xi_j = \eta_c \lambda^p_c (\vv{\xi^p_i} {-} \vv{\xi^p_j})
								  + \eta_c (1 {-} \lambda^p_c) (\vv{\xi^c_i} {-} \vv{\xi^c_j})
\quad\text{in } \overline{D} \cap B_{\bar r}(\mathcal{T}_p) \cap \overline{W^{p}_c}.
\end{align}
In particular, it follows based on~\eqref{eq:calibrationProperty} and~\eqref{Length} that, in $\overline{D} \cap B_{\bar r}(\mathcal{T}_p) \cap \overline{W^{p}_c}$,
\begin{equation}
\label{eq:representationCalibrationVectors2b}
\begin{aligned}
\xi_i-\xi_j &= \eta_c\sigma_{i,j}\bar{\vec{n}}_{i,j} && \text{if } \mathcal{T}_c \subset I_{i,j}
\,(\text{and thus } \mathcal{T}_p \subset I_{i,j}), \\
|\xi_i - \xi_j| &\leq \eta_c\lambda^p_c\sigma_{i,j} + \eta_c(1{-}\lambda^p_c)\delta_1(\sigma)\sigma_{i,j} 
&& \text{if } \mathcal{T}_c \not\subset I_{i,j}
\text{ but } \mathcal{T}_p \subset I_{i,j}, \\
|\xi_i - \xi_j| &\leq \eta_c\delta_1(\sigma)\sigma_{i,j} && \text{if } \mathcal{T}_p \not\subset I_{i,j}
\,(\text{and thus } \mathcal{T}_c \not\subset I_{i,j}).
\end{aligned}
\end{equation}
In combination with the properties of the admissible pair of localization scales 
$(\bar r,\delta)$ from Definition~\ref{def:locRadius}, the definition~\eqref{eq:localizationStraightLineSegment} 
of the cutoffs~$\eta_c$, and the definition~\eqref{eq:localizationTripleJunction} of the weights~$\lambda^p_c$,
the estimates from~\eqref{eq:representationCalibrationVectors2b} in turn 
immediately imply Property~\textit{ii)} in the region
$\overline{D} \cap B_{\bar r}(\mathcal{T}_p) \cap \overline{W^{p}_c}$.

For a proof of Property~\textit{iii)} within $\overline{D} \cap B_{\bar r}(\mathcal{T}_p) \cap \overline{W^{p}_c}$,
we may assume without loss of generality that $\mathcal{T}_c \not \subset I_{i,j}$ but 
$\mathcal{T}_p \subset I_{i,j}$. In the other two cases, either the first or the third line
of~\eqref{eq:representationCalibrationVectors2b} apply and the argument is the same as before.
Consider now a point $x \in \overline{D} \cap B_{\bar r}(\mathcal{T}_p) \cap \overline{W^{p}_c}$
such that $|\frac{\xi_i(x)-\xi_j(x)}{\sigma_{i,j}}|>1-\delta_2(\sigma,\kappa)$.
If $\lambda^p_c(x) \in [0,1{-}\frac{\kappa}{8})$, then by the assumption on~$x$ and 
the second line of~\eqref{eq:representationCalibrationVectors2b} it follows
$1{-}\delta_2(\sigma,\kappa) < \delta_1(\sigma) + \lambda^p_c(x)(1{-}\delta_1(\sigma))
< (1{-}\frac{\kappa}{8}) + \frac{\kappa}{8}\delta_1(\sigma)$. Hence, in view of $\delta_1(\sigma)\in (0,1)$
it is just a matter of choosing $\delta_2(\sigma,\kappa)$ small enough to obtain a contradiction. 
In particular, we are safe to assume that $\lambda^p_c(x)\in (1{-}\frac{\kappa}{8},1]$.
The assumption on~$x$ and the second line of~\eqref{eq:representationCalibrationVectors2b} then
trivially imply $1 {-} \eta_c(x) \leq \delta_2(\sigma,\kappa)$. 
Restricting if needed $\delta_2(\sigma,\kappa)\in (0,\frac{\kappa}{4})$
and noting that $\vv{\xi^p_i} {-} \vv{\xi^p_j} = \sigma_{i,j}\bar{\vec{n}}_{i,j}$
due to~\eqref{eq:calibrationProperty} and the assumption $\mathcal{T}_p\subset I_{i,j}$, 
it follows from~\eqref{eq:representationCalibrationVectors2} and~\eqref{Length} that
\begin{align*}
\Big|\frac{\xi_i(x) {-} \xi_j(x)}{\sigma_{i,j}} - \bar{\vec{n}}_{i,j}\Big|
&\leq \big(1{-}\eta_c(x)\big)\lambda^p_c(x) + 
\big(1{-}\lambda^p_c(x)\big)\bigg|\eta_c(x)\frac{\vv{\xi^c_i} {-} \vv{\xi^c_j}}{\sigma_{i,j}} - \bar{\vec{n}}_{i,j}\bigg|
\\&
\leq \frac{\kappa}{4} + \frac{\kappa}{4} = \frac{\kappa}{2}.
\end{align*}
This proves Property~\textit{iii)} in the region
$\overline{D} \cap B_{\bar r}(\mathcal{T}_p) \cap \overline{W^{p}_c}$. 

If $\mathcal{T}_c \subset I_{i,j}$ (and thus also $\mathcal{T}_p \subset I_{i,j}$),
the representation of $\xi_i - \xi_j$ from the first line of~\eqref{eq:representationCalibrationVectors2b}
allows to employ the same argument as above for a proof of Property~\textit{iv)}
in $\overline{D} \cap B_{\bar r}(\mathcal{T}_p) \cap \overline{W^{p}_c}$. 
If $\mathcal{T}_p \not\subset I_{i,j}$ (and thus also $\mathcal{T}_c \not\subset I_{i,j}$),
then necessarily $\bar\Omega_j \cap B_{\bar r}(\mathcal{T}_p) \cap \{\dist(\cdot, \mathcal{I}_i) < \bar r\} = \emptyset$
by the properties of the localization scale~$\bar r$ from Definition~\ref{def:locRadius}.
In particular, in this regime nothing has to be proven with respect to Property~\textit{iv)}.
Consider finally the regime $\mathcal{T}_p \subset I_{i,j}$ and $\mathcal{T}_c \not\subset I_{i,j}$.
We choose $\delta_3(\sigma)\in (0,\frac{\delta'}{4})$ small enough such that 
\begin{align}
\bar\Omega_j \cap B_{\bar r}(\mathcal{T}_p) \cap \overline{W^p_c}
\cap \{\dist(\cdot, \mathcal{I}_i) < \delta_3(\sigma)\bar r\} 
\cap \supp \big(\eta_c(1{-}\lambda^p_c)\big) = \emptyset.
\end{align}
This is always possible by the choice of the profile $\theta$ in \textit{Step 6} of this
proof and the definition~\eqref{eq:localizationTripleJunction} of the weight $\lambda^p_c$.
In particular, based on~\eqref{eq:representationCalibrationVectors2} we then have
\begin{align*}
\xi_i-\xi_j = \eta_c\sigma_{i,j}\bar{\vec{n}}_{i,j}
\end{align*}
in $\bar\Omega_j \cap B_{\bar r}(\mathcal{T}_p) \cap \overline{W^p_c}
\cap \{\dist(\cdot, \mathcal{I}_i) < \delta_3(\sigma)\bar r\}$,
so that in this region Property~\textit{iv)} holds true based on the same arguments as above.

This in turn eventually concludes the proof of Lemma~\ref{LemmaExistenceLocalCalibration}. \qed

\section{Flat partitions are local minimizers: Proof of Theorem~\ref{MainResult}}
Since the partitions $\chi$ and $\bar\chi$ have coinciding traces on $\partial D$,
the identity in \eqref{EnergyEstimate} simply reduces to
\begin{align}
\label{EnergyEstimateSameDirichletData}
E[\chi] =
E[\bar\chi] + E[\chi|\bar\chi]
+ \sum_{i,j=1,i\neq j}^P \int_{D} 2(\chi_i{-}\bar\chi_i) \bar\chi_j (\nabla\cdot\xi_i{-}\nabla\cdot\xi_j) \dx,
\end{align}
where we take the vector fields $(\xi_i)_{i\in\{1,\ldots,P\}}$ to be a local paired calibration in the sense of
Definition~\ref{DefinitionLocalCalibration}. Because of Lemma~\ref{LemmaExistenceLocalCalibration}, given
a regular flat partition $\bar\chi$ such a family of vector fields always exists, provided that the associated scales $(\bar r,\delta) \in (0,1]{\times}(0,\frac{1}{2}]$ 
are admissible in the sense of Definition~\ref{def:locRadius} for the regular flat partition $\bar\chi$ and that $B_{2\bar r}(\mathcal{T}_b) \cap D$ is convex for all $b\in \mathcal{B}$. The latter assumption will turn out to be of a technical nature, we will remove it in the proof of Theorem~\ref{MainResult} proper.

The idea now is to make use of the coercivity properties of the functional $E[\chi|\bar\chi]$ together with the
fact that the family of Lipschitz vector fields $(\xi_i)_{i\in\{1,\ldots,P\}}$ is a local paired calibration 
with respect to $(\bar r,\delta)$ in the precise sense of 
Definition~\ref{DefinitionLocalCalibration}. We aim to 
prove---at least for sufficiently small perturbations of the phases in $L^1$---that the last two terms in 
\eqref{EnergyEstimateSameDirichletData} combine to a non-negative quantity. To this end, we will proceed in several lemmas.

First, we want to exploit that the relative energy~$E[\chi|\bar\chi]$ gives control over 
the length of the interfaces of the partition~$\chi$ whenever the vector~$\frac{\xi_i-\xi_j}{\sigma_{i,j}}$
is short. For a local paired calibration this happens to be the case outside of a 
strip around the boundaries 
\begin{align}
\label{eq:defGrainBoundary}
\mathcal{I}_i:=\bigcup_{j=1}^PI_{i,j},\quad i\in\{1,\ldots,P\}.
\end{align} 
The contribution of the bulk term in \eqref{EnergyEstimateSameDirichletData} in that region may then be absorbed into 
the relative energy $E[\chi|\bar\chi]$ by an application of the relative isoperimetric inequality.

\begin{lemma}\label{lemma:control_shortness_constraint}
There exists $\bar s>0$ sufficiently small such that if $\chi = (\chi_1,\ldots, \chi_P)$ is a partition of D with finite surface energy satisfying $\max_{i=1,\ldots P} \| \chi_i - \bar \chi_i \|_{L^1(D)} \leq \bar s $ and trace coinciding with $\bar \chi$ we have
\begin{align}\label{ProofTheoremAux1}
E[\chi] &\geq E[\bar\chi] + \frac{1}{2} E[\chi|\bar\chi]
\\&~~~\nonumber
+\sum_{i,j=1,i\neq j}^P \int_{U^{\mathcal{I}_i}_{(\frac{\bar r}{2},\delta)}} 
2(\chi_i{-}\bar\chi_i) \bar\chi_j (\nabla\cdot\xi_i{-}\nabla\cdot\xi_j) \dx.
\end{align}
\end{lemma}

\begin{proof}
Let
	\begin{align}
\label{eq:farAway}
U_i^{\mathrm{bulk}} := D \setminus U^{\mathcal{I}_i}_{(\frac{\bar r}{2},\delta)},
\quad i\in\{1,\ldots,P\}.
\end{align}
Note first that $|\nabla(\chi_i{-}\bar\chi_i)|(\partial D)=0$ 
for all $i\in\{1,\ldots,P\}$ since by assumption $\chi_i$ 
and $\bar\chi_i$ have coinciding traces on~$\partial D$.
We then have $|\nabla(\chi_i{-}\bar\chi_i)|\mres U_i^{\mathrm{bulk}} = |\nabla\chi_i|\mres U_i^{\mathrm{bulk}}$.
Now, let $\sigma_{\mathrm{min}}:=\min_{i<j}\sigma_{i,j}>0$ and choose the scale $\bar s \in (0,1)$ small
enough such that 
\begin{align}
\label{smallnessCond1}
\bar s \leq \gamma(U_i^{\mathrm{bulk}})
\quad\text{for all phases } i\in\{1,\ldots,P\},
\end{align} 
where $\gamma(U_i^{\mathrm{bulk}})$ is the constant from the relative isoperimetric inequality, 
Theorem~\ref{TheoremRelIsoperimetricInequ}. 
Because of property~\textit{iii)} from Definition~\ref{def:locRadius} of an admissible localization radius,
we may indeed apply the relative isoperimetric inequality \eqref{RelIsoperimetricInequ} 
with the choices $\Omega = U_i^{\mathrm{bulk}}$ and 
$G=\supp (\chi_i{-}\bar\chi_i)\cap U_i^{\mathrm{bulk}}$ for all $i\in\{1,\ldots,P\}$ 
to obtain the bound
\begin{align*}
&\sum_{j=1,j\neq i}^P\int_{U_i^{\mathrm{bulk}}} 2(\chi_i{-}\bar\chi_i) 
\bar\chi_j (\nabla\cdot\xi_i{-}\nabla\cdot\xi_j) \dx
\\&
\geq -C_1\max_{i=1,\ldots,P} \|\xi_i\|_{W^{1,\infty}(D)}
\bar s^\frac{1}{2} \int_{U_i^{\mathrm{bulk}}} 1\,\mathrm{d}|\nabla\chi_i|
\\&
\geq -\frac{C_1}{\sigma_{\mathrm{min}}(1{-}\delta_1)}
\max_{i=1,\ldots,P} \|\xi_i\|_{W^{1,\infty}(D)} \bar s^\frac{1}{2}
\sum_{j=1,j\neq i}^P\sigma_{i,j}\int_{S_{i,j}\cap U_i^{\mathrm{bulk}}}
1-\frac{\xi_i{-}\xi_j}{\sigma_{i,j}}\cdot\vec{n}_{i,j}\,\mathrm{d}\mathcal{H}^1,
\end{align*}
where in the last step we used that 
$U_i^{\mathrm{bulk}} \subset \overline{D} \setminus U^{I_{i,j}}_{(\bar r,\delta)}$ 
for all distinct $i,j\in\{1,\ldots,P\}$ together with property~\textit{ii)} 
from Definition~\ref{DefinitionLocalCalibration} of a local paired calibration. 
Hence, by choosing $\bar s \in (0,1)$ sufficiently small such that
\begin{align} 
\label{smallnessCond2}
\frac{C_1}{\sigma_{\mathrm{min}}(1{-}\delta_1)}
\max_{i=1,\ldots,P} \|\xi_i\|_{W^{1,\infty}(D)} 
\bar s^\frac{1}{2} \leq \frac{1}{2},
\end{align}	
we deduce the claim, provided that the two smallness conditions~\eqref{smallnessCond1} and~\eqref{smallnessCond2} hold true.
\end{proof}

Therefore, we only have to consider the contributions in the last term of the estimate~\eqref{ProofTheoremAux1} over regions close to the corresponding boundaries.
The second lemma exploits the fact that the relative energy $E[\chi|\bar\chi]$ not only directly
penalizes the part of the interfaces of the partition~$\chi$ where the vector~$\frac{\xi_i-\xi_j}{\sigma_{i,j}}$
is short but also the part where this vector points into a direction sufficiently different from the one of 
the unit normal vector~$\vec{n}_{i,j}$.
If this happens ``too often'', we are able to prove the desired inequality $E[\chi]\geq E[\bar \chi]$ right away.

\begin{lemma}\label{lemma:tilt_excess_control}
	Under the assumptions of Lemma \ref{lemma:control_shortness_constraint}, let there exist $\eps_1, \eps_2 \in (0,1)$ such that we have
	\begin{align}
\label{AssumptionStep2}
\sum_{i,j=1,i\neq j}^P\sigma_{i,j}
\mathcal{H}^1\big( \big\{x\in S_{i,j} \colon \frac{\xi_i{-}\xi_j}{\sigma_{i,j}}
\cdot \vec{n}_{i,j}\leq 1{-}\varepsilon_1\big\}\big)
\geq \varepsilon_2 > 0.
\end{align}
	Then we can choose $\bar s> 0$ small enough depending on $\eps_1$ and $\eps_2$ such that we have $E[\chi]\geq E[\bar \chi]$.
\end{lemma}

\begin{proof}
We estimate
\begin{align*}
&\sum_{i,j=1,i\neq j}^P \int_{U^{\mathcal{I}_i}_{(\bar r,\delta)}} 
2(\chi_i{-}\bar\chi_i) \bar\chi_j (\nabla\cdot\xi_i{-}\nabla\cdot\xi_j) \dx
\\&
\geq -C_2\max_{i=1,\ldots,P}
\|\xi_i\|_{W^{1,\infty}(D)} \bar s
\\&
\geq -\frac{C_2\sigma_{\mathrm{\min}}^{-1}}{\varepsilon_1\varepsilon_2}
\max_{i=1,\ldots,P} \|\xi_i\|_{W^{1,\infty}(D)} \bar s
\sum_{i,j=1,i\neq j}^P\sigma_{i,j}
\int_{\{x\in S_{i,j}\colon \frac{\xi_i{-}\xi_j}{\sigma_{i,j}}\cdot\vec{n}_{i,j}\leq 1{-}\varepsilon_1\}} 
\varepsilon_1 \,\mathrm{d}\mathcal{H}^1.
\end{align*}
Hence, choosing $\bar s \in (0,1)$ sufficiently small in the form of
\begin{align}
\label{smallnessCond3}
\frac{C_2\sigma_{\mathrm{\min}}^{-1}}{\varepsilon_1\varepsilon_2}
\max_{i=1,\ldots,P}\|\xi_i\|_{W^{1,\infty}(D)} \bar s 
\leq \frac{1}{4},
\end{align}
we obtain under the assumption of~\eqref{AssumptionStep2} and the smallness condition~\eqref{smallnessCond3}
\begin{equation}
\begin{aligned}
&\sum_{i,j=1,i\neq j}^P \int_{U^{\mathcal{I}_i}_{(\bar r,\delta)}} 
2(\chi_i{-}\bar\chi_i) \bar\chi_j (\nabla\cdot\xi_i{-}\nabla\cdot\xi_j) \dx
\\&\label{ProofTheoremAux2}
\geq -\frac{1}{4}\sum_{i,j=1,i\neq j}^P\sigma_{i,j}
\int_{\{x\in S_{i,j}\colon \frac{\xi_i{-}\xi_j}{\sigma_{i,j}}\cdot\vec{n}_{i,j}\leq 1{-}\varepsilon_1\}} 
1-\frac{\xi_i{-}\xi_j}{\sigma_{i,j}}\cdot\vec{n}_{i,j} \,\mathrm{d}\mathcal{H}^1.
\end{aligned}
\end{equation}
Combining this with the result of Lemma \ref{lemma:control_shortness_constraint}, we get $E[\chi]\geq E[\bar \chi]$.
\end{proof}

We are thus able to later \emph{choose} constants $\eps_1$ and $\eps_2$ to our liking and assume that \begin{align}
\sum_{i,j=1,i\neq j}^P \mathcal{H}^1\big(\big\{x\in S_{i,j}\colon \frac{\xi_i{-}\xi_j}{\sigma_{i,j}}
\cdot \vec{n}_{i,j}\leq 1{-}\varepsilon_1\big\}\big)
\leq \varepsilon_2,
\end{align}
i.e., that the interfaces of the competitors are not too ``steep'' with respect to those of the putative local minimizer.
Indeed, the rest of the argument will distinguish between such points on the inferfaces of competitors, those where the interface is essentially a somewhat flat graph over the regular partition, and the residual points where the interface is flat, but not a graph.

As a next step, we will work towards a definition of the set of ``flat graph'' points.
This will furthermore involve a slicing condition, which will need to be compatible with multiple different normals due to the presence of triple points.
The following lemma provides the appropriate geometric setup.

For any $i \in \{1,\ldots,P\}$, let $\mathcal{N}_i \subset \{1,\ldots,N\}$
denote the subset of all topological features at which phase~$i$ is present.
Conversely, given topological feature $n\in\mathcal{N}$, we will denote 
by $\mathcal{K}_n \subset \{1,\ldots,P\}$ those indices $k$ for which phase $k$ is present at~$\mathcal{T}_n$.

\begin{lemma}\label{lemma:geometry}
	 Let $i \in \{1,\ldots,P\}$ and $n \in \mathcal{N}_i$.
	 Then there exists $\nu_{n,i} \in \mathbb{S}^1$ such that for all $j\in \mathcal{K}_n \setminus \{i\}$ we have
\begin{align}\label{nu_p_pointing_in_strong_normal}
	\nu_{n,i} \cdot \bar{n}_{i,j} >0.
\end{align}
	For $n \in \mathcal{T}_c$, we have $\nu_{n,i} = \bar{n}_{i,j}$.
	
	Furthermore, under the assumptions of Lemma \ref{lemma:control_shortness_constraint} and for all sufficiently small $\eps_1$ depending only on $\sigma$ and $\bar \chi$ we have the following:
	For all $ x\in S_{i,j} \cap Q_n$ with
	\begin{align*}
		\Big(\frac{\xi_i {-} \xi_j}{\sigma_{i,j}}\cdot\vec{n}_{i,j}\Big)(x) > 1-\varepsilon_1
	\end{align*}
	we have
\begin{equation}
\begin{aligned}
\label{eq:coareaFactor}
 \nu_{n,i}\cdot\vec{n}_{i,j}(x) \geq c > 0,
\end{aligned}
\end{equation}
while for all $x\in  S_{l,m} \cap Q_{n}$ with  $l,m\in\{1,\ldots,P\},\,l\neq m,\, l \notin \mathcal{K}_n$ it holds
\begin{equation}
\begin{aligned}
 \label{eq:shortness}
\Big(\frac{\xi_l {-} \xi_m}{\sigma_{l,m}} \cdot \vec{n}_{l,m}\Big)(x) \leq 1 - \varepsilon_1.
\end{aligned}
\end{equation}
\end{lemma}

\begin{proof}
\textit{Step 1: Construction of $\nu_{n,i}$ such that property \eqref{nu_p_pointing_in_strong_normal} holds.}
For $c\in\mathcal{N}_i\cap\mathcal{C}$, we define $\nu_{c,i} := \bar{\vec{n}}_{i,j}|_{\mathcal{T}_c}$.

For $b \in \mathcal{N}_i\cap\mathcal{B}$ we consider the unique $j \in \mathcal{K}_b\setminus\{i\}$.
We set $\nu_{b,i} := \tau_{\partial D}(\mathcal{T}_b)$
where $\tau_{\partial D}(\mathcal{T}_b) \in \mathbb{S}^1$
is the unique unit-length tangent vector of $\partial D$ at $\mathcal{T}_b$
such that $\bar{\vec{n}}_{i,j}(\mathcal{T}_b) \cdot \tau_{\partial D}(\mathcal{T}_b) > 0$, 
see Figure~\ref{fig:slicing_normal_boundary}.

Consider finally the case of a triple junction $p \in \mathcal{N}_i\cap\mathcal{P}$.
Let $j,l \in \mathcal{K}_p\setminus\{i\}$ be the two unique distinct phases being also
present at~$\mathcal{T}_p$.
Denote by $\nu_{p,i} \in \mathbb{S}^1$ the unique 
unit-length vector bisecting the cone with interior normals given by $\bar{\vec{n}}_{i,j}|_{\mathcal{T}_p}$
and $\bar{\vec{n}}_{i,l}|_{\mathcal{T}_p}$, see Figure~\ref{fig:slicing_normal_triple}.
By the strict triangle inquality \eqref{TriangleInequalitySurfaceTensions} satisfied by $\sigma$, the opening angle of its completent is strictly positive.
Since $\nu_{p,i}$ halves the angle between $\bar{\vec{n}}_{i,j}|_{\mathcal{T}_p}$
and $\bar{\vec{n}}_{i,l}|_{\mathcal{T}_p}$, we therefore also in this case have $\nu_{p,i}  \cdot \bar{\vec{n}}_{i,j}|_{\mathcal{T}_p}>0$ and $\nu_{p,i}  \cdot \bar{\vec{n}}_{i,l}|_{\mathcal{T}_p}>0$.

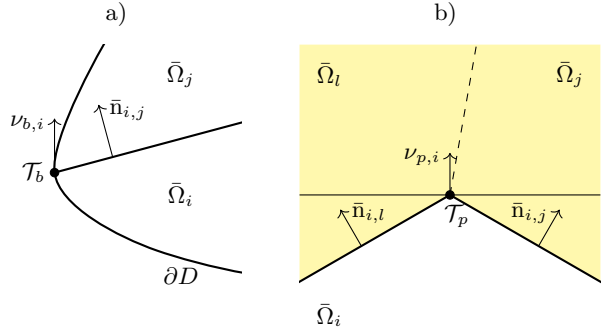
\begin{figure}[t]
\centering
\subcaptionbox{\label{fig:slicing_normal_boundary}}{
\centering
\begin{tikzpicture}[scale=.8]
	\clip (-4,-3.2)--(-4,1) -- (0,1) -- (0,-3.2) -- (-4,-3.2);
	\draw[thick] plot [smooth cycle] coordinates {(0:3)(40:4)(90:3.5)(140:2.5)(200:3.3)(240:2.8)(300:3.2)};
	\draw[fill] (200:3.3) circle (2pt) node[left] {\small$\;\;\mathcal{T}_b$};
	\draw[<-] ($(200:3.3) + (0,.9)$) -- (200:3.3) node[left, pos=0.1] {\small$\nu_{b,i}$};
	\draw[thick] (200:3.3) -- ($(200:3.3) + (15:4)$);
	\draw (-1,-1.5) node {\small$\bar\Omega_i$};
	\draw (-1,.5) node {\small$\bar\Omega_j$};
	\draw (-1,-2.85) node {\small$\partial D$};
	\draw[<-]  ($(200:3.3) + (15:1) + (105:.9)$) --($(200:3.3) + (15:1) $) node[right, pos=0.1] {\small$\bar{\vec{n}}_{i,j}$};
\end{tikzpicture}
}
\subcaptionbox{\label{fig:slicing_normal_triple}}{
\centering
\begin{tikzpicture}[scale=.8]
\draw[name path = interfaceA, draw = none, rotate = -30] (0,0) -- (3.5,0);
\draw[name path = interfaceB, draw = none, rotate = 210] (0,0) -- (3.5,0);
\draw[name path = box, draw = none] (-2.5,-2.5) rectangle (2.5,2.5); 
\path[name intersections={of=interfaceA and box, by={X}}];
\path[name intersections={of=interfaceB and box, by={Y}}];
\fill[yellow!40] (0,0) -- (X) -- (2.5,2.5) -- (-2.5,2.5) -- (Y) -- cycle;
\clip(-2.5,-2.5) rectangle (2.5,2.5); 
\draw[fill] (0,0) circle (2pt);
\draw[thick, rotate = -30] (0,0) -- (3.5,0) node[below, pos=0] {\small$\;\;\mathcal{T}_p$};
\draw[dashed, rotate = 80] (0,0) -- (3.2,0);
\draw[thick, rotate = 210] (0,0) -- (3.7,0);
\draw (-2,-2) node {\small$\bar\Omega_i$};
\draw (2,2) node {\small$\bar\Omega_j$};
\draw (-2,2) node {\small$\bar\Omega_l$};
\draw (-2.5,0) -- (2.5,0);
\draw[<-] (0,.7) -- (0,0) node[left, pos=0.1] {\small$\nu_{p,i}$};
\draw[<-] ($(210:1.7) + (120:0.7)$) -- (210:1.7) node[right, pos=0.1] {\small$\bar{\vec{n}}_{i,l}$};
\draw[<-] ($(-30:1.7) + (60:0.7)$) -- (-30:1.7) node[left, pos=0.1] {\small$\bar{\vec{n}}_{i,j}$};
\end{tikzpicture}
}
\caption{An illustration of the slicing normals in the cases a) 
a boundary point $b$ and b) a triple junction $p$. The 
yellow region is the cone with interior normals given by $\bar{\vec{n}}_{i,j}|_{\mathcal{T}_p}$
and $\bar{\vec{n}}_{i,l}|_{\mathcal{T}_p}$, 
coinciding locally with $\bar \Omega_l \cup \bar \Omega_j$.}
\label{fig:slicing_normals}
\end{figure}

\textit{Step 2: Choice of $\eps_1$.}
By property \eqref{nu_p_pointing_in_strong_normal}, we may choose constants $c,c' \in (0,1)$ such that for any $v \in \mathbb{S}^1$
the condition $\bar{\vec{n}}_{i,j}|_{\mathcal{T}_n} \cdot v \geq 1 - c'$
implies $\nu_{n,i} \cdot v \geq c$.
Then, we choose $\varepsilon_1$ sufficiently small
in the sense that $\varepsilon_1 < \min\{\smash{\frac{c'}{2}},\delta_2(\sigma,\smash{\frac{c'}{2}})\}$,
where $\delta_2(\sigma,\smash{\frac{c'}{2}})$ is the constant from property~\textit{iii)}
of Definition~\ref{DefinitionLocalCalibration} of a local paired calibration.
Hence, observing $\bar{\vec{n}}_{i,j}|_{\mathcal{T}_n} \cdot v \geq 
-|\smash{\frac{\xi_{i} {-} \xi_{j}}{\sigma_{i,j}}} {-} \bar{\vec{n}}_{i,j}|_{\mathcal{T}_n}|
+ \smash{\frac{\xi_{i} {-} \xi_{j}}{\sigma_{i,j}}} \cdot v$, it follows
from the above choices of the constants $c,c'$ and $\varepsilon$ as well
as from property~\textit{iii)} of Definition~\ref{DefinitionLocalCalibration} 
of a local paired calibration that
\begin{align*}
&\forall x\in S_{i,j} \cap Q_n\colon
\Big(\frac{\xi_i {-} \xi_j}{\sigma_{i,j}}\cdot\vec{n}_{i,j}\Big)(x) > 1-\varepsilon_1
\quad\Longrightarrow\quad \nu_{n,i}\cdot\vec{n}_{i,j}(x) \geq c > 0.
\end{align*}

By possibly reducing $\varepsilon_1 \in (0,1)$ even further
(this time depending only on the constant $\delta_1$ from Definition~\ref{DefinitionLocalCalibration}
of a local paired calibration), we may on top guarantee
\begin{align*}
&\forall x\in  S_{l,m} \cap Q_{n}, 
\quad l,m\in\{1,\ldots,P\},\,l\neq m,\, l \notin \mathcal{K}_n
\\&
\quad\Longrightarrow\quad
\Big(\frac{\xi_l {-} \xi_m}{\sigma_{l,m}} \cdot \vec{n}_{l,m}\Big)(x) \leq 1 - \varepsilon_1
\end{align*}
due to~\eqref{eq:auxPropertyCubes100}, the localization
properties of an admissible pair of localization scales in the sense
of Definition~\ref{def:locRadius}, as well as property~\textit{ii)}
from Definition~\ref{DefinitionLocalCalibration} of a local paired calibration.
\end{proof}

For each $i\in \{1,\ldots,P\}$ and each $n\in \mathcal{N}_i$ we are now in a position to define the change of variables $\Psi_{n,i} : \mathcal{L}_{n,i} \times \mathbb{R} \to \mathbb{R}^2$ via
\begin{align}
	\Psi_{n,i}(x',y) := x'  + y \nu_{n,i},
\end{align}
where $\mathcal{L}_{n,i}$ is the unique line normal to $\nu_{n,i}$ and passing through $\mathcal{T}_n$.
By $\mathcal{P}_{n,i}$ we denote the orthogonal projection onto $\mathcal{L}_{n,i}$.

For each $n \in \mathcal{N}_i$, we further define a set of auxiliary neighborhoods 
on the localization scales $(\bar r,\delta)$ by means of
\begin{align}
\label{eq:defNeighborhoodTopFeatures}
U_{n} := 
\begin{cases}
B_{\bar r}(\mathcal{T}_n) \cap D
& \text{if } n \in \mathcal{N}_i \cap (\mathcal{P} \cup \mathcal{B}), \\
B_{\delta\bar r}(\mathcal{T}_n) \setminus \bigcup_{n' \in \mathcal{P} \cup \mathcal{B} : \,n' \sim n} 
B_{\bar r}(\mathcal{T}_{n'}) & \text{if } n \in \mathcal{N}_i \cap \mathcal{C}.
\end{cases}
\end{align}
and
\begin{align}
\label{eq:defNeighborhoodCuboid}
Q_{n,i} := 
\begin{cases}
\Psi_{n,i}\big((B_{\sqrt{2}\bar r}(\mathcal{T}_n) \cap \mathcal{L}_{n,i})
{\times}(-\sqrt{2}\bar r,\sqrt{2}\bar r)\big)
& \text{if } n \in \mathcal{N}_i \cap ( \mathcal{P} \cup \mathcal{B} ), \\
\Psi_{n,i}\big(P_{\mathcal{L}_{n,i}}(U_{n})
{\times}(-\delta\bar r,\delta\bar r)\big)
& \text{if } n \in \mathcal{N}_i \cap \mathcal{C}.
\end{cases}
\end{align}
We refer to Figure~\ref{fig:nbhdSlices} for a visualization 
of the neighborhoods~$U_n$ and~$Q_n$,
respectively.
Note that this allows to represent~\eqref{eq:dumbbellNbhdPhase} in the form of
\begin{align}
\label{eq:decompNeighborhoodGrainBoundaryByTopFeatures}
U^{\mathcal{I}_i}_{(\bar r,\delta)} \cap D  = \bigcup_{n\in\mathcal{N}_{i}} U_n.
\end{align}
Furthermore, we have $U_{n}\subset Q_{n,i}$ and, due to $\delta\in (0,\frac{1}{2}]$, that
\begin{align}
\label{eq:auxPropertyCubes100}
Q_{c,i} \cap B_{\frac{\bar{r}}{2}}(\mathcal{T}_p) = \emptyset
\quad\text{if } c\in\mathcal{N}_i\cap\mathcal{C},\,c\sim p\in\mathcal{P}.
\end{align}

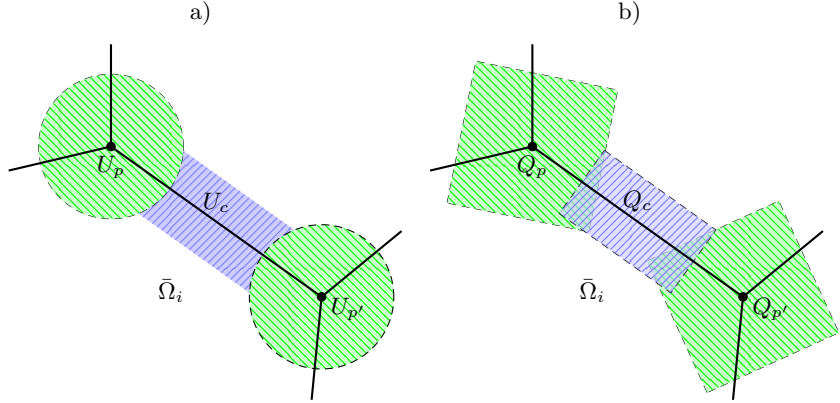
\begin{figure}[tbhp]
\centering
\subcaptionbox{\label{fig:nbhdSlicesA}}{
\centering
\begin{tikzpicture}[scale=.8]
\draw[name path = circle1, draw = none] (-3.5,2.5) circle (1.2cm);	
\draw[name path = circle2, draw = none] (0,0) circle (1.2cm);
\draw[draw=none] (-3.5,2.5) -- node[sloped,inner sep=0cm,above,pos=0,anchor=south west,
																minimum height=.5cm,minimum width=.5cm]	(NW) {}
													 node[sloped,inner sep=0cm,above,pos=1,anchor=south west,
																minimum height=.5cm,minimum width=.5cm]	(NE) {}
													 node[sloped,inner sep=0cm,below,pos=0,anchor=north west,
																minimum height=.5cm,minimum width=.5cm]	(SW) {}
													 node[sloped,inner sep=0cm,below,pos=1,anchor=north west,
																minimum height=.5cm,minimum width=.5cm]	(SE) {} (0,0);
\fill[blue!20] 
			(NW.north west) -- (NE.north west) -- (SE.south west) -- (SW.south west) -- cycle;
\fill[pattern = north east lines, pattern color = blue!50] 
			(NW.north west) -- (NE.north west) -- (SE.south west) -- (SW.south west) -- cycle;	
\fill[green!20] (0,0) circle (1.2cm);
\fill[pattern=north west lines,pattern color = green] (0,0) circle (1.2cm);
\draw[densely dashed] (0,0) circle (1.2cm);
\draw[fill] (0,0) circle (2pt);
\draw (0,-0.2) node[right] {\small$U_{p'}$};
\draw[densely dashed] (-3.5,2.5) circle (1.2cm);
\fill[green!20] (-3.5,2.5) circle (1.2cm);
\fill[pattern=north west lines,pattern color = green] (-3.5,2.5) circle (1.2cm);
\draw[fill] (-3.5,2.5) circle (2pt);
\draw (-3.5,2.5) node[below] {\small$U_{p}$};
\draw[thick] (-3.5,2.5) -- node[midway, above] {\small$U_c$} (0,0);
\draw[thick,rotate=120,scale=.4] (0,0) -- (-3.5,2.5);
\draw[thick,rotate=255,scale=.4] (0,0) -- (-3.5,2.5);
\draw[thick] (-3.5,2.5) -- (-3.5,4.2);
\draw[thick] (-3.5,2.5) -- (-5.2,2.1);
\draw (-2.5,0.1) node {\small$\bar\Omega_i$};
\end{tikzpicture}
}
\subcaptionbox{\label{fig:nbhdSlicesB}}{
\centering
\begin{tikzpicture}[scale=.8]
\draw[name path = circle1, draw = none] (-3.5,2.5) circle (1.2cm);	
\draw[name path = circle2, draw = none] (0,0) circle (1.2cm);
\draw[draw=none] (-3.5,2.5) -- node[sloped,inner sep=0cm,above,pos=0,anchor=south west,
																minimum height=.5cm,minimum width=.5cm]	(NW) {}
													 node[sloped,inner sep=0cm,above,pos=1,anchor=south west,
																minimum height=.5cm,minimum width=.5cm]	(NE) {}
													 node[sloped,inner sep=0cm,below,pos=0,anchor=north west,
																minimum height=.5cm,minimum width=.5cm]	(SW) {}
													 node[sloped,inner sep=0cm,below,pos=1,anchor=north west,
																minimum height=.5cm,minimum width=.5cm]	(SE) {} (0,0);
\draw[name path = auxRectangle, draw = none] 
			(NW.north west) -- (NE.north west) -- (SE.south west) -- (SW.south west) -- cycle;
\draw[densely dashed, rotate=24.5] (-1.2,-1.2) rectangle (1.2,1.2);
\fill[green!20,rotate=24.5] (-1.2,-1.2) rectangle (1.2,1.2);
\fill[pattern=north west lines,pattern color = green,rotate=24.5] (-1.2,-1.2) rectangle (1.2,1.2);
\draw[fill] (0,0) circle (2pt);
\draw (0,-0.2) node[right] {\small$Q_{p'}$};
\draw[shift={(-3.5,2.5)}, densely dashed, rotate=-11.5] (-1.2,-1.2) rectangle (1.2,1.2);
\fill[shift={(-3.5,2.5)}, green!20,rotate=-11.5] (-1.2,-1.2) rectangle (1.2,1.2);
\fill[shift={(-3.5,2.5)}, pattern=north west lines,pattern color = green,rotate=-11.5] 
	(-1.2,-1.2) rectangle (1.2,1.2);
\draw[fill] (-3.5,2.5) circle (2pt);
\draw (-3.5,2.5) node[below] {\small$Q_p$};
\path[name intersections={of=circle1 and auxRectangle, by={NW1,SW1}},
			name intersections={of=circle2 and auxRectangle, by={NE1,SE1}}];
\draw[densely dashed]	(NW1) -- (NE1) -- (SE1) -- (SW1) -- cycle;
\fill[blue!20,opacity=0.6]	(NW1) -- (NE1) -- (SE1) -- (SW1) -- cycle;
\fill[pattern = north east lines, pattern color = blue!50]	(NW1) -- (NE1) -- (SE1) -- (SW1) -- cycle;
\draw[thick] (-3.5,2.5) -- node[midway, above] {\small$Q_c$} (0,0);
\draw[thick,rotate=120,scale=.4] (0,0) -- (-3.5,2.5);
\draw[thick,rotate=255,scale=.4] (0,0) -- (-3.5,2.5);
\draw[thick] (-3.5,2.5) -- (-3.5,4.2);
\draw[thick] (-3.5,2.5) -- (-5.2,2.1);
\draw (-2.5,0.1) node {\small$\bar\Omega_i$};
\end{tikzpicture}
}

\caption{An illustration of the neighborhoods~$U_n$ and~$Q_n$, $n\in\{c,p,p'\}\subset\mathcal{N}_i$, 
in the setting of a straight line segment~$\mathcal{T}_c$ connecting two distinct triple 
junctions~$\mathcal{T}_p$ and~$\mathcal{T}_{p'}$. \label{fig:nbhdSlices}}
\end{figure}

\begin{figure}
\centering
\subcaptionbox{\label{fig:flatGraphA}}{
\centering
\begin{tikzpicture}[scale=.8]
\draw[name path = interfaceA, draw = none, rotate = -30] (0,0) -- (3.5,0);
\draw[name path = interfaceB, draw = none, rotate = 210] (0,0) -- (3.5,0);
\draw[name path = box, draw = none] (-2.5,-2.5) rectangle (2.5,2.5); 
\path[name intersections={of=interfaceA and box, by={X}}];
\path[name intersections={of=interfaceB and box, by={Y}}];
\fill[yellow!40] (0,0) -- (X) -- (2.5,2.5) -- (-2.5,2.5) -- (Y) -- cycle;
\fill[green!40, opacity=0.6] (-0.9,2.5) rectangle (1.6,-2.5); 
\fill[pattern = vertical lines, pattern color = green] (-0.9,2.5) rectangle (1.6,-2.5);
\draw[fill, blue!80] (0.8,1.5) circle (1.5pt);
\draw[blue!80, thick] (0.8,1.5) -- (1.6,1) node[right] {\tiny$S_{i,j}$};
\draw[blue!80, thick] (0.8,1.5) -- (-0.9,0.3) node[left] {\tiny$S_{l,i}$};
\draw[blue!80, thick] (0.8,1.5) -- (1,2.5) node[above] {\tiny$S_{j,l}$};
\draw[densely dashed, thick] (-2.5,-2.5) rectangle (2.5,2.5); 
\draw[fill] (0,0) circle (2pt);
\draw[thick, rotate = -30] (0,0) -- (3.5,0) node[below, pos=0] {\small$\;\;\mathcal{T}_p$};
\draw[thick, rotate = 80] (0,0) -- (3.2,0);
\draw[thick, rotate = 210] (0,0) -- (3.7,0);
\draw (-2,-2) node {\small$\bar\Omega_i$};
\draw (2,2) node {\small$\bar\Omega_j$};
\draw (-2,2) node {\small$\bar\Omega_l$};
\draw (-3.5,0) -- (3.5,0) node[above, pos=0.08] {\small$\mathcal{L}_{p,i}$};
\draw[<-] (-0.9,3.3) -- (-0.9,0);
\draw[densely dashed] (-0.9,0) -- (-0.9,-3.3);
\draw[fill] (-0.9,0) circle (2pt) node[below] {\tiny$x'_1$};
\draw[<-] (1.6,3.3) -- (1.6,0) node[right, pos=0.1] {\small$\nu_{p,i}$};
\draw[densely dashed] (1.6,0) -- (1.6,-3.3);
\draw[fill] (1.6,0) circle (2pt) node[below] {\tiny$x'_2$};
\end{tikzpicture}
}
\subcaptionbox{\label{fig:flatGraphB}}{
\centering
\begin{tikzpicture}[scale=.8]
\fill[yellow!40] (2.5,0) -- (2.5,2.5) -- (-2.5,2.5) -- (-2.5,0) -- cycle;
\fill[green!40, opacity=0.6] (-0.5,2.5) rectangle (1.6,-2.5); 
\fill[pattern = vertical lines, pattern color = green] (-0.5,2.5) rectangle (1.6,-2.5);
\draw[blue!80, thick] (-0.5,1.52) -- (1.6,1.48) node[above, pos=0.7] {\tiny$S_{i,j}$};
\draw[densely dashed, thick] (-2.5,-2.5) rectangle (2.5,2.5); 
\draw[thick] (-3.5,0) -- (3.5,0);
\draw (-2,-2) node {\small$\bar\Omega_i$};
\draw (-2,2) node {\small$\bar\Omega_j$};
\draw (-3.5,0) -- (0,0) node[above, pos=0.14] {\small$\mathcal{L}_{c,i}$};
\draw[<-] (-0.5,3.3) -- (-0.5,0);
\draw[densely dashed] (-0.5,0) -- (-0.5,-3.3);
\draw[fill] (-0.5,0) circle (2pt) node[below] {\tiny$x'_1$};
\draw[<-] (1.6,3.3) -- (1.6,0) node[right, pos=0.1] {\small$\nu_{c,i}$};
\draw[densely dashed] (1.6,0) -- (1.6,-3.3);
\draw[fill] (1.6,0) circle (2pt) node[below] {\tiny$x'_2$};
\end{tikzpicture}
}
\caption{An illustration of the setting in the neighborhood~$Q_{n,i}$
of either a) a triple junction $n=p\in\mathcal{N}_i$
or b) a straight line segment $n=c\in\mathcal{N}_i$. The yellow regions
are the intersection of~$\{(1{-}\bar\chi_i)>0\}$ with~$Q_{n,i}$.
The green regions hatched from top to bottom represent the images of
the sets~$L_{n,i}(x')$, as defined by~\eqref{eq:oneDimSlices2}  
for $x'\in\mathcal{L}_{n,i}\cap Q_{n,i}$,
under the change of variables~$\Psi_{n,i}$. The arrows on top indicate the underlying
``orientation'' of the slices~$L_{n,i}(x')$ as determined by the slicing normal~$\nu_{n,i}$.
Within the green regions, a possible location of the network of interfaces (blue lines) from the
competitor partition~$(\Omega_1,\ldots,\Omega_P)$ is indicated.
The two regimes a) and b) serve as the main motivation for the definition~\eqref{eq:flatGraphCase}
of the sets~$S^{\mathrm{flat}}_{i,j,n}$, $j \in \mathcal{K}_n\setminus\{i\}$.
\label{fig:flatGraph}}
\end{figure}
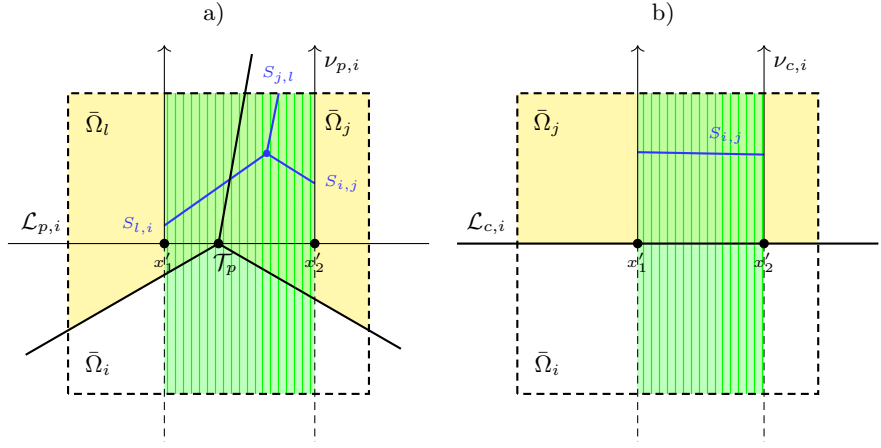

Fix $n \in \mathcal{N}_i$.
Since each phase $\Omega_k $ for $k=1,\ldots,P$ of the competitor is a 
set of finite perimeter in~$\Rd[2]$
the one-dimensional slices 
\begin{align}
\label{eq:oneDimSlices}
(\Omega_k)_{n,i,x'} := 
\big\{y\in\Rd[]\colon\chi_{k}\big(\Psi_{n,i}(x',y)\big) = 1 \big\},
\quad x'\in\mathcal{L}_{n,i},
\end{align} 
are of finite perimeter in $\Rd[]$ for $\mathcal{H}^{1}$ 
almost every $x'\in~\mathcal{L}_{n,i}$ due to Theorem~\ref{OneDimRestrictions}.
In particular, they can be represented by means of a finite disjoint
union of closed intervals, and the reduced boundary $\partial^*(\Omega_i)_{n,i,x'}$ 
is simply given by the corresponding interval endpoints, cf.\ \cite[Proposition 3.52]{AmbrosioFuscoPallara}. 
In particular, $(\Omega_i)_{n,i,x'}$ is also of finite perimeter in the open interval
\begin{align}
\label{eq:oneDimSlices2}
L_{n,i}(x') :=
\big\{y\in\Rd[] \colon \Psi_{n,i}(x',y) \in Q_{n,i}\big\}
\end{align}
for $\mathcal{H}^{1}$ almost every $x' \in \mathcal{L}_{n,i} \cap Q_{n,i}$.
We refer to Figure~\ref{fig:flatGraph} for a visualization of the 
definitions~\eqref{eq:oneDimSlices2} and \eqref{eq:flatGraphCase}, the latter below.

For each $j \in \mathcal{K}_n\setminus\{i\}$, we finally introduce a 
set representing roughly speaking the region
where the rectifiable set $S_{i,j}$ shows up in the set $U_{n} \cap \{(1{-}\bar\chi_i)>0\}$, where it 
appears almost flat when screened with respect to 
the ``calibration normal'' $\frac{\xi_i-\xi_j}{\sigma_{i,j}}$,
and where below~$S_{i,j}$ no phase $\Omega_k$ of the competitor 
with $k \in \mathcal{K}_n\setminus\{i\}$ shows up.
We formalize this as follows: for all $n \in \mathcal{N}_i$, $j \in \mathcal{K}_n\setminus\{i\}$, and $\eps_1 \in (0,1)$ we define
\begin{align}
\label{eq:flatGraphCase}
S^{\mathrm{flat}}_{i,j,n}(\eps_1)
:= \Big\{&x {=} \Psi_{n,i}(x',y) \in S_{i,j} \cap Q_{n,i} \cap \{(1{-}\bar\chi_i)>0\} \colon
\Big(\frac{\xi_i{-}\xi_j}{\sigma_{i,j}}\cdot\vec{n}_{i,j}\Big)(x) > 1{-}\varepsilon_1,\,
\\&\nonumber~~~~~~~~~~~~~~~~~~~~~~~~~~~~~~~~~~~~ L_{n,i}(x')\cap (0,y)\cap 
\bigcup_{k \in \mathcal{K}_n\setminus\{i\}} (\Omega_k)_{n,i,x'} = \emptyset  \Big\}.
\end{align}

We next argue that $S^{\mathrm{flat}}_{i,j,n}(\eps_1)$ cannot have too much length away from its corresponding interface of the supposed local minimizer, as otherwise too much $L^1$-error is accrued.

\begin{lemma}\label{lemma:flat_close_to_interface}
	For all sufficiently small $\eps_1\in (0,1)$ only depending on properties of $\bar \chi$, we have the following:
	Under the assumptions of Lemma \ref{lemma:control_shortness_constraint} and for all $\eps_3 \in (0,1)$ we can choose $\bar s >0$ small enough so that for all $i \in \{1,\ldots,P\}$ we have
	\begin{equation}
\begin{aligned}
\label{Assumption3Step3}
\sum_{n \in \mathcal{N}_{i}} \sum_{j \in \mathcal{K}_n\setminus\{i\}} 
\mathcal{H}^{1} \Big(S^{\mathrm{flat}}_{i,j,n}(\eps_1)
\cap \Big\{\dist(\cdot,\mathcal{I}_i) > \frac{\delta_3}{2}\bar r\Big\}\Big)
< \varepsilon_3,
\end{aligned}
\end{equation}
where $\delta_3$ is taken from property iv) of local paired calibrations, see Definition \ref{DefinitionLocalCalibration}.
\end{lemma}

\begin{proof}
	Towards a contradiction, we assume that for all $\bar s>0$ there exists $i \in\{1,\ldots,P\}$ such that
\begin{equation}
\sum_{n \in \mathcal{N}_{i}} \sum_{j \in \mathcal{K}_n\setminus\{i\}} 
\mathcal{H}^{1} \Big(S^{\mathrm{flat}}_{i,j,n}(\eps_1)
\cap \Big\{\dist(\cdot,\mathcal{I}_i) > \frac{\delta_3}{2}\bar r\Big\}\Big)
> \varepsilon_3.
\end{equation}
In particular, there exists
a constant $C_4>0$, some $n \in \mathcal{N}_i$, and some $j \in \mathcal{K}_n\setminus\{i\}$ 
such that we have
\begin{align}
\label{Assumption3Step3Refined}
\mathcal{H}^{1} \Big(S^{\mathrm{flat}}_{i,j,n}(\eps_1)
\cap \Big\{\dist(\cdot,\mathcal{I}_i) > \frac{\delta_3}{2}\bar r\Big\}\Big) 
> \frac{\varepsilon_3}{C_4}.
\end{align}

A slicing argument together with definition~\eqref{eq:flatGraphCase}
of the sets~$S^{\mathrm{flat}}_{i,j,n}(\eps_1)$ 
then shows that, provided $\eps_1$ is sufficiently small in a universal sense, there exists a constant $\varepsilon_4>0$ such that
\begin{align}
\sum_{k\in\mathcal{K}_n\setminus\{i\}}\|\chi_k{-}\bar\chi_k\|_{L^1(D)} 
\geq \varepsilon_4.
\end{align}	
Note that for $n \in \mathcal{B}$ this step requires the assumed convexity of $B_{2\bar{r}}(\mathcal{T}_n) \cap D$.
It thus suffices to choose $\bar s < \varepsilon_4$ in order to obtain a contradiction.
\end{proof}

In a next step, we transfer the smallness obtained in Lemma \ref{lemma:flat_close_to_interface} to all parts of the interfaces, regardless of steepness.
To this end we introduce the partition
\begin{align}
S_{i,j} \cap Q_{n,i} \cap \{(1{-}\bar\chi_i)>0\}
 \label{eq:decomp}
=: S^{\mathrm{flat}}_{i,j,n}(\eps_1)
\cupdot S^{\mathrm{steep}}_{i,j,n}(\eps_1)
\cupdot S^{\mathrm{res}}_{i,j,n}(\eps_1)
\end{align}
we alluded to after Lemma \ref{lemma:tilt_excess_control},
where we used
\begin{align*}
S^{\mathrm{steep}}_{i,j,n} (\eps_1)
:= \Big\{&x \in S_{i,j} \cap Q_{n,i} \cap \{(1{-}\bar\chi_i)>0\} \colon
\Big(\frac{\xi_i{-}\xi_j}{\sigma_{i,j}}\cdot\vec{n}_{i,j}\Big)(x) \leq 1{-}\varepsilon_1\Big\},
\\
S^{\mathrm{res}}_{i,j,n} (\eps_1)
:= \Big\{&x {=} \Psi_{n}(x',y) \in S_{i,j} \cap Q_{n,i} \cap \{(1{-}\bar\chi_i)>0\} \colon
\Big(\frac{\xi_i{-}\xi_j}{\sigma_{i,j}}\cdot\vec{n}_{i,j}\Big)(x) > 1{-}\varepsilon_1,\,
\\&\nonumber~~~~~~~~~~~~~~~~~~~~~~~~~~~~~~~~~~~ L_{n,i}(x')\cap (0,y)\cap 
\bigcup_{k \in \mathcal{K}_n\setminus\{i\}} (\Omega_k)_{n,x'} \neq \emptyset  \Big\}.
\end{align*}
We will also record an estimate of the length of the two additional sets in terms of the relative energy, as this is what we will later compensate the signed bulk term in estimate \eqref{ProofTheoremAux1} with.

Due to Lemma \ref{lemma:flat_close_to_interface}, we will mostly have to deal with the residual parts of the interface.
To this end, we slice towards $\bar \Omega_i$ and observe that either we hit a phase that should not be present at the considered topological feature $n$ or we hit a steep part of the interface.
The estimate then relies on applications of the coarea formula, for which we need a bound away from zero on the coarea factor.
This will be the flatness condition of $S^{\mathrm{res}}_{i,j,n} (\eps_1)$, mediated by the estimate \eqref{eq:coareaFactor}.

\begin{lemma}\label{lemma:short_inside_strip}
Under the assumptions of Lemma \ref{lemma:control_shortness_constraint} there exists a constant $\varepsilon_1 \in (0,1)$ sufficiently small to apply Lemmas \ref{lemma:geometry} and \ref{lemma:flat_close_to_interface} so that the following holds:

For all $\eps_2,\eps_3\in (0,1)$ let $\bar s >0$ be small enough such that
\begin{align}
\label{AssumptionStep3}
\sum_{i,j=1,i\neq j}^P \mathcal{H}^1\big(\big\{x\in S_{i,j}\colon \frac{\xi_i{-}\xi_j}{\sigma_{i,j}}
\cdot \vec{n}_{i,j}\leq 1{-}\varepsilon_1\big\}\big)
\leq \varepsilon_2
\end{align}
and the conclusion of Lemma \ref{lemma:flat_close_to_interface} hold.

Then there exists constant $C_3 > 0$ such that we have
\begin{equation}
\begin{aligned}
\label{eq:smallnessFlatGraphCase}
&\sum_{i=1}^P\sum_{n\in\mathcal{N}_i}
\mathcal{H}^1 \Big(\supp|\nabla\chi_{i}| \cap U_n \cap \{(1{-}\bar\chi_i)>0\} 
\cap \Big\{\dist(\cdot,\mathcal{I}_i) > \frac{\delta_3}{2}\bar r\Big\}\Big)
\\&~~~
\leq C_3(\varepsilon_2 + \varepsilon_3).
\end{aligned}
\end{equation}
Additionally, for $i \in \{1,\ldots,P\}$ and $n \in \mathcal{N}_i$ we have
\begin{equation}
\begin{aligned}
\sum_{j \in \mathcal{K}_n \setminus\{i\}}
\mathcal{H}^1\big(S^{\mathrm{steep}}_{i,j,n}(\eps_1)
\cupdot S^{\mathrm{res}}_{i,j,n}(\eps_1) \big) +\sum_{j \notin \mathcal{K}_n}
\mathcal{H}^1 \Big(S_{i,j} \cap Q_{n,i} \Big)
\label{eq:estimateNoFlatGraphRelativeEnergy}
\leq \frac{C_3}{\sigma_{\mathrm{min}}\varepsilon_1} 
E[\chi|\bar \chi].
\end{aligned}
\end{equation}
\end{lemma}

\begin{proof}
\textit{Step 1: Preliminary estimates}.
Let $i \in \{1,\ldots,P\}$ and $n \in \mathcal{N}_i$.
We clearly have
\begin{equation}
\begin{aligned}
&\mathcal{H}^1 \Big(\supp|\nabla\chi_{i}| \cap U_n \cap \{(1{-}\bar\chi_i)>0\}
\cap \Big\{\dist(\cdot,\mathcal{I}_i) > \frac{\delta_3}{2}\bar r\Big\}\Big)
\\& \label{eq:unionBoundSmallness}
\leq \sum_{j=1,\,j\neq i}^P \mathcal{H}^1 \Big(S_{i,j} \cap U_n \cap \{(1{-}\bar\chi_i)>0\} + 
\cap \Big\{\dist(\cdot,\mathcal{I}_i) > \frac{\delta_3}{2}\bar r\Big\}\Big).
\end{aligned}
\end{equation}

For a phase $j \notin \mathcal{K}_n$, 
we necessarily have $\frac{|\xi_i-\xi_j|}{\sigma_{i,j}}\leq \delta_1 < 1$ 
throughout the set $U_n$ 
by property~\textit{ii)} in Definition~\ref{DefinitionLocalCalibration} of a local paired calibration,
the property~\eqref{eq:auxPropertyCubes100}, and Definition~\ref{def:locRadius}
of an admissible pair of localization scales. Hence, choosing $\varepsilon_1\in (0,1)$ small 
enough (depending in this context only on $\delta_1$ from 
Definition~\ref{DefinitionLocalCalibration} of a local paired calibration),
we obtain that
\begin{align}\label{eq:smallnessCond4}
\sum_{j \notin \mathcal{K}_n}
\mathcal{H}^1 \Big(S_{i,j} \cap Q_{n,i}\Big)
\leq 
\frac{C_3}{\sigma_{\mathrm{min}}\varepsilon_1} 
\sum_{
j=1,\,j\neq i}^P \sigma_{i,j} \int_{S_{i,j}}
1 - \frac{\xi_{i}{-}\xi_j}{\sigma_{i,j}} \cdot \vec{n}_{i,j} \,\mathrm{d}\mathcal{H}^1,
\end{align}
as well as, by assumption \eqref{AssumptionStep3}, that
\begin{align}\label{eq:smallnessCond5}
&\sum_{j \notin \mathcal{K}_n}
\mathcal{H}^1 \left(S_{i,j} \cap Q_{n,i}\right)
\leq 
\eps_2.
\end{align}

Fix now a phase $j\in\mathcal{K}_n$.
Assumptions~\eqref{AssumptionStep3} and, after choosing $\eps_1$ sufficiently small, the applicability of Lemma \ref{lemma:flat_close_to_interface} immediately imply
\begin{align}
\label{eq:estimateFlatAndSteepGraph}
&\sum_{j \in \mathcal{K}_n \setminus\{i\}}
\mathcal{H}^1 \Big(S^{\mathrm{flat}}_{i,j,n} (\eps_1) \cap 
\Big\{\dist(\cdot,\mathcal{I}_i) > \frac{\delta_3}{2}\bar r\Big\}\Big)
{+} \mathcal{H}^1 \big(S^{\mathrm{steep}}_{i,j,n}(\eps_1) \big)
\leq C_3(\varepsilon_3 {+} \varepsilon_2).
\end{align}
By the definition of the set~$S^{\mathrm{steep}}_{i,j,n}(\eps_1)$, the
part of the estimate~\eqref{eq:estimateNoFlatGraphRelativeEnergy} concerning the
$\mathcal{H}^1$ measure of~$S^{\mathrm{steep}}_{i,j,n}(\eps_1)$ is also immediate.

\textit{Step 2: The slicing argument.}
Let $x \in S^{\mathrm{res}}_{i,j,n}(\eps_1)$.
Additionally, let $x' := P_{n,i}(x)$
and $y \in L_{n,i}(x')$ be the coordinates such that $x=\Psi_{n,i}(x',y)$.
Note that we also have  $x' \in \mathcal{L}_{n,i} \cap Q_{n,i}$ and that if $n \in \mathcal{B}$ all the sets considered here are convex by the assumption of $D$ being locally convex at $\mathcal{T}_n$.

As $\vec{n}_{i,j}$ points from phase $i$ to phase $j$, the flatness condition of $S^{\mathrm{res}}_{i,j,n}$ and property~\eqref{eq:coareaFactor} imply that the phase~$j$ of the competitor lies above phase~$i$ in the direction of the slicing normal $\nu_{n,i}$.
Furthermore, the ``non-graph'' condition of $S^{\mathrm{res}}_{i,j,n}$ says that the map
\begin{align*}
	\tilde g(x) := \min\left\{ t>0 : x -t \nu_{n,i}  \in \partial^* \Omega_i \cap \{(1-\bar \chi_i )>0 \}\right\}
\end{align*}
is well-defined and
\begin{align}
	g(x) := x - \tilde g(x) \nu_{n,i} 
\end{align}
is injective with
\begin{align*}
	g\left(S^{\mathrm{res}}_{i,j,n}(\eps_1) \right) \subset \left(\bigcup_{k \not\in \mathcal{K}_n} \partial^*\Omega_k \cup \bigcup_{k \in \mathcal{K}_n\setminus\{i\}} \partial^*\Omega_k \right) \cap Q_{n,i} \cap \{(1-\bar \chi_i )>0 \}.
\end{align*}
If the corresponding phase satisfies $k \in \mathcal{K}_n\setminus\{i\}$, then we must have $g(x) \in S^{\mathrm{steep}}_{i,k,n}(\eps_1)$ since otherwise equation \eqref{eq:coareaFactor} would give $\nu_{n,i} \cdot n_{i,k}(g(x)) >0$, which cannot be as we exit phase $i$ at $g(x)$ by going down.
In particular, we get
\begin{align}\label{eq:map_to_steep_or_garbage}
	g\left(S^{\mathrm{res}}_{i,j,n}(\eps_1) \right) \subset \left( \bigcup_{k \not\in \mathcal{K}_n} \partial^*\Omega_k  \cap Q_{n,i}\right) \cup \bigcup_{k \in \mathcal{K}_n\setminus\{i\}}S^{\mathrm{steep}}_{i,k,n}(\eps_1).
\end{align}

For $m \in \mathbb{N}$ we define the set
\begin{align*}
	\hat P_m := \left\{x' \in \mathcal{L}_{n,i} \cap Q_{n,i}: |\{x \in S^{\mathrm{res}}_{i,j,n}(\eps_1): P_{n,i}(x)= x'\}| = k \right\}.
\end{align*}
By the flatness condition, estimate \eqref{eq:coareaFactor} and the coarea formula, we have that
\begin{align*}
	\mathcal{H}^1(S^{\mathrm{res}}_{i,j,n}(\eps_1) \leq C_3 \sum_{m\in \mathbb{N}} m \mathcal{H}^1\left(\hat P_m\right).
\end{align*}
As the set $g(\{ x\in S^{\mathrm{res}}_{i,j,n}: P_{n,i}(x) \in \hat P_m\})$ is a union of $m$ graphs over $\hat P_m$ by definition, the coarea formula together with the inclusion \eqref{eq:map_to_steep_or_garbage}, we finally get
\begin{align*}
	\mathcal{H}^1\left(S^{\mathrm{res}}_{i,j,n}(\eps_1)\right) & \leq C_3 \left(\sum_{j \notin \mathcal{K}_n}  \mathcal{H}^1 \left(S_{i,j} \cap Q_{n,i}\right)+ \sum_{j \in \mathcal{K}_n \setminus\{i\}}
  \mathcal{H}^1 \big(S^{\mathrm{steep}}_{i,j,n}(\eps_1) \big)\right).
\end{align*}
The statement \eqref{AssumptionStep3} therefore follows from estimates \eqref{eq:unionBoundSmallness}, \eqref{eq:smallnessCond5} and \eqref{eq:estimateFlatAndSteepGraph}, while the claim \eqref{eq:estimateNoFlatGraphRelativeEnergy} follows from estimate \eqref{eq:smallnessCond5} and the remark after estimate \eqref{eq:estimateFlatAndSteepGraph}.
\end{proof}

Equipped with these statements, we are able to estimate the only remaining, possibly negative term in the inequality \eqref{ProofTheoremAux1}.
To this end, we exploit property~\textit{iv)} 
in Definition~\ref{DefinitionLocalCalibration} of a local paired calibration, which we have not done so far.
Thus we only have to consider contributions that are contained in a strip-like set away from the interfaces, such as the one shown in Figure \ref{fig:nbhdEstimate}.
As a result of the shortness condition provided by Lemma \ref{lemma:short_inside_strip}, we get that such contributions have to stay ``trapped'' in at least a slightly widened strip, so that we may again use a relative isoperimetric inequality.
Again Lemma \ref{lemma:short_inside_strip} provides the smallness, as well as the estimate against the relative energy for the steep and residual parts of the interfaces.
It finally remains to prove a similar estimate for the flat parts.

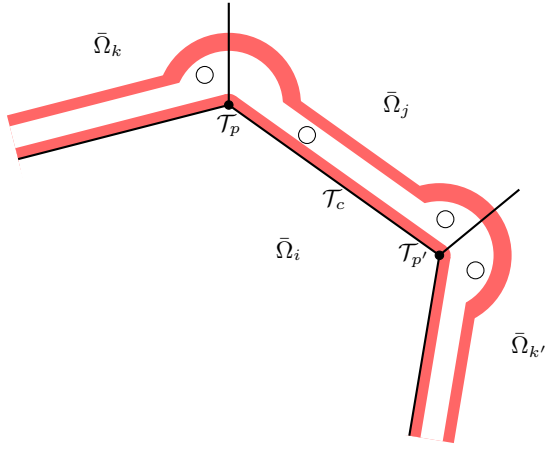
\begin{figure}
\centering
\begin{tikzpicture}[scale=.8]
\draw[red!60,line width=1.2cm] (-3.5,2.5) -- (0,0);
\draw[red!60,line width=1.2cm] (0,0) -- (-0.5,-3);
\draw[red!60,line width=1.2cm] (-3.5,2.5) -- (-7,1.6);
\fill[red!60] (0,0) circle (1.2cm);
\fill[white] (0,0) circle (0.9cm);
\fill[red!60] (-3.5,2.5) circle (1.2cm);
\fill[white] (-3.5,2.5) circle (0.9cm);
\draw[white,line width=0.9cm] (-3.5,2.5) -- (0,0);
\draw[white,line width=0.9cm] (0,0) -- (-0.5,-3);
\draw[white,line width=0.9cm] (-3.5,2.5) -- (-7,1.6);
\draw[red!60,line width=0.3cm] (-3.5,2.5) -- (0,0);
\draw[red!60,line width=0.3cm] (0,0) -- (-0.5,-3);
\draw[red!60,line width=0.3cm] (-3.5,2.5) -- (-7,1.6);
\fill[red!60] (0,0) circle (5.3pt);
\fill[red!60] (-3.5,2.5) circle (5.3pt);
\fill[white] (-7,1.6) -- (-3.5,2.5) -- (0,0) -- (-0.5,-3) -- (-6,-3) -- cycle;
\draw[fill] (0,0) circle (2pt);
\draw (0,0) node[left] {\small$\mathcal{T}_{p'}$};
\draw[fill] (-3.5,2.5) circle (2pt);
\draw (-3.5,2.5) node[below] {\small$\mathcal{T}_{p}$};
\draw (0.6,-0.25) circle (4pt);
\draw (0.1,0.6) circle (4pt);
\draw (-2.2,2) circle (4pt);
\draw (-3.9,3) circle (4pt);
\draw[thick] (-3.5,2.5) -- node[midway, below] {\small$\mathcal{T}_c$} (0,0);
\draw[thick] (0,0) -- (-0.5,-3);
\draw[thick,rotate=255,scale=.4] (0,0) -- (-3.5,2.5);
\draw[thick] (-3.5,2.5) -- (-3.5,4.2);
\draw[thick] (-3.5,2.5) -- (-7,1.6);
\draw (-2.5,0.1) node {\small$\bar\Omega_i$};
\draw (-0.7,2.5) node {\small$\bar\Omega_j$};
\draw (1.5,-1.5) node {\small$\bar\Omega_{k'}$};
\draw (-5.5,3.5) node {\small$\bar\Omega_{k}$};
\end{tikzpicture}
\caption{The stripe in between the two red regions
indicates the domain with which the proof of Proposition \ref{prop:stripe_lower_bound} is concerned and in which stray boundaries must remain trapped.
 \label{fig:nbhdEstimate}}
\end{figure}

\begin{proposition}\label{prop:stripe_lower_bound}
	Under the assumptions of Lemma \ref{lemma:control_shortness_constraint}, we can choose $\eps_1,\eps_2,\eps_3>0$ and $\bar s>0$ small enough such that we have
	\begin{align*}
		\sum_{i,j=1,i\neq j}^P \int_{U^{\mathcal{I}_i}_{(\frac{\bar r}{2},\delta)}} 
		2(\chi_i{-}\bar\chi_i) \bar\chi_j (\nabla\cdot\xi_i{-}\nabla\cdot\xi_j) \dx 
		\geq -\frac{1}{4}E[\chi|\bar\chi].
	\end{align*}
\end{proposition}

\begin{proof}
Let $\eps_1$ be as in Lemma \ref{lemma:flat_close_to_interface}.
We then choose $\eps_2\in (0,1)$ and $\eps_3 \in (0,1)$, which we can do in dependence on $\eps_1$ due to Lemma \ref{lemma:flat_close_to_interface} holding for all of them.
The proof will indicate the details in the end.
Finally, we choose $\bar s$ sufficiently small to apply Lemmas \ref{lemma:tilt_excess_control} and \ref{lemma:flat_close_to_interface}.
We thus have Lemmas \ref{lemma:flat_close_to_interface} and \ref{lemma:short_inside_strip} at our disposal after using Lemma \ref{lemma:tilt_excess_control} to deal with the case of estimate \eqref{AssumptionStep2} holding.


By means of the divergence constraint~\textit{iv)}
in Definition~\ref{DefinitionLocalCalibration} of a local paired calibration, we get
\begin{equation}
\begin{aligned}
&\sum_{i,j=1,i\neq j}^P \int_{U^{\mathcal{I}_i}_{(\frac{\bar r}{2},\delta)}} 2(\chi_i {-} \bar\chi_i) 
\bar\chi_j (\nabla\cdot\xi_i{-}\nabla\cdot\xi_j) \dx
\\&\label{ProofTheoremAux1Step3}
\geq -4\max_{i, \ldots, P}\|\xi_i\|_{W^{1,\infty}(D)}
\sum_{i,j=1, i\neq j}^P 
\int_{U^{\mathcal{I}_i}_{(\frac{\bar r}{2},\delta)} \cap \{\dist(\cdot,\mathcal{I}_i) > \delta_3\bar r\}} \chi_{i}\bar\chi_j\dx.
\end{aligned}
\end{equation}
Here, we also exploited the fact that~$\bar\chi$ represents a partition of~$D$ 
so that $\bar\chi_i\bar\chi_j=0$ for all $i,j\in\{1,\ldots,P\}$ with $i\neq j$.
Furthermore, exploiting $\sum_{j\in \{1,\ldots,P\}\setminus \{j\} } \bar\chi_j = 1{-}\bar\chi_i$
a.e.\ in $U_n$ we obtain
\begin{equation}
\begin{aligned}
&\quad \sum_{i,j=1, i\neq j}^P 
\int_{U^{\mathcal{I}_i}_{(\frac{\bar r}{2},\delta)} \cap \{\dist(\cdot,\mathcal{I}_i) > \delta_3\bar r\}} \chi_{i}\bar\chi_j\dx
&\label{ProofTheoremAux1Step3Refined}
=\sum_{i=1}^P 
\int_{U^{\mathcal{I}_i}_{(\frac{\bar r}{2},\delta)} \cap \{\dist(\cdot,\mathcal{I}_i) > \delta_3\bar r\}} 
\chi_{i} (1 {-} \bar\chi_i)\dx,
\end{aligned}
\end{equation}
so that it is sufficient to estimate the last term from above.

As by the structure result for sets of finite perimeter in the plane, the phase boundary $\partial^*\Omega_i$ for each $i=1,\ldots,P$ consists of at most countably many, disjoint, closed Lipschitz curves whose total length coincides with the perimeter. 
We now choose $\eps_2 \in (0,1)$ and $\eps_3 \in (0,1)$ small enough such that all boundary curves of phases $i =1,\ldots,P$ intersecting $U^{\mathcal{I}_i}_{(\frac{\bar r}{2},\delta)}  \cap \{\dist(\cdot,\mathcal{I}_i) > \delta_3\bar r\} \setminus \bar \Omega_i$ must still be contained in $V_{i,n} := U^{\mathcal{I}_i}_{(\bar r,\delta)} \cap \{\dist(\cdot,\mathcal{I}_i) > \frac{\delta_3}{2}\bar r\} \setminus \bar \Omega_i$ after an application of Lemma~\ref{lemma:short_inside_strip}.

Let $i \in \{1,\ldots,P\}$ and let $G_{i} \subset \Omega_i \cap V_{i}$ be the corresponding parts of $\Omega_i$.
The standard isoperimetric inequality gives
\begin{align}\label{eq:main_isoperimetric}
	\int_{U^{\mathcal{I}_i}_{(\frac{\bar r}{2},\delta)} \cap \{\dist(\cdot,\mathcal{I}_i) > \delta_3\bar r\}} 
\chi_{i} (1 {-} \bar\chi_i)\dx \leq C P(G_{i})^2.
\end{align}
By Lemma \ref{lemma:short_inside_strip}, we have
\begin{align}\label{eq:main_smallness}
	P(G_{i})^2 \leq C_3(\eps_2+\eps_3) P(G_{i}).
\end{align}
It therefore remains to estimate $P(G_i)$ also in terms of the relative energy.

Note that we must have
\begin{align}\label{eq:G_i_inclusion}
  \begin{split}
	\partial^* G_{i}  & \subset \bigcup_{n \in N_i} \Bigg( \bigcup_{j \not\in \mathcal{K}_n\setminus\{i\}} S_{i,j} \cap Q_{n,i} \cap \{(1{-}\bar\chi_i)>0\} \cap \Big\{\dist(\cdot,\mathcal{I}_i) > \frac{\delta_3}{2}\bar r\Big\} \\
	&\qquad \cup \bigcup_{j\in \mathcal{K}_n\setminus\{i\}} \left(S_{i,j,n}^{\mathrm{flat}}(\eps_1) \cup S_{i,j,n}^{\mathrm{steep}}(\eps_1) \cup S_{i,j,n}^{\mathrm{res}}(\eps_1)  \right)\Bigg).
  \end{split}
\end{align}
Again Lemma \ref{lemma:short_inside_strip} provides the desired estimate at all but the flat parts.

Let $n \in N_i$ and $j\in \mathcal{K}_n\setminus\{i\}$.
For each $x \in \partial^*G_i \cap  S_{i,j,n}^{\mathrm{flat}}(\eps_1)$ we consider the function
\begin{align*}
	\tilde g(x) := \min \left\{ t>0: x - t \nu_{p,i} \in \partial^*G_i \cap Q_{n,i} \right\},
\end{align*}
where $\nu_{p,i}$ were defined in Lemma \ref{lemma:geometry}.
Note that the function $\tilde g$ is well-defined as $G_i \subset V_i$ must be an inclusion into $D\setminus \bar \Omega_i$.
Furthermore, we define $g(x) := x -  \tilde g(x) \nu_{p,i}$ and note that by the definition \eqref{eq:flatGraphCase} of $S_{i,j,n}^{\mathrm{flat}}(\eps_1)$ we have
\begin{align}\label{eq:projection_on_minority_phase}
	g(x) \in \bigcup_{k \not\in \mathcal{K}_n\setminus\{i\}} S_{i,k} \cap Q_{n,i}.
\end{align}
By the same slicing argument as in Step 2 of Lemma \ref{lemma:short_inside_strip} and estimate \eqref{eq:estimateNoFlatGraphRelativeEnergy}, we have
\begin{align}\label{eq:fff}
	\mathcal{H}^1( \partial^*G_i \cap  S_{i,j,n}^{\mathrm{flat}}(\eps_1)) \leq \sum_{j \notin \mathcal{K}_n}
\mathcal{H}^1 \Big(S_{i,j} \cap Q_{n,i} \Big)
\leq \frac{C_3}{\sigma_{\mathrm{min}}\varepsilon_1} 
E[\chi|\bar \chi].
\end{align}

Using Lemma \ref{lemma:short_inside_strip}, and combining the result with the estimates
\eqref{ProofTheoremAux1Step3}, \eqref{ProofTheoremAux1Step3Refined}--\eqref{eq:G_i_inclusion} and \eqref{eq:fff}, we thus get
\begin{align*}
	 \sum_{i,j=1,i\neq j}^P \int_{U^{\mathcal{I}_i}_{(\frac{\bar r}{2},\delta)}} 2(\chi_i {-} \bar\chi_i)  \bar\chi_j (\nabla\cdot\xi_i{-}\nabla\cdot\xi_j) \dx  \geq -C \frac{\eps_2+\eps_3}{\eps_1} E[\chi|\bar\chi].
\end{align*}
Choosing $\eps_2, \eps_3 \in (0,1)$ small enough, we obtain the statement.
\end{proof}

\begin{proof}[Proof of Theorem \ref{MainResult}]
	Since all boundary points $\mathcal{T}_b$ with $b\in \mathcal{B}$ are separated from each other, we can slightly enlarge $D$ to be locally convex around each $\mathcal{T}_b$ and extend the locally flat partition simply by elongating all straight lines until they hit the boundary, possibly up to reducing the admitted localization scale we work with.
	Also all competitors may be extended similarly.
	Therefore, the result immediately follows from Lemma \ref{lemma:control_shortness_constraint} and Proposition \ref{prop:stripe_lower_bound}.
\end{proof}

\section{Stationary points: Proof of Theorem~\ref{thm:stationary}}

This section is devoted to the proof of Theorem~\ref{thm:stationary} stating that any critical point of the interface length functional is a regular flat partition. 
We first state and prove two auxiliary results and then present the proof of Theorem~\ref{thm:stationary} based on these lemmas.

The two auxiliary results only rely on domain variations and therefore can already be found in the literature in similar forms. 
The first auxiliary result states the Euler--Lagrange equation that is satisfied by any stationary point.

\begin{lemma}[Equilibrium equation]\label{lemma:ELeq}
	In the setting of Theorem~\ref{thm:stationary}, for any compactly supported test vector field $\eta\in C^1_{cpt}(D;\mathbb{R}^2)$, it holds
	\begin{align}
		\label{EqEquation}
		\sum_{i=1}^P \int_D \bigg(\Id-\frac{\nabla \chi_i}{|\nabla \chi_i|}\otimes \frac{\nabla \chi_i}{|\nabla \chi_i|} \bigg) : \nabla \eta \,d|\nabla \chi_i| =0.
	\end{align}
\end{lemma}

\begin{proof}
	Let $\eta\in C^1_{cpt}(D;\mathbb{R}^2)$ be a compactly supported test vector field. 
	Define $\tilde \chi(x):=\chi(x+\delta\eta(x))$ for $\delta>0$. 
	It is not too difficult to see that $d_H(\tilde \chi,\chi)\leq C(\eta) \delta$. Furthermore, by the area formula we have
	\begin{align*}
		&\frac{1}{2} \sum_{i=1}^P |\nabla \tilde \chi_i|(D)
		\\&=
		\frac{1}{2} \sum_{i=1}^P |\nabla \chi_i|(D)
		-\frac{1}{2} \delta \sum_{i=1}^P \int_D \bigg(\Id-\frac{\nabla \chi_i}{|\nabla \chi_i|}\otimes \frac{\nabla \chi_i}{|\nabla \chi_i|} \bigg) : \nabla \eta \,d|\nabla \chi_i|
		+O(\delta^2).
	\end{align*}
	Plugging these two estimates into the stationarity condition  \eqref{EqCondition} and letting $\delta\downarrow 0$ and then using the symmetry $\eta \mapsto -\eta$, we see that $\chi$ satisfies the equilibrium equation~\eqref{EqEquation}.
\end{proof}

The next lemma states a monotonicity formula for stationary points similar to the result by Allard and Almgren~\cite{AllardAlmgren}.
\begin{lemma}[Monotonicity formula]\label{lemma:monotonicity}
	In the setting of Theorem~\ref{thm:stationary}, for any point $x\in D$, there exists a radius $\bar r>0$ such that the network of interfaces $\mathcal{I} = \cup_{i,j} I_{i,j}$ of the partition $\chi$ satisfies
	\begin{align}
		\frac{1}{r} \mathcal{H}^1\big(B_r(x)\cap\mathcal{I} \big)
		\label{RAverage}
		=\sum_{y\in \partial B_r(x) \cap\mathcal{I} } \cos \alpha_y
	\end{align}
	for almost every $r\in (0,\bar r)$, where $\alpha_y$ denotes the (smaller) angle between the normal of $\partial B_r(x)$ and $\mathcal{I}$ at $y$.
	Furthermore, in the distributional sense, for $r\in (0,\bar r)$, it holds
	\begin{align}\label{eq:monotonicityformula}
		\notag \partial_r \bigg[\frac{1}{r} \mathcal{H}^1\big(B_r(x)\cap \mathcal{I} \big)\bigg] 
		&= \sum_{y\in \partial B_r(x) \cap \mathcal{I}}\bigg(\frac{1}{\cos \alpha_y}-\cos \alpha_y\bigg)
				\\&~~~
				+ \frac{1}{r} \partial_r\bigg[ \mathcal{H}^1\Big(B_r(x)\cap \mathcal{I} \cap \big\{y:\tfrac{\nabla \chi_i}{|\nabla \chi_i|}(y) \parallel \tfrac{y-x}{|y-x|} \big\}\Big)\bigg]\geq 0.
	\end{align}
\end{lemma}

\begin{proof}
	We wish to plug the variation~$\eta(y) = \chi_{B_r(x)}(y) \frac{y-x}{r}$ into the equilibrium equation \eqref{EqEquation}. 
	To make this rigorous, we test~\eqref{EqEquation} with $\eta_\delta(y):= f_\delta(\frac{|y-x|}{r}) \frac{y-x}{r}$ for $\delta>0$ and a.\,e.\ $r\in (0,\dist(x,\partial D))$, where~$f_\delta$ is Lipschitz continuous with~$f_\delta= 1 $ in~$[0,1-\delta]$, $f_\delta  = 0$ in~$[1,\infty)$ and $f_\delta$ linear in~$[1-\delta,1]$.
	Then taking~$\delta\downarrow0$ we obtain for a.\,e.\ $r\in (0,\dist(x,\partial D))$ that
	\begin{align*}
		&
		\frac{1}{r} \mathcal{H}^1\Big(B_r(x)\cap \bigcup_{i,j} I_{i,j}\Big)
		\\& 
		=\sum_{ i<j} \int_{\partial B_r(x) \cap  I_{i,j}} \bigg(\Id-\tfrac{\nabla \chi_i}{|\nabla \chi_i|}\otimes \tfrac{\nabla \chi_i}{|\nabla \chi_i|} \bigg) : \tfrac{y-x}{|y-x|}\otimes \tfrac{y-x}{|y-x|}  ~~ \frac{1}{\big|\tfrac{\nabla \chi_i}{|\nabla \chi_i|}^\perp \cdot \tfrac{y-x}{|y-x|}\big|} \,d\mathcal{H}^0(y)
		\\&
		=\sum_{y\in \partial B_r(x) \cap \cup_{i,j} I_{i,j}} \cos \alpha_y,
	\end{align*}
	where $\alpha_y$ denotes the (smaller) angle between the normal $(y-x)/|y-x|$ of $\partial B_r(x)$ and $\cup_{i,j} I_{i,j}$ at $y$.
	This is~\eqref{RAverage}.
	
	By a slicing argument and~\eqref{RAverage} we have in the distributional sense
	\begin{align*}
		&\partial_r \bigg[\frac{1}{r} \mathcal{H}^1\Big(B_r(x)\cap \bigcup_{i,j} I_{i,j}\Big)\bigg]
		\\&
		=\frac{1}{r} \partial_r \int_0^r \sum_{1\leq i<j\leq P}  \int_{\partial B_{s}(x) \cap I_{i,j}} \frac{1}{\big|\tfrac{\nabla \chi_i}{|\nabla \chi_i|}^\perp\cdot \frac{y-x}{|y-x|}\big|} \,d\mathcal{H}^0 \,ds
		\\&~~~
		+\frac{1}{r} \partial_r\bigg[ \mathcal{H}^1\Big(B_r(x)\cap \bigcup_{i,j} I_{i,j}\cap \big\{y:\tfrac{\nabla \chi_i}{|\nabla \chi_i|}(y) \parallel \tfrac{y-x}{|y-x|} \big\}\Big)\bigg]
		\\&~~~
		-\frac{1}{r^2} \mathcal{H}^1\Big(B_r(x)\cap \bigcup_{i,j} I_{i,j}\Big)
		\\&
		= \sum_{y\in \partial B_r(x) \cap \cup_{i,j} I_{i,j}}\bigg(\frac{1}{\cos \alpha_y}-\cos \alpha_y\bigg)
		\\&~~~
		+ \frac{1}{r} \partial_r\bigg[ \mathcal{H}^1\Big(B_r(x)\cap \bigcup_{i,j} I_{i,j} \cap \big\{y:\tfrac{\nabla \chi_i}{|\nabla \chi_i|}(y) \parallel \tfrac{y-x}{|y-x|} \big\}\Big)\bigg].
	\end{align*}
	Finally, the monotonicity formula~\eqref{eq:monotonicityformula} follows since both terms on the right-hand side are non-negative. 
\end{proof}

Now we are in the position to prove the second main result of this paper.

\begin{proof}[Proof of Theorem~\ref{thm:stationary}]
The proof is carried out in three steps. 
First, we rule out junctions of very high order. 
Second, by the monotonicity formula, we go down to small scales and see that for any point, the number of intersections with sufficiently small circles around that points is finite and constant. 
Third, we use the stationarity condition to show that this number of intersections can actually only be $0$, $2$ or $3$ and that in the latter two cases, the intersection points lie at almost equal angles.

\emph{Step 1: Let $x\in D$. Then for all sufficiently small $r>0$ we have 
\[
\frac{1}{r} \mathcal{H}^1\big(B_r(x)\cap \mathcal{I}\big) \leq 7.
\]}

Suppose that we had $\tfrac{1}{r} \mathcal{H}^1(B_r(x)\cap \mathcal{I})>7$. 
Then we construct a competitor by filling in the ball $B_r(x)$ with one phase. 
More precisely, define $\tilde \chi_i$ by $\tilde \chi_i(y):=\chi_i(y)$ outside of $B_r$ and $\chi_i(y):=\delta_{im}$ in $B_r$ for some $m$ with $\supp \chi_m\cap B_r(x)\neq\emptyset$. 
Then we have\footnote{In fact, to guarantee this bound on the Hausdorff distance we would need to slightly modify the above construction and place an arbitrarily small volume of the phases $i$ with $\supp \chi_i \cap B_r(x)$ also in the $\tilde \chi_i$ in $B_r(x)$. For readability, we omit this technicality.} $d_H(\tilde\chi, \chi)\leq 2r$ and
\begin{align*}
\frac{1}{2} \sum_{i=1}^P |\nabla \tilde \chi_i|\big(\overline{ B_{r}(x)}\big)
\leq 2\pi r
\leq 7r-0.1r
\leq
\frac{1}{2} \sum_{i=1}^P |\nabla \chi_i|\big(\overline{ B_{r}(x)}\big) -0.1r,
\end{align*}
as well as
\begin{align*}
\frac{1}{2} \sum_{i=1}^P |\nabla \tilde \chi_i|\big(D\setminus \overline{ B_{r}(x)}\big)
=
\frac{1}{2} \sum_{i=1}^P |\nabla \chi_i|\big(D\setminus \overline{ B_{r}(x)}\big).
\end{align*}
For $r>0$ sufficiently small, this contradicts the stationarity condition \eqref{EqCondition}.

\emph{Step 2: Let $x\in D$. Then, for a.\,e.\ $r>0$ sufficiently small,  $\mathcal{I}$ intersects~$\partial B_r(x)$ in a fixed finite number~$k$ of points (with~$k$ being independent of~$r$), and must do so orthogonally.}

By the monotonicity formula~\eqref{eq:monotonicityformula}, the limit 
\[
\lim_{r\downarrow 0} \frac{1}{r} \mathcal{H}^1\big(B_r(x)\cap \mathcal{I}\big)
\] 
exists. Furthermore, this limit must be an integer and hence is attained for all $r>0$ small enough: Indeed, if $\tfrac{1}{r} \mathcal{H}^1\big(B_r(x)\cap \mathcal{I}\big)=a$ holds for some $a \geq0$ and some $r>0$, there exists a radius $s\in (0,r)$ such that $\mathcal{I}\cap B_{s}(x)$ consists of at most $\lfloor a \rfloor$ points. By \eqref{RAverage} and the estimate $|\cos \alpha_y|\leq 1$, we get $\tfrac{1}{s} \mathcal{H}^1\big(B_{s}(x)\cap \mathcal{I}\big)\leq \lfloor a \rfloor$ and thus by monotonicity $\tfrac{1}{s} \mathcal{H}^1\big(B_{s}(x)\cap \mathcal{I}\big)\leq \lfloor a \rfloor$ for all $s<r$. Due to the upper bound $\tfrac{1}{r} \mathcal{H}^1(B_r(x)\cap \mathcal{I}) \leq 7$, this shows that for all $r>0$ sufficiently small we have $\tfrac{1}{r} \mathcal{H}^1(B_r(x)\cap \mathcal{I})=k \in \{0,1,\ldots,7\}$, with $k$ being independent of $r$. 

Furthermore, for $r>0$ small enough the intersections with the circle $\partial B_r(x)$ must be orthogonal at $\mathcal{H}^1$-a.\,e.\ point of $\mathcal{I}$: Otherwise, the identity~\eqref{RAverage} would imply $\frac{1}{r} \mathcal{H}^1\big(B_r\cap \mathcal{I}\big)<k$, a contradiction.

\emph{Step 3: Let $x\in D$. Then for all $r>0$ sufficiently small, $\mathcal{I}$ intersects $\partial B_r(x)$ in either zero, two, or three points.
Furthermore, as $r\downarrow 0$, in case of two points the angle between them must converge to $180^\circ$, while in case of three points the angle between two of them must converge to $120^\circ$.}

Note that the case of one intersection point is excluded for a.\,e.\ $r>0$ small enough by topological considerations.

Suppose now that we had $k\in \{4,5,6,7\}$ points, or that one of the angles deviates from $180^\circ$ respectively $120^\circ$ by more than $\delta>0$. 
In this case, we can construct $\tilde \chi_i$ by setting $\tilde \chi_i(y):=\chi_i(y)$ outside of $B_r(x)$ and by constructing $\tilde \chi_i(y)$ in $B_r(x)$ using basic Steiner tree constructions. 

More precisely, in case of two boundary points which are not opposing, we simply connect them by a straight line segment.
Due to
\begin{align*}
	2r = \mathcal{H}^1\big(B_r(x)\cap \mathcal{I}\big) = \frac{1}{2} \sum_{i=1}^P |\nabla \chi_i|\big(\overline{B_{r}(x)}\big)
\end{align*}
we thus get
\begin{align*}
\frac{1}{2} \sum_{i=1}^P |\nabla \tilde \chi_i|\big(\overline{B_{r}(x)}\big)
\leq 2r - c\delta^2 r 
=
\frac{1}{2} \sum_{i=1}^P |\nabla \chi_i|\big(\overline{B_{r}(x)}\big)- c\delta^2 r,
\end{align*}
where the gain $\delta^2$ is due to the fact that opposite points on the circle are furthest apart.
Additionally, we have $d_H(\tilde \chi,\chi)\leq 2r$, so that that the permitted size of $\delta$ must converge to zero as $r\downarrow 0$ by the stationarity condition \eqref{EqCondition}.

In case of three boundary points which are not approximately equispaced at $120^\circ$, we choose two with the smallest angle between them and use the Steiner tree construction in Lemma~\ref{lemma:steinertreeconstrunction} to connect those two boundary points to the center with a fork. Then we connect the third remaining point with a straight line to the center.
The computation is then similar to the above case of two boundary points.

In case of $k\geq 4$ boundary points, two of those points must be neighbors with an angle $\alpha\leq 90^\circ$. 
We apply again Lemma~\ref{lemma:steinertreeconstrunction} to connect those two points to the center.
Then, as before, we connect all other points to the center by straight line segments to obtain
\begin{align*}
&\frac{1}{2} \sum_{i=1}^P |\nabla \tilde \chi_i|\big(\overline{B_{r}(x)}\big)
\leq (k- 2)r + (2-c)r
=
\frac{1}{2} \sum_{i=1}^P |\nabla \chi_i|\big(\overline{B_{r}(x)}\big)- c r.
\end{align*}
Again, it holds that\footnote{Again, to guarantee this bound we need to modify our construction and place an arbitrarily small amount of the phases in $B_r(x)$ that we eliminate from $B_r(x)$ in the construction into some place in $B_r(x)$. As the size of this phase is arbitrarily small, it does not affect our estimates.} $d_H(\tilde \chi,\chi)\leq 2r$.
Therefore, the stationarity condition \eqref{EqCondition} implies that $k\geq 4$ 
is excluded for small enough $r$.

\emph{Step 4: Conclusion.} 
By \emph{Step 2} and \emph{Step 3} the network of interfaces $\mathcal{I}$ consists of finitely many straight line segments that can only meet in triple junctions with exactly $120^\circ$.
\end{proof}

\section{Appendix}

\subsection{Sets of finite perimeter}
We collect some results from the theory of sets of finite perimeter.
We start with the isoperimetric inequality relative to an open set.

\begin{theorem}[Relative isoperimetric inequality, cf.\ \cite{AmbrosioFuscoPallara}]
\label{TheoremRelIsoperimetricInequ}
Let $\Omega\subset\Rd$ be an open set with compact Lipschitz boundary $\partial\Omega$.
Then there exist constants $\gamma=\gamma(\Omega)>0$ and $C=C(\Omega)>0$ such that
for all Borel sets $G\subset\Omega$ with $\mathcal{L}^d(G)\leq\gamma$ it holds
\begin{align}
\label{RelIsoperimetricInequ}
\mathcal{L}^d(G) \leq C\mathrm{Per}(G;\Omega)^\frac{d}{d-1}.
\end{align} 
\end{theorem}

The next result we want to mention concerns the one-dimensional restrictions of sets of finite
perimeter.

\begin{theorem}[One-dimensional restrictions of sets of finite perimeter, cf.\ \cite{AmbrosioFuscoPallara}]
\label{OneDimRestrictions}
Consider a set $G$ of finite perimeter in $\Rd$, denote by 
$\vec{\nu}^{\,G}=(\nu^G_{x_1},\ldots,\nu^G_{x_{d-1}},\nu^G_{y})\in\Rd$
the associated measure theoretic inner unit normal vector field of 
the reduced boundary $\partial^*G$, and let $\chi^*_G$ be the
precise representative of the bounded variation function $\chi_G$. 
Then for Lebesgue almost every $x\in\Rd[d-1]$ the one-dimensional sections 
$G_{x}:=\{y\in\Rd[]\colon (x,y)\in G\}$ satisfy the following properties:
\begin{itemize}\itemsep3pt
	\item[i)] $G_{x}$ is a set of finite perimeter in $\Rd[]$, 
						$\chi_G(x,\cdot)=\chi_G^*(x,\cdot)$ Lebesgue almost everywhere in $G_x$,
	\item[ii)] $(\partial^*G)_{x}=\partial^*G_{x}$, 
	\item[iii)] $\nu_y^G(x,y_0)\neq 0$ for all $y_0\in\Rd[]$ such that $(x,y_0)\in\partial^*G$, and
	\item[iv)] $\lim_{y\downarrow y_0}\chi^*_G(x,y)=1$ and $\lim_{y\uparrow y_0}\chi^*_G(x,y)=0$
							if $\nu^G_y(x,y_0)>0$, and vice versa if $\nu^G_y(x,y_0)<0$.
\end{itemize}
\end{theorem}

Finally, we rely on the following structure and regularity result for the reduced boundary of planar sets 
of finite perimeter. To formulate the statement, a set $\mathcal{C}\subset\Rd[2]$ is called a rectifiable
Jordan curve if $\mathcal{C}$ is the image of a continuous and bijective map from the unit sphere $\mathbb{S}^1$
and $\mathcal{H}^1(\mathcal{C})<\infty$. It is a fact that any rectifiable Jordan curve admits a Lipschitz
reparametrization (cf.\ \cite{Ambrosio2001}). Moreover, by the Jordan 
curve theorem $\Rd[2]\setminus\mathcal{C}$ 
has exactly two connected components one of which is bounded. The bounded  component 
will be denoted by $\mathrm{int}(\mathcal{C})$.  

\begin{theorem}[Structure of planar sets of finite perimeter, cf.\ \cite{Ambrosio2001}]
\label{BdryLipschitzCurveDecomp}
Let $G\subset\Rd[2]$ be a set of finite perimeter with positive and bounded Lebesgue measure.
Then there exists a unique decomposition of the reduced boundary of $G$ into
rectifiable Jordan curves $(\mathcal{C}^\pm_m)_{m\in\mathbb{N}}$ with the following properties:
\begin{itemize}
\item[i)] $\mathrm{Per}(G)=\sum_{m}\mathcal{H}^1(\mathcal{C}^+_m) + \sum_{m}\mathcal{H}^1(\mathcal{C}^-_m)$.
\item[ii)] For each $m\neq m''$ the sets $\mathrm{int}(\mathcal{C}^+_{m})$ 
and $\mathrm{int}(\mathcal{C}^+_{m''})$
           are either disjoint or one is a subset of the other. 
					In the second case, there exists $m'\in\mathbb{N}$
					 such that $\mathrm{int}(\mathcal{C}^-_{m'})$ lies in between the two. 
					An analogous property holds for the family
					 of sets $(\mathrm{int}(\mathcal{C}^-_m))_{m\in\mathbb{N}}$. Moreover, 
					for each $m'\in\mathbb{N}$ there 
					 exists $m\in\mathbb{N}$ such that $\mathrm{int}(\mathcal{C}^-_{m'})$ 
					is a subset of $\mathrm{int}(\mathcal{C}^+_m)$. 
\item[iii)] Define $L_{m}:=\{m'\in\mathbb{N}\colon\mathrm{int}(\mathcal{C}^-_{m'})
\subset\mathrm{int}(\mathcal{C}^+_m)\}$
            as well as $G_m:=\mathrm{int}(\mathcal{C}^+_m)\setminus\bigcup_{m'\in L_m} \mathrm{int}(\mathcal{C}^-_{m'})$.
						The sets $G_m$ then form a family of pairwise disjoint, indecomposable 
						sets of finite perimeter such that
						$G=\bigcup_{m\in\mathbb{N}}G_m$.
\end{itemize}
\end{theorem}

\subsection{Steiner tree construction}
For the convenience of the reader, we state a simple but crucial tool in constructing competitors for the interface length functional via Steiner trees.

\begin{lemma}[{Length of Steiner tree, cf.~\cite[Lemma 5.20]{DavidLeger}}]\label{lemma:steinertreeconstrunction}
	For an arc $I \subset \partial B_1(0)$ of length $\alpha = \mathcal{H}^1(I) \leq \frac23 \pi$, let $Y(I)=[0,\xi]\cup[\xi,a]\cup[\xi,b]$ denote the fork connecting the center $0$ of the ball to the two endpoints $a$ and $b$ of the arc $I$, in which the branch point $\xi$ is chosen such that the three segments form equal angles of $120^\circ$. 
	Then
	\begin{align}
	\mathcal{H}^1(Y(I))=f(\alpha),
	\end{align}
	where $f(\alpha)=2\sin(\tfrac\alpha2+\tfrac\pi6)$ for $0<\alpha<\frac23 \pi$.
\end{lemma}

\section*{Acknowledgments}
This project has received funding from the European Union's Horizon 2020 research and 
innovation programme under the Marie Sk\l{}odowska-Curie Grant Agreement No.\ 665385 
\begin{tabular}{@{}c@{}}\includegraphics[width=3ex]{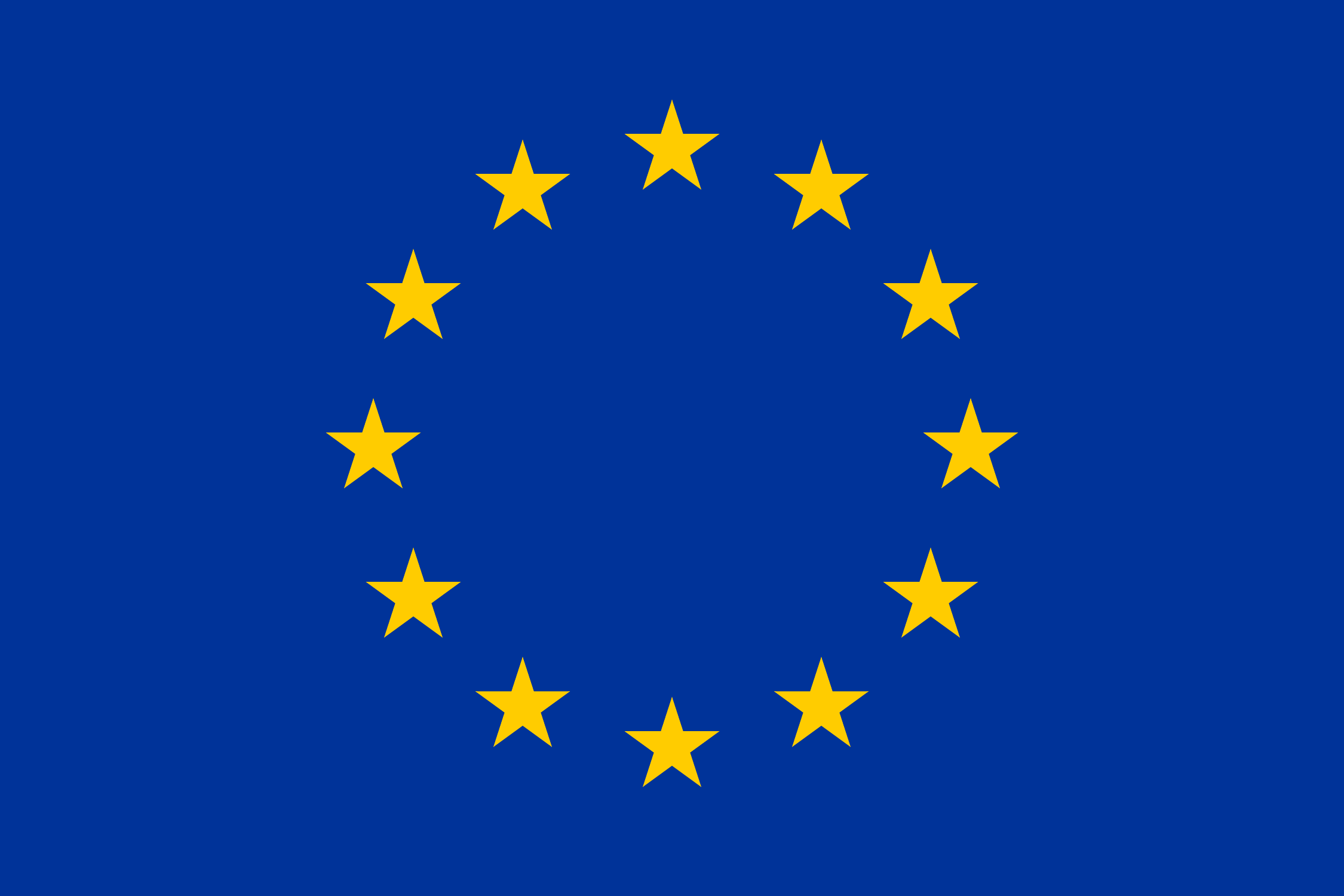}\end{tabular}, from the
European Research Council (ERC) under the European Union's Horizon 2020
research and innovation programme (grant agreement No 948819)
\smash{
\begin{tabular}{@{}c@{}}\includegraphics[width=6ex]{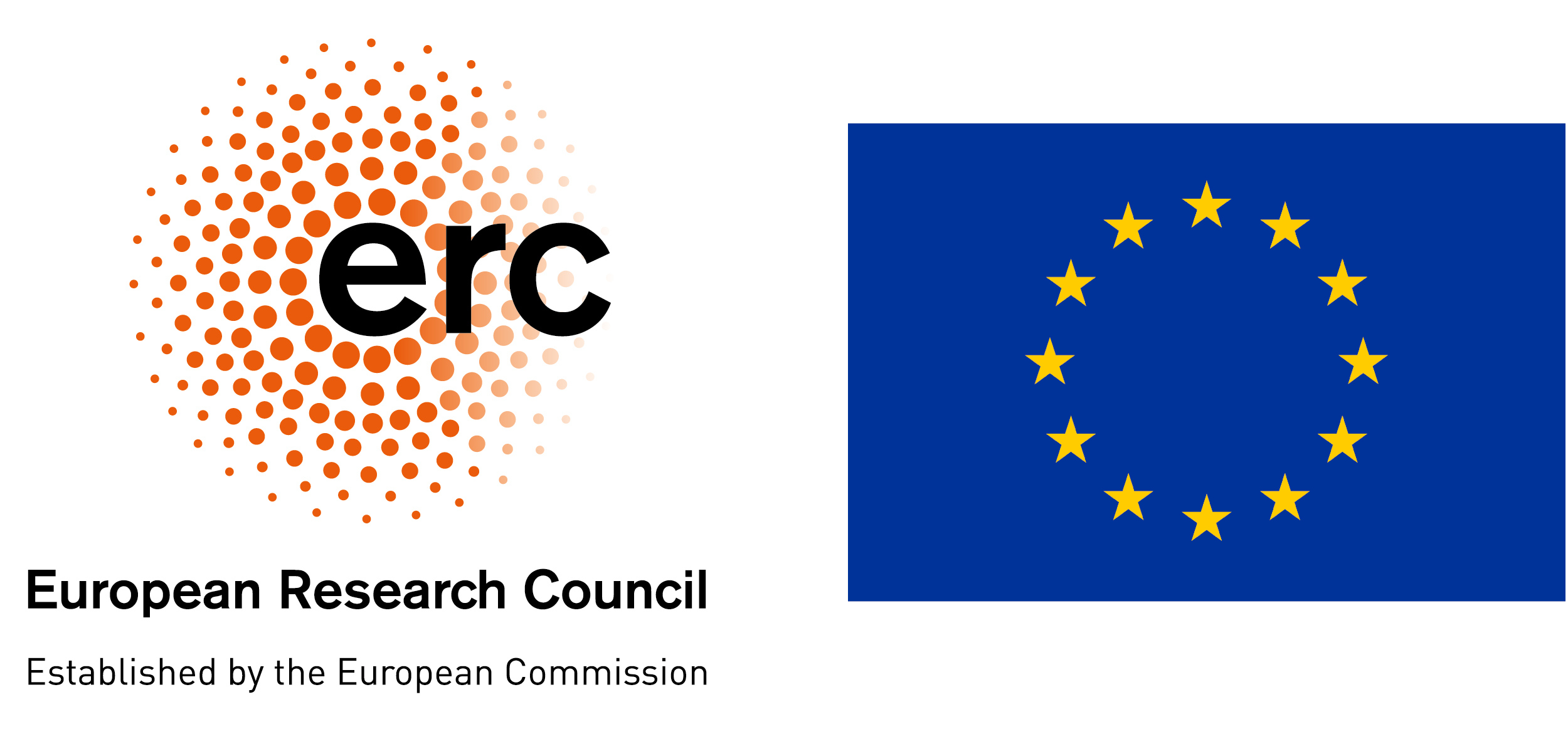}\end{tabular}
},
and from the Deutsche Forschungsgemeinschaft (DFG, German Research Foundation) 
under Germany's Excellence Strategy -- EXC-2047/1 -- 390685813 and EXC 2044 --390685587, Mathematics M{\"u}nster: Dynamics--Geometry--Structure.

\bibliographystyle{abbrv}
\bibliography{locminarea}

\end{document}